\documentclass[11pt,a4paper,reqno]{amsart}
\pdfoutput=1
\usepackage{amsfonts,amsmath,amscd,amssymb,amsbsy,amsthm,amstext,amsopn}
\usepackage{fullpage,mathrsfs,subfigure}
\usepackage{graphicx}
\usepackage{array}
\usepackage{multirow}
\usepackage[all]{xy}
\input xy
\xyoption{all}
\usepackage[colorlinks=true, pdfpagemode=none, pdfmenubar=false, pdfstartview=FitH, linkcolor=blue, citecolor=blue, urlcolor=blue, pdffitwindow=false]{hyperref}

\numberwithin{equation}{section}
\newtheorem{theorem}{Theorem}[section]
\newtheorem*{theorem*}{Theorem}
\newtheorem{proposition}[theorem]{Proposition}
\newtheorem{lemma}[theorem]{Lemma}
\newtheorem{corollary}[theorem]{Corollary}
\theoremstyle{definition}
\newtheorem{defn}{Definition}[section]
\newtheorem*{defn*}{Definition}
\newtheorem{example}[theorem]{Example}
\theoremstyle{remark}
\newtheorem{remark}[theorem]{Remark}

\newcommand{\CC}{\ensuremath{\mathbb{C}}}

\newcommand{\ghilb}{\ensuremath{G}\operatorname{-Hilb}}

\newcommand{\SL}{\operatorname{SL}} 

\newcommand{\one}{\ensuremath{(\mathrm{i})}}
\newcommand{\two}{\ensuremath{(\mathrm{ii})}}
\newcommand{\three}{\ensuremath{(\mathrm{iii})}}

\title{Derived Reid's recipe \\for abelian subgroups of $\gsl_3(\mathbb{C})$}

 \author{Sabin Cautis, Alastair Craw and Timothy Logvinenko} 
 
 \address{Department of Mathematics \\ University of British Columbia \\ Vancouver, Canada}
 \email{cautis@math.ubc.ca}

\address{Department of Mathematical Sciences\\ University of Bath\\ Bath BA2 7AY, UK}
 \email{A.Craw@bath.ac.uk} 

 \address{School of Mathematics \\ Cardiff University \\ Senghennydd
Road \\  Cardiff, CF24 4AG \\  UK}
 \email{LogvinenkoT@cardiff.ac.uk}

\DeclareMathOperator{\krn}{Ker}

\DeclareMathOperator{\img}{Im}

\DeclareMathOperator{\homm}{Hom}
\DeclareMathOperator{\shhomm}{{\it\mathcal{H}om\rm}}

\DeclareMathOperator{\gl}{GL}
\DeclareMathOperator{\gsl}{SL}

\DeclareMathOperator{\spec}{Spec}

\DeclareMathOperator{\hilb}{Hilb}

\DeclareMathOperator{\irr}{Irr}

\DeclareMathOperator{\supp}{Supp}

\DeclareMathOperator{\cohcat}{Coh}

\DeclareMathOperator{\rder}{\bf R}

\DeclareMathOperator{\id}{Id}
\DeclareMathOperator{\hex}{Hex}
\DeclareMathOperator{\Except}{Exc}
\DeclareMathOperator{\sinksource}{SS}
\DeclareMathOperator{\zerofibre}{ZF}
\DeclareMathOperator{\givrep}{V}
\DeclareMathOperator{\regrep}{V_{\text{reg}}}
\DeclareMathOperator{\regring}{R}
\DeclareMathOperator{\convhull}{Conv}

\def\twalg{{\regring \rtimes G}}
\def\mckquiv{\mbox{Q}(G)}
\def\gnat{gnat}
\def\boldomega{{\boldsymbol\omega}}

\def\L{{\mathcal{L}}}
\def\O{{\mathcal{O}}}
\def\gcd{\mbox{gcd}}
\def\lcm{\mbox{lcm}}

\def\tD{{\tilde{D}}}

\begin{document}

\begin{abstract}
For any finite subgroup $G \subset \gsl_3(\mathbb{C})$, 
work of Bridgeland-King-Reid constructs an equivalence
between the $G$-equivariant derived category of $\CC^3$ 
and the derived category of the crepant resolution
$Y = \ghilb \CC^3$ of $\CC^3/G$. When $G$ is abelian 
we show that this equivalence gives a natural correspondence 
between irreducible representations of $G$ and certain sheaves
on exceptional subvarieties of $Y$, thereby extending the McKay 
correspondence from two to three dimensions. This categorifies Reid's 
recipe and extends earlier work from \cite{CautisLogvinenko} and 
\cite{Logvinenko-ReidsRecipeAndDerivedCategories} which dealt only 
with the case when $\CC^3/G$ has one isolated singularity. 
\end{abstract}
    
\maketitle

\tableofcontents

\section{Introduction}

Originating in observations by John McKay
\cite{McKay-GraphsSingularitiesAndFiniteGroups}, 
the classical McKay correspondence for a finite subgroup $G \subset \SL_2(\CC)$ 
is a bijection between the nontrivial irreducible representations of $G$ and 
the irreducible exceptional divisors on the minimal 
resolution $Y$ of $\CC^2/G$. 

The representation ring of $G$ is naturally isomorphic to
$K^{G}(\CC^2)$, the Grothendieck group of $G$-equivariant coherent
sheaves on $\CC^2$. Subsequently, the McKay correspondence was
realized geometrically in 
\cite{GsV-ConstructionGeometriqueDeLaCorrespondanceDeMcKay} 
as an isomorphism $K^G(\CC^2) \xrightarrow{\sim} K(Y)$. 
In \cite{KapranovVasserot-KleinianSingularitiesDerivedCategoriesAndHallAlgebras} 
this isomorphism was lifted to an equivalence $D^G(\CC^2)
\xrightarrow{\sim} D(Y)$ of derived categories of coherent sheaves.
For each non-trivial irreducible $G$-representation $\rho$
it sends the sheaf $\O_0 \otimes \rho$ to the structure sheaf 
of the corresponding exceptional divisor (twisted by $\mathcal{O}(-1)$).

In dimension three, for a finite subgroup $G \subset
\SL_3(\CC)$, the $G$-Hilbert scheme $Y = \ghilb \CC^3$ is a
crepant resolution of $\CC^3/G$. This is a consequence of 
the derived equivalence $\Psi\colon D^G(\CC^3) \xrightarrow{\sim} D(Y)$ 
constructed by Bridgeland--King--Reid~\cite{BKR01}. 
Such an equivalence was conjectured by Reid~\cite{Kinosaki-97} 
while building on Nakamura's description of 
$\ghilb(\CC^3)$\cite{Nak00}. For
$G$ abelian, Reid also defined in \emph{loc.cit.} a basis of
$H^*(Y,\mathbb{Z})$ using an ad-hoc combinatorial construction that 
was dubbed ``Reid's recipe'' by
Craw~\cite{Craw-AnexplicitconstructionoftheMcKaycorrespondenceforAHilbC3}.

In \cite{CautisLogvinenko} the first and third authors
conjectured that $\Psi\colon D^G(\CC^3) \xrightarrow{\sim} D(Y)$ 
sends $\O_0 \otimes \rho$ to a sheaf supported on 
the exceptional subvariety of $Y$ which is either a single divisor, 
a single curve or a chain of divisors, depending on the role of 
$\rho$ in Reid's recipe. The conjectured correspondence between 
irreducible representations of $G$ and sheaves on exceptional 
subvarieties of $Y$ generalises naturally 
the classical McKay correspondence for $\gsl_2(\mathbb{C})$ and 
was subsequently called the \em derived Reid's recipe\rm.  

The fact that the image of each $\O_0 \otimes \rho$ is a single sheaf 
was proved in \cite{CautisLogvinenko} when $G$ is abelian and 
$\mathbb{C}^3/G$ has a single isolated singularity. 
Under these same conditions, the third author proved 
in \cite{Logvinenko-ReidsRecipeAndDerivedCategories} that 
the support of this sheaf is indeed determined by 
the role of $\rho$ in Reid's recipe.  

When the singularities of $\mathbb{C}^3/G$ are not isolated, the
geometry of the exceptional locus of $Y$ is more complicated and
the methods of \cite{CautisLogvinenko} and 
\cite{Logvinenko-ReidsRecipeAndDerivedCategories} do not apply. 
Moreover, these methods only compute the support of 
$\Psi(\O_0 \otimes \rho)$ rather than the sheaf itself. 
In dimension two this extra data encoded very little, being 
just the line bundle $\mathcal{O}(-1)$ for each exceptional $\mathbb{P}^1$.  
In dimension three, the situation is more subtle and 
the extra data provided by the sheaf forms a meaningful part 
of the correspondence. 

In this paper, we compute the whole of derived Reid's recipe 
for any finite abelian subgroup of $\gsl_3(\mathbb{C})$. We
can do this due to a new approach via \em CT-subdivisions \rm
(see Section \ref{section-CT-subdivisions}). 
This technique may make it possible to generalise 
the derived Reid's recipe to dimer models by extending the work 
of 
\cite{BocklandtCrawQuinteroVelez-GeometricReidsRecipeForDimerModels}.

\subsection{Summary of results}
Let $G \subset \SL_3(\CC)$ be a
finite subgroup. The \emph{$G$-Hilbert scheme} $Y=\ghilb \CC^3$ is
the fine moduli space parametrizing subschemes $Z\subset \CC^3$ for which
$H^0(\mathcal{O}_Z)$ is isomorphic to $\CC[G]$ as a $\CC[G]$-module.
Let $\mathcal{Z}\subset Y\times \CC^3$ denote the universal subscheme.
As a Fourier-Mukai kernel, $\mathcal{O}_{\mathcal{Z}}$ induces 
a functor $\Psi\colon D^G(\CC^3) \to D(Y)$ between the bounded 
derived categories of coherent sheaves on $Y$ and $G$-equivariant 
coherent sheaves on $\CC^3$. Bridgeland--King--Reid~\cite{BKR01} showed that
$\Psi$ is an exact equivalence of triangulated categories and 
that the Hilbert--Chow morphism 
$\pi\colon Y\to \CC^3/G$ is a projective, crepant resolution.
  
An irreducible representation $\rho$ of $G$ defines two natural
$G$-equivariant sheaves on $\CC^3$, namely $\O_{\CC^3} \otimes \rho$
and $\O_0 \otimes \rho$, where $\O_0$ is the structure sheaf of 
the origin in $\CC^3$. 

The image $\Psi(\mathcal{O}_{\CC^3} \otimes \rho)$ is isomorphic
to $\mathcal{L}^{\vee}_\rho$, where $\mathcal{L}_\rho$ 
is one of the \emph{tautological vector bundles}. 
These are vector bundles on $Y$ defined via 
$\pi_{Y * } \mathcal{O}_{\mathcal{Z}} = 
\bigoplus \mathcal{L}_\rho \otimes \rho$ 
where $\pi_Y\colon Y\times \CC^3\to Y$ is the natural projection
and the decomposition is with respect to the 
trivial $G$-action on $Y$. 

On the other hand, $\Psi(\mathcal{O}_{0} \otimes \rho)$ is more 
complicated to describe. We do this for abelian $G$, 
so the irreducible representations of $G$ are 
the characters $\chi \in G^\vee$. A priori, 
each $\Psi(\mathcal{O}_{0} \otimes \chi)$ is an abstract complex
in $D(Y)$, but our first main result is:

\begin{theorem}
\label{theorem-transform-is-a-sheaf-for-non-trivial-chi}
Let $G \subset \gsl_3(\mathbb{C})$ be a finite abelian subgroup
and let $\chi \in G^\vee$ be nontrivial. Then 
$\Psi(\mathcal{O}_0 \otimes \chi)$ is the pushforward of a (shift of a) 
coherent sheaf $\mathcal{F}_\chi$ from the exceptional subvariety 
$Z_\chi = \supp \Psi(\mathcal{O}_0 \otimes \chi)$ of $Y$. 
\end{theorem}

Our second result describes explicitly $\Psi(\mathcal{O}_0 \otimes \chi)$ 
in terms of Reid's recipe. The variety $Y$ is a toric variety
and its toric fan defines a triangulation $\Sigma$ of 
the junior simplex $\Delta$, cf. Section
\ref{section-ghilb-and-toric}. Reid's recipe marks each interior 
vertex and each interior edge of $\Sigma$ with one or two non-trivial 
characters of $G$, cf. Section \ref{subsection-reids-recipe}.  

Craw~\cite{Craw-AnexplicitconstructionoftheMcKaycorrespondenceforAHilbC3}
proves that every non-trivial $\chi$ marks in $\Sigma$ either 
a unique vertex, a unique edge, a single chain of edges or 
three chains of edges meeting in a vertex. We give a worked example
illustrating the results of this paper for $G = \frac{1}{15}(1,5,9)$ in Section
$\ref{section-worked-examples}$. In particular, there is a picture
of Reid's recipe in Fig.\ \ref{figure-34b}. 
There are further examples  
in \cite[\S 6]{Kinosaki-97}, \cite[\S 3]{Craw-AnexplicitconstructionoftheMcKaycorrespondenceforAHilbC3} and \cite[\S 1]{CautisLogvinenko}. 

The vertices and edges of $\Sigma$ correspond to the exceptional toric 
divisors and curves of $Y$. A chain of edges of $\Sigma$ corresponds 
to a chain of divisors on $Y$: the edges correspond to the curves in 
which the divisors intersect.  We say that $\chi$ marks 
a divisor, a curve or a chain of divisors, if it marks the corresponding 
vertex, edge or chain of edges. 

Finally, to deal with the trivial character $\chi_0$ which did not 
appear in Reid's recipe, we need to consider the fibre $\zerofibre$ 
of $Y$ over $0\in \CC^3/G$. In general, $\zerofibre$ is not
equidimensional and we split it up into $\zerofibre_1$ and 
$\zerofibre_2$, scheme-theoretic unions of its $1$- and $2$-dimensional 
components. 

\begin{theorem}[Derived Reid's recipe]
\label{theorem-derived-reids-recipe}
Let $G \subset \gsl_3(\mathbb{C})$ be a finite abelian subgroup
and let $\chi \in G^\vee$. Then $\mathcal{H}^i(\Psi(\mathcal{O}_0
\otimes \chi)) = 0$ unless $i = 0,-1,-2$. Moreover, one of the following
holds:
\begin{footnotesize} 
\begin{align*}
\begin{array}{|c|c|c|c|} 
\hline 
\text{Reid's recipe} 
&
\mathcal{H}^{-2}
&
\mathcal{H}^{-1}
&
\mathcal{H}^{0}
\\
\hline
\chi \text{ marks a single divisor } E 
&
0
&
0
&
\mathcal{L}^{-1}_\chi \otimes \mathcal{O}_E 
\\
\hline
\chi \text{ marks a single curve } C 
&
0
&
0
&
\mathcal{L}^{-1}_\chi \otimes \mathcal{O}_{C}
\\
\hline
\chi \text{ marks a chain of divisors } 
& &
\mathcal{L}^{-1}_\chi(- E - F) \otimes \mathcal{O}_Z, &
\\
\text{ which starts at a divisor } E 
&
0
& 
\text{ where $Z$ is the reduced union of }
&
0
\\
\text{ and terminates at a divisor } F
& 
& 
\text{ the internal divisors of the chain } 
&
\\
\hline
\chi \text{ marks three chains of divisors,} 
& 
&
&
\\
\text{which start at $E_x$, $E_y$ and $E_z$ } &
0
& 
\mathcal{L}^{-1}_\chi(- E_x - E_y -E_z) \otimes \mathcal{V}_Z^\ddagger
&
0
\\
\text{and meet at a divisor } P
& 
& 
&
\\
\hline
\chi \text{ marks nothing ($\chi = \chi_0$)}
&
\omega_{\zerofibre_2} ~^\dagger
& 
\omega_{\zerofibre_1}(\zerofibre_2) ^\dagger 
&
0
\\
\hline
\end{array}
\end{align*}

$\dagger$: Here $\omega_{\zerofibre_i}$ is the dualizing sheaf
of $\zerofibre_i$. In fact, $\Psi(\mathcal{O}_0 \otimes \chi_0)$ 
is the dualizing complex $\boldomega_{\zerofibre}$ of $\zerofibre$, 
cf. \S \ref{section-case-of-chi=chi_0}.

$\ddagger$: The support of $\mathcal{V}_Z$ is the reduced union $Z$
of the internal divisors of the chains. Away from where the three chains 
meet $\mathcal{V}_Z \simeq \mathcal{O}_Z$, but on that locus it 
is locally free of rank $2$. To be precise,
$\mathcal{V}_Z$ is the cokernel of $\mathcal{O}_{Y}
\xrightarrow{E_{x} \oplus E_{y} \oplus E_{z}}
\mathcal{O}_{Z_{x}}(E_{x}) 
\oplus
\mathcal{O}_{Z_{y}}(E_{y}) 
\oplus
\mathcal{O}_{Z_{z}}(E_{z})$
where $Z_x = E_y P \cup E_z P$ and similarly for $Z_y$ and $Z_z$, 
cf. \S\ref{section-case-of-a-meeting-of-champions}
\end{footnotesize}
\end{theorem}

Thus, roughly, the \em derived Reid's recipe \rm 
assigns to each non-trivial $\chi$ one of the following: 
\begin{itemize}
\item An exceptional divisor or an exceptional curve.
\item A chain of exceptional divisors with two marked curves on the two endpoint divisors. 
\item A tree of exceptional divisors with one fork and three branches, together with three marked curves on the three endpoint divisors.
\end{itemize}

We close with two remarks. First, the sheaf data 
is necessary. For instance, when $G = \frac{1}{6}(1,2,3)$
there exist two representations which correspond to the same 
divisor but with different marked curves.  Secondly, the representations 
that correspond to a single divisor or a curve are called 
\em essential\rm. They possess an independent characterisation in 
terms of moduli of quiver representations 
\cite{Takahashi-OnEssentialRepresentationsInSL3CMcKayCorrespondence}. 

\medskip
\noindent\textbf{Acknowledgements.}  
S.C. is supported by NSF grant DMS-1101439 and the Alfred P. Sloan
foundation, A.C. is supported by EPSRC grant EP/G004048/1 and T.L. 
is supported by EPSRC grant EP/H023267/1. We would like to thank Miles
Reid, Michael Wemyss, Yukari Ito, Keisuke Takahashi and Yuhi Sekiya
for useful discussions while writing this paper. 

\newpage
\section{Preliminaries}

\subsection{Action of $G$ on $\mathbb{C}^3$}
\label{section-prelims-action-of-G-on-C3}

Let $G$ be a finite subgroup of $\gsl_3(\mathbb{C})$. It 
acts naturally on $\mathbb{C}^3$ on the left. Let $\givrep$ 
be the corresponding representation. We identify 
$\givrep^\vee$ with the space of linear functions 
on $\mathbb{C}^3$.  Let $x_1$, $x_2$ and $x_3$ be the basis 
of $\givrep^\vee$ dual to the canonical basis of $\mathbb{C}^3$. 
Sometimes, we use $x$, $y$ and $z$ for $x_1$, $x_2$ and $x_3$. 
The (left) action of $G$ on $\givrep$ defines a \em right \rm action 
of $G$ on $\givrep^\vee$, which we make into a left action by
setting 
\begin{align}
\label{eqn-G-action-on-givrep-vee}
g\cdot m(v) = m(g^{-1} \cdot v) \quad \text{ for all }
   g \in G,\; m \in \givrep^\vee \text{ and } v \in \givrep.
\end{align}
This action extends naturally to an action on 
the symmetric algebra $R = S(\givrep^\vee) = \mathbb{C}[x, y, z]$. 
We give $\mathbb{C}^3$ the structure of a $G$-scheme
by identifying it with the closed points of $\spec R$. 

Denote by $G^{\vee}$ the character group 
$\homm_\mathbb{Z}(G,\mathbb{C}^*)$ of $G$.
A rational function $f \in K(\mathbb{C}^{3})$ is said to be 
a \em semi-invariant of $G$ \rm, or a \em $G$-eigenvector, \rm 
if there is $\chi \in G^\vee$ such that $g \cdot f = \chi(g) f$. 
We denote the weight $\chi$ of $f$ by $\kappa(f)$. 
Suppose $G$ is abelian, then $\mathbb{C}^3$ has a basis
of $G$-eigenvectors. Replacing $G$ by a conjugate subgroup 
of $\gsl_3(\mathbb{C})$ if necessary, we assume
that the canonical basis of $\mathbb{C}^3$ is one such, 
i.e. every Laurent monomial is a $G$-eigenvector. 

\subsection{$G$-$\hilb \mathbb{C}^3$ and its toric geometry} 
\label{section-ghilb-and-toric}

A \em $G$-cluster \rm is a $G$-invariant subscheme of $\mathbb{C}^3$
whose global sections are $G$-equivariantly isomorphic 
to the regular representation $\regrep$. 
The support of any $G$-cluster is a set-theoretic $G$-orbit, 
and we think of $G$-clusters as scheme-theoretic orbits of 
$G$ in $\mathbb{C}^3$. By \cite{BKR01} the fine moduli space
$G$-$\hilb \mathbb{C}^3$ of $G$-clusters together with its 
\em Hilbert-Chow map \rm 
$G\text{-}\hilb \mathbb{C}^3 \rightarrow \mathbb{C}^3/G$, 
which sends each $G$-cluster to the corresponding $G$-orbit, 
is a crepant resolution. For abelian $G$ this resolution 
is well understood in terms of toric geometry. Below
we summarize the main points, 
for more detail cf. \cite{Nak00}, \cite{CrawReid} 
or
\cite[\S3.1]{Logvinenko-Families-of-G-constellations-over-resolutions-of-quotient-singularities}. 

Let $\mathbb{Z}^3$ be the lattice of Laurent monomials, 
where we identify $m = (m_1, m_2, m_3)$ with the monomial 
$x_1^{m_1} x_2^{m_2} x_3^{m_3}$. Let $M \subset \mathbb{Z}^3$ 
be the sublattice of $G$-invariant monomials. 
The dual lattice of $G$-weights $L = \homm_\mathbb{Z}(M,\mathbb{Z})$
contains $\homm_\mathbb{Z}(\mathbb{Z}^3, \mathbb{Z})$ as a sublattice.
As $G$ is finite we have $L \subset \mathbb{Q}^3$ and 
write each $l \in L$ as a triplet $(l_1, l_2, l_3) \in \mathbb{Q}^3$. 
Let $\sigma_+$ denote the cone $\{(e_1, e_2, e_3)  \;|\; e_i \geq 0\}$ 
in $\mathbb{R}^3 \simeq L \otimes \mathbb{R}$. 
Then:
\begin{itemize}
\item $\mathbb{C}^3$ is defined as a toric variety by
the lattice $\mathbb{Z}^3$ and the cone $\sigma_+$.
\item $\mathbb{C}^3/G$ is defined by the lattice $L$ and the cone $\sigma_+$. 
\item A toric resolution of $\mathbb{C}^3/G$ is defined 
by the lattice $L$ and a fan that subdivides the cone $\sigma_+$ 
into regular subcones. 
\end{itemize}

Let the \em junior simplex \rm $\Delta$ be the triangle 
formed by cutting $\sigma_+$ by the hyperplane $\sum e_i = 1$. 
A fan $\mathfrak{F}$ which subdivides $\sigma_+$ 
into regular subcones is determined by how it 
triangulates $\Delta$. Let $\mathfrak{E}$ be 
the set of the generators of the rays of $\mathfrak{F}$. 
For any crepant resolution,  
$\mathfrak{E} = L \cap \Delta$ and moreover:
\begin{itemize}
\item The vertices of $\Sigma$ are the elements of $\mathfrak{E}$. 
For any vertex $e$ of $\Sigma$ we write $E_{e}$ for 
the toric divisor which corresponds to the ray $\left<e\right>$. 
\item The edges of $\Sigma$ correspond to the two-dimensional 
cones of $\mathfrak{F}$. Two vertices $e,f \in \mathfrak{E}$ are joined 
by an edge if and only if divisors $E_e$ and $E_f$ intersect. 
The intersection is then the torus-invariant curve that corresponds to
the cone $\langle e, f \rangle$, and this curve is isomorphic to
$\mathbb{P}^1$ or $\mathbb{A}^1$ according to whether or not 
the edge $\left(e, f\right)$ lies in the interior of $\Delta$.
\item The triangles of $\Sigma$ correspond to three-dimensional 
cones of $\mathfrak{F}$. A triangle $(e,f,g)$ corresponds 
to $E_e$, $E_f$ and $E_g$ intersecting at the torus fixed point 
$\left<e,f,g\right>$.
\item Each three-dimensional cone $\sigma$ in $\mathfrak{F}$ 
defines a toric affine chart $A_\sigma$ isomorphic to
$\mathbb{C}^3$. In terms of toric geometry, $A_\sigma$ is 
the union of the torus orbits defined by subcones of $\sigma$. 
\end{itemize} 

\cite{CrawReid} describes how to construct the triangulation $\Sigma$ 
of $\Delta$ corresponding to $Y = G$-$\hilb \mathbb{C}^3$. 
The toric divisors of $Y$ defined by vertices of $\Sigma$
divide into three groups:
\begin{itemize}
\item The corner vertices of $\Delta$ give the strict 
transforms of coordinate hyperplanes of $\mathbb{C}^3/G$. 
\item The vertices on the sides of $\Delta$ give 
the non-compact exceptional divisors. 
Each of these is a $\mathbb{P}^1$-fibration over one of 
the coordinate lines of $\mathbb{C}^3/G$, possibly degenerating 
over the origin to a pair of $\mathbb{P}^1$s intersecting transversally.
\item The vertices in the interior of $\Delta$
give the compact exceptional divisors. These lie over the origin 
of $\mathbb{C}^3/G$ and are either a $\mathbb{P}^2$, a rational scroll 
blown up in $0$, $1$ or $2$ points, or a del Pezzo surface of degree
six \cite[Cor.\ 1.4]{CrawReid}.  
\end{itemize}

Let $\zerofibre$ denote the fiber of $Y$ over the origin 
of $\mathbb{C}^3/G$. In general, $\zerofibre$ is just a section of the
exceptional set $\Except(Y)$. It is a reducible variety 
with a two-dimensional and a one-dimensional stratum 
which we denote by $\zerofibre_2$ and $\zerofibre_1$. 
The irreducible components of $\zerofibre_2$ are 
the compact exceptional divisors of $Y$
defined by the internal vertices of $\Delta$.  
The non-compact exceptional divisors of $Y$ 
each meet $\zerofibre$ in a $\mathbb{P}^1$ or a pair 
of $\mathbb{P}^1$s. The irreducible components of $\zerofibre_1$ 
are those of these $\mathbb{P}^1$s which 
do not lie on any of the divisors in $\zerofibre_{2}$. 
These are the curves defined by the edges of 
$\Sigma$ which cross the interior of $\Delta$ but 
whose endpoints lie on the sides of $\Delta$. 

\subsection{Reid's recipe}
\label{subsection-reids-recipe}

Reid's recipe \cite{Kinosaki-97},
\cite{Craw-AnexplicitconstructionoftheMcKaycorrespondenceforAHilbC3}
is an algorithm to construct the cohomological version of 
the McKay correspondence for abelian $G \subset \gsl_3(\mathbb{C})$.
It begins with a toric geometry computation which marks 
internal edges and vertices of the $G$-$\hilb$ triangulation $\Sigma$ 
with characters of $G$. This marking is a key to stating our main 
result, Theorem \ref{theorem-derived-reids-recipe}, so below we give 
a brief summary of its construction. See Section 
\ref{section-example-the-group-and-the-resolution} and 
Figure \ref{figure-34} for a worked example. 

First we mark each edge $(e,f)$ in $\Sigma$ 
according to the following rule. The one-dimensional 
ray in $M$ perpendicular to the hyperplane $\left< e,f\right>$ in $L$ has two
primitive generators: $\frac{m_1}{m_2}$ and $\frac{m_2}{m_1}$, where
$m_1, m_2$ are co-prime regular monomials in $\regring$. 
As $M$ is the lattice of $G$-invariant Laurent monomials, $m_1$ and $m_2$ have
to be of the same character $\chi$ for some $\chi \in G^\vee$. We say
that $(e,f)$ is \it carved out \rm by the ratio $m_1 : m_2$
(or $m_2 : m_1$) and \it mark it \rm by $\chi$.  

Then we mark each internal vertex $e$ of $\Sigma$ according to whether
the corresponding exceptional divisor is a $\mathbb{P}^2$, 
a rational scroll blown-up in $0$, $1$ or $2$ points or a del Pezzo 
surface $dP_6$:  
\begin{enumerate}
\item If $E_e$ is a $\mathbb{P}^2$, then there are three
edges incident to $e$ in $\Sigma$. By
\cite[Lemma 3.1]{Craw-AnexplicitconstructionoftheMcKaycorrespondenceforAHilbC3} 
they lie on three lines joining $e$ to the three corner vertices 
of $\Delta$ and are marked with same character $\chi \in G^\vee$. 
Reid's recipe prescribes for $E_e$ to be marked with 
character $\chi \cdot \chi$.  

\item If $E_e$ is a rational scroll blown-up in $0$, $1$ or $2$
points then there are $4$, $5$ or $6$ edges incident to $e$ in
$\Sigma$, respectively. By 
\cite[Lemma 3.2-3.3]{Craw-AnexplicitconstructionoftheMcKaycorrespondenceforAHilbC3} 
two of these lie on a straight line joining $e$ to a corner vertex 
of $\Delta$ and are marked by same character $\chi \in G^\vee$. 
Of the remaining ones exactly two are marked by the same character.
Denote this character by $\chi' \in G^\vee$, then Reid's recipe 
prescribes for $E_e$ to be marked with character $\chi \cdot \chi'$. 

\item  If $E_e$ is a del Pezzo surface $dP_6$, then there are $6$ 
edges incident to $e$ in $\Sigma$. These lie on three straight lines
which intersect at $e$. By 
\cite[Lemma 3.4]{Craw-AnexplicitconstructionoftheMcKaycorrespondenceforAHilbC3} 
the two toric projections $E_e \rightarrow \mathbb{P}^2$ are given 
by monomial maps $(m_0 : m_1 : m_2)$ and $(m'_0 : m'_1 : m'_2)$
with $m_i \in R$ being all of same character $\chi$
and $m'_i \in R$ being all of same character $\chi'$. 
Reid's recipe prescribes $E_e$ to be marked with two characters:
$\chi$ and $\chi'$. 
\end{enumerate}
 
\subsection{Universal family $\mathcal{M}$ and its dual $\widetilde{\mathcal{M}}$}
\label{section-universal-family-and-its-dual}
The fine moduli space $Y = G$-$\hilb \mathbb{C}^3$ 
comes equipped with the universal family 
$Z \subset Y \times \mathbb{C}^3$ of $G$-clusters. 
Let $G$ act trivially on $Y$, then we can speak 
of $G$-equivariant sheaves on $Y \times \mathbb{C}^3$. Let 
$\mathcal{M} \in \cohcat^G ( Y \times \mathbb{C}^n)$ denote
the structure sheaf of $Z$. It is a coherent $G$-sheaf flat
over $Y$ whose fiber over any point of $Y$ is a finite-length
$G$-sheaf on $\mathbb{C}^3$. Therefore 
$\pi_{Y *} \mathcal{M} \in \cohcat^G(Y)$ is a locally free 
$G$-sheaf of finite rank on $Y$. Since $G$ acts
trivially on $Y$, we can decompose $\pi_{Y *} \mathcal{M}$ 
into $G$-eigensheaves 
as $\bigoplus_{\rho \in \irr G} \mathcal{L}_\rho \otimes \rho$ 
where $\mathcal{L}_\rho$ are locally-free sheaves
of rank $\dim(\rho)$. In the literature $\mathcal{L}_\rho$ are 
known as \em tautological \rm or \em Gonzalez-Sprinberg 
and Verdier (GSp-V) \rm sheaves. 

The equivalence $\Phi\colon D(Y) \rightarrow D^G(\mathbb{C}^3)$ 
of \cite{BKR01} is a Fourier-Mukai transform whose Fourier-Mukai
kernel is $\mathcal{M}$. To compute derived Reid's recipe we 
need its inverse $\Psi\colon D^G(\mathbb{C}^3) \rightarrow D(Y)$.
\cite{CautisLogvinenko} showed that $\Psi$ is also a
Fourier-Mukai transform and that its Fourier-Mukai kernel 
$\widetilde{\mathcal{M}}$ is also a flat family of finite-length
sheaves on $\mathbb{C}^3$ which can be understood as follows.

A \em $G$-constellation \rm is a coherent $G$-sheaf on $\mathbb{C}^3$
whose space of global sections is isomorphic to $\regrep$
as a representation of $G$. We view $G$-clusters as a subclass 
of $G$-constellations by looking at their structure sheaves.
From the point of view of the McKay correspondence we are only 
interested in $G$-constellations which still make sense as generalized
$G$-orbits, i.e. whose support in $\mathbb{C}^3$ is a single $G$-orbit. 
Flat families of such $G$-constellations over $Y = G\text{-}\hilb
\mathbb{C}^3$ whose Hilbert-Chow map $Y \rightarrow \mathbb{C}^3/G$ 
coincides with the resolution morphism are called \em $\gnat$-families\rm. 
Tautologically, the universal family $\mathcal{M}$ of $G$-clusters 
is a $\gnat$-family. Given a $\gnat$-family $\mathcal{F}$ over $Y$  
we denote by $\mathcal{F}_p$ the $G$-constellation parametrised 
in $\mathcal{F}$ by a point $p \in Y$. 

Since $\mathbb{C}^3$ is affine, any $G$-constellation
$\mathcal{W}$ is determined by its space
of global sections $\Gamma(\mathcal{W})$ with the natural  
action of $\twalg$ on it. Since $\regrep^* \simeq \regrep$, 
the complex vector space dual $\Gamma(\mathcal{W})^*$ with 
the dual $\twalg$-action defines a new $G$-constellation
$\widetilde{\mathcal{W}}$ called \em the dual of \rm $\mathcal{F}$. 
Given a $\gnat$-family $\mathcal{F}$ we define similarly 
\em the dual $\gnat$-family \rm $\widetilde{\mathcal{F}}$ which is 
a fibre-wise dual of $\mathcal{F}$. 
In \cite[\S2]{CautisLogvinenko} we prove an alternative
description of dualizing a $G$-constellation:
$\widetilde{\mathcal{W}} = \mathcal{W}^\vee[n]$, 
where $\mathcal{W}^\vee$ is the derived dual
$\rder \shhomm_Y(\mathcal{W}, \mathcal{O}_Y)$. 
In other words, $\widetilde{\mathcal{W}}$ is the unique
non-zero cohomology sheaf of $\mathcal{W}^\vee$, located
in degree $n$. Similarly, $\widetilde{\mathcal{F}} = \mathcal{F}^\vee[n]$ 
\cite[Prop.\ 2.1]{CautisLogvinenko}. 
It follows naturally
\cite[Lemma 4]{Logvinenko-DerivedMcKayCorrespondenceViaPureSheafTransforms}
that on $Y = G\text{-}\hilb \mathbb{C}^3$ the inverse $\Psi$ of
the equivalence $\Phi$ of \cite{BKR01} is the Fourier-Mukai transform  
defined by the dual family $\widetilde{\mathcal{M}}$. 

\subsection{The McKay quiver of $G$ and its representations}
\label{section-mckay-quiver-of-G}

To any finite subgroup $G$ of $\gl_n(\mathbb{C})$ we associate 
a quiver $\mckquiv$ called \it the McKay quiver of $G$\rm. For 
an abelian $G \subset \gsl_3(\mathbb{C})$ quiver $Q(G)$ has 
as its vertices the characters $\chi \in G^\vee$ of $G$ and 
from every vertex $\chi$ there are three arrows going to 
vertices $\kappa(x_i) \chi$ for $i = 1,2,3$.
For each arrow $(\chi, x_i) \colon \chi \xrightarrow{x_{i}} \kappa(x_{i})\chi$
we fix a choice of basis for (one-dimensional) space 
$G$-$\homm_{\mathbb{C}}(\chi \otimes \mathbb{C} x_{i}, \chi \kappa(x_{i}))$.
This can be thought of as fixing a choice of all the 
Schur's lemma isomorphisms we need. 

There is a standard planar embedding of $Q(G)$ into a real two
dimensional torus $T_G$. It was first constructed 
by Craw and Ishii in \cite[\S10.2]{Craw-Ishii-02}. This embedding 
tesselates $T_G$ into $2|G|$ regular triangles. Near each vertex 
$\chi$ of $Q(G)$ this embedding looks as on Figure \ref{figure-15}. 
Note that since $\mckquiv$ is embedded in a real two-torus, 
vertices that appear distinct in Figure \ref{figure-15} 
need not be; the extreme case is that of $G$ being the trivial group, 
in such case all seven vertices of Figure \ref{figure-15} coincide.

We denote by $\hex(\chi)$ the subquiver of $\mckquiv$ which consists 
of the six triangles which meet at $\chi$. 
\begin{figure}[h] \begin{center}
\includegraphics[scale=0.17]{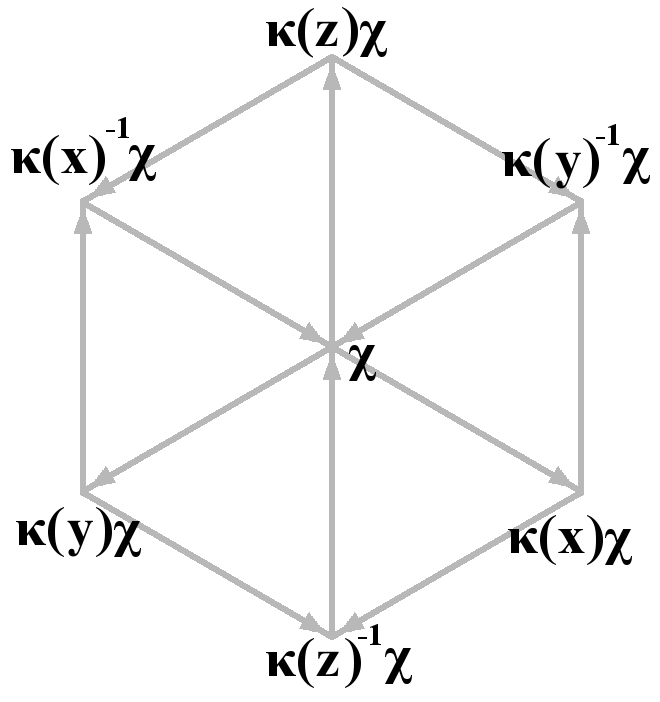} \end{center}
\caption{\label{figure-15} Subquiver $\hex(\chi)$ of $Q(G)$ 
(embedded in a real $2$-torus)}
\end{figure}

A representation of $\mckquiv$ consists of a vector space $V_\chi$
for each vertex $\chi$ and a linear map 
$\alpha_{\chi,i}: V_{\chi} \rightarrow V_{\chi \kappa(x_i)}$
for each arrow $(\chi,i)$. The category of 
the representations of $Q(G)$ in which  
the linear maps representing each square 
\begin{minipage}[c]{2in}
\xymatrix @R=0.25cm {
& 
\underset{\chi \kappa(x_i)}{\bullet}
\ar@<1ex>[dr]^{x_j}
& 
\\
\underset{\chi}{\bullet}
\ar@<1ex>[ur]^{x_i}
\ar@<1ex>[dr]_{x_j}
&
&  
\underset{\chi \kappa(x_i x_j)}{\bullet}
\\
&  
\underset{\chi\kappa(x_j)}{\bullet}
\ar@<1ex>[ur]_{x_i}
& 
} 
\end{minipage}
commute is equivalent to the category of quasi-coherent $G$-sheaves on 
$\mathbb{C}^n$. Given any quasi-coherent $G$-sheaf 
$\mathcal{M}$ on $\mathbb{C}^n$
we define its \em associated representation $\mckquiv_\mathcal{M}$ \rm 
as follows. The space of global sections $\Gamma(\mathcal{M})$ of
$\mathcal{M}$ is a $\twalg$-module and the evaluation map 
\begin{align}
\label{eqn-canonical-G-eigenspace-decomposition}
\bigoplus_{\chi \in G^\vee}
G\text{-}\homm_\mathbb{C}\left(\chi,\Gamma(\mathcal{M})\right) \otimes \chi 
\;\overset{\sim}{\longrightarrow}\;
\Gamma(\mathcal{M}) 
\end{align}
is an isomorphism which decomposes
$\Gamma(\mathcal{M})$ into $G$-eigenspaces. We define 
$\mckquiv_\mathcal{M}$ by setting 
$V_\chi = G\text{-}\homm_\mathbb{C}(\chi,\Gamma(\mathcal{M}))$, 
then under \eqref{eqn-canonical-G-eigenspace-decomposition}
the actions of $x_i$ on $\Gamma(\mathcal{M})$ become maps  
$$ \left( \bigoplus_{\chi \in G^\vee}
V_\chi \otimes \chi \right) \otimes \mathbb{C} x_i 
\; \xrightarrow{\text{action of } x_i}\;
\bigoplus_{\chi \in G^\vee} V_\chi \otimes \chi $$
and the choice of Schur's lemma isomorphisms we've fixed above
makes these into the requisite linear maps 
$\alpha_{\chi, x_i}\colon V_\chi \rightarrow V_{\chi \kappa(x_i)}$.  

Similarly, given a $\gnat$-family $\mathcal{F} \in \cohcat^G(Y \times
\mathbb{C}^n)$ the direct image $\pi_{Y *} \mathcal{F}$ is a sheaf
of $\mathcal{O}_Y\otimes_{\mathbb{C}}\left(\twalg\right)$-modules 
on $Y$ which is
locally-free as an $\mathcal{O}_Y$-module. Since the action of $G$ on 
$Y$ is trivial, we can decompose $\pi_{Y *} \mathcal{F}$ into
$G$-eigensheaves as $\bigoplus \mathcal{L}_\chi \otimes \chi$, 
where each $\mathcal{L}_\chi$ is a line bundle on $Y$. We then
define the \em associated representation $\mckquiv_{\mathcal{F}}$ \rm  
by representing each vertex $\chi$ of $\mckquiv$
by $\mathcal{L}_\chi$ and representing each arrow
$(\chi, x_i)$ of $\mckquiv$ by the map 
$\alpha_{\chi,i} \colon \mathcal{L}_{\chi} \rightarrow
\mathcal{L}_{\chi \kappa(x_i)}$, which is obtained as above:
taking the action of $x_i$ on $\pi_{Y *} \mathcal{F}$ and
restricting it to $\mathcal{L}_\chi$ via our choice of 
Schur's lemma isomorphisms.

We say that in $\mathcal{F}$ an arrow $(\chi,x_i)$ 
vanishes at a point $p \in Y$ \rm if so does
map $\alpha_{\chi,x_i}$ of $\mckquiv_{\mathcal{F}}$. 
The locus of all such points is a divisor called the 
\em vanishing divisor $B_{\chi,x_i}$ of $(\chi, x_i)$ 
in $\mathcal{F}$\rm. Note that divisors $B_{\chi,x_i}$ are 
independent of our choice of 
Schur's lemma isomorphisms, since a different choice would 
amount to multiplying each $\alpha_{\chi, x_i}$ by a non-zero scalar. 

For any abelian $G \subset \gsl_3(\mathbb{C})$ any
$\gnat$-family $\mathcal{F}$ on $Y$ can be written down
numerically using the classification of $\gnat$-families 
introduced in \cite{Logvinenko-Natural-G-Constellation-Families}. 
The vanishing divisors $B_{\chi,x_i}$ of $\mckquiv_{\mathcal{F}}$ can then 
be explicitly computed in terms of the numerical data defining $\mathcal{F}$
\cite[\S4.6]{Logvinenko-DerivedMcKayCorrespondenceViaPureSheafTransforms}. 
A detailed example of such computation for the dual family $\widetilde{M}$ 
of the universal family of $G$-clusters is given
in \cite[\S6]{CautisLogvinenko}.

\subsection{Transforms $\Psi(\mathcal{O}_0 \otimes \chi)$ and
skew-commutative cubes of lines bundles}
\label{section-psi-O-chi-and-skew-commutative-cubes-of-line-bundle}

Derived Reid's recipe assigns to each character $\chi \in G^\vee$ 
the object $\Psi(\mathcal{O}_0 \times \chi)$ of $D(Y)$. 
It was shown in \cite[\S2]{CautisLogvinenko}
that $\Psi(\mathcal{O}_0 \otimes \chi)$ 
is isomorphic to the total complex of the skew-commutative cube 
of line bundles obtained
from the subrepresentation $\hex(\chi^{-1})_{\widetilde{\mathcal{M}}}$ of 
$\mckquiv_{\widetilde{\mathcal{M}}}$ in the following way.

The generalities on cubes of line bundles are laid out in 
\cite{CautisLogvinenko}, \S3. For our purposes, a cube of 
line bundles on $Y$ is the following 
collection of line bundles and maps between them:
\begin{align}
\label{eqn-abstract-cube-of-line-bundles}
\xymatrix{ & \mathcal{L}_{23} \ar"2,5"^<<{\alpha^3_{2}}
\ar"3,5"_<<{\alpha^2_{3}} & & & \mathcal{L}_1 \ar"2,6"^{\alpha^1} & \\
\mathcal{L}_{123} \ar"1,2"^{\alpha^1_{23}} \ar"2,2"^{\alpha^2_{13}}
\ar"3,2"_{\alpha^3_{12}} & \mathcal{L}_{13}
\ar"1,5"^>>>>{\alpha^3_{1}} \ar"3,5"_>>>>>{\alpha^1_{3}} & & &
\mathcal{L}_2 \ar"2,6"^{\alpha^2} & \mathcal{L} \\ & \mathcal{L}_{12}
\ar"1,5"^<<{\alpha^2_{1}} \ar"2,5"_<<{\alpha^1_{2}} & & &
\mathcal{L}_3 \ar"2,6"_{\alpha^3} & } 
\end{align}
If the cube in \eqref{eqn-abstract-cube-of-line-bundles} is
skew-commutative, we define its \em total complex $T^\bullet$ \rm to be
$$ 0 
   \;\rightarrow\; 
   \mathcal{L}_{123} 
   \;\rightarrow\; 
   \mathcal{L}_{23} \oplus \mathcal{L}_{13} \oplus \mathcal{L}_{12}
   \;\rightarrow\; 
   \mathcal{L}_{1} \oplus \mathcal{L}_{2} \oplus \mathcal{L}_{3}
   \;\rightarrow\; 
   \underline{\mathcal{L}}
   \;\rightarrow\; 
   0 
$$
The differential maps of $T^\bullet$ are defined by summing up 
all the appropriate maps of the cube. 

On Fig.\ \ref{figure-13a} we draw the subquiver $\hex(\chi^{-1})$ of 
the McKay quiver $\mckquiv$. We arrange the subrepresentation 
$\hex(\chi^{-1})_{\widetilde{\mathcal{M}}}$ into a commutative cube 
of line bundles as depicted on Fig.\ \ref{figure-13b}. 
\begin{figure}[!htb] \centering 
\subfigure[The subquiver $\hex(\chi^{-1})$] { \label{figure-13a}
\includegraphics[scale=0.29]{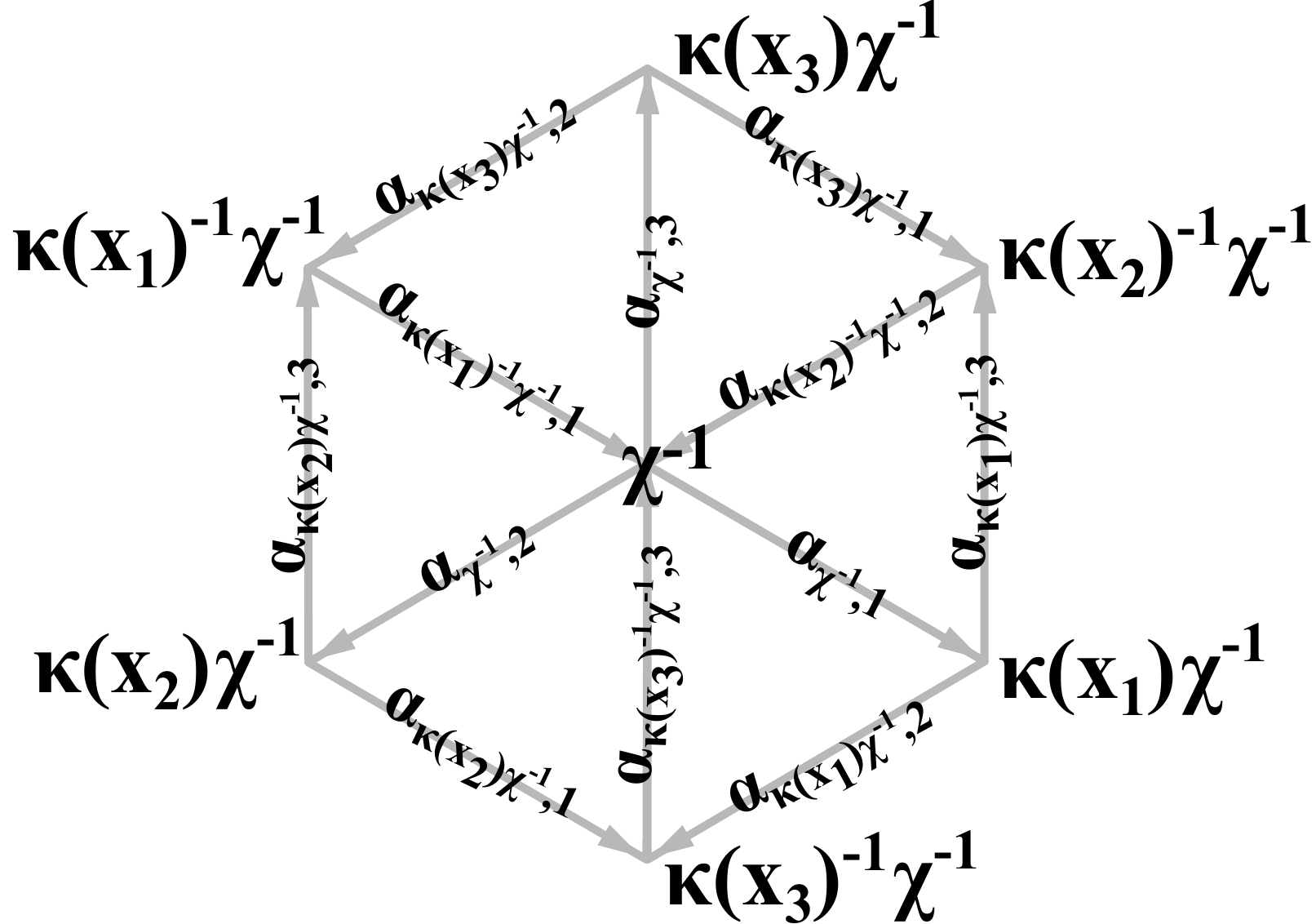}} 
\subfigure[The corresponding cube] { \label{figure-13b}
\includegraphics[scale=0.30]{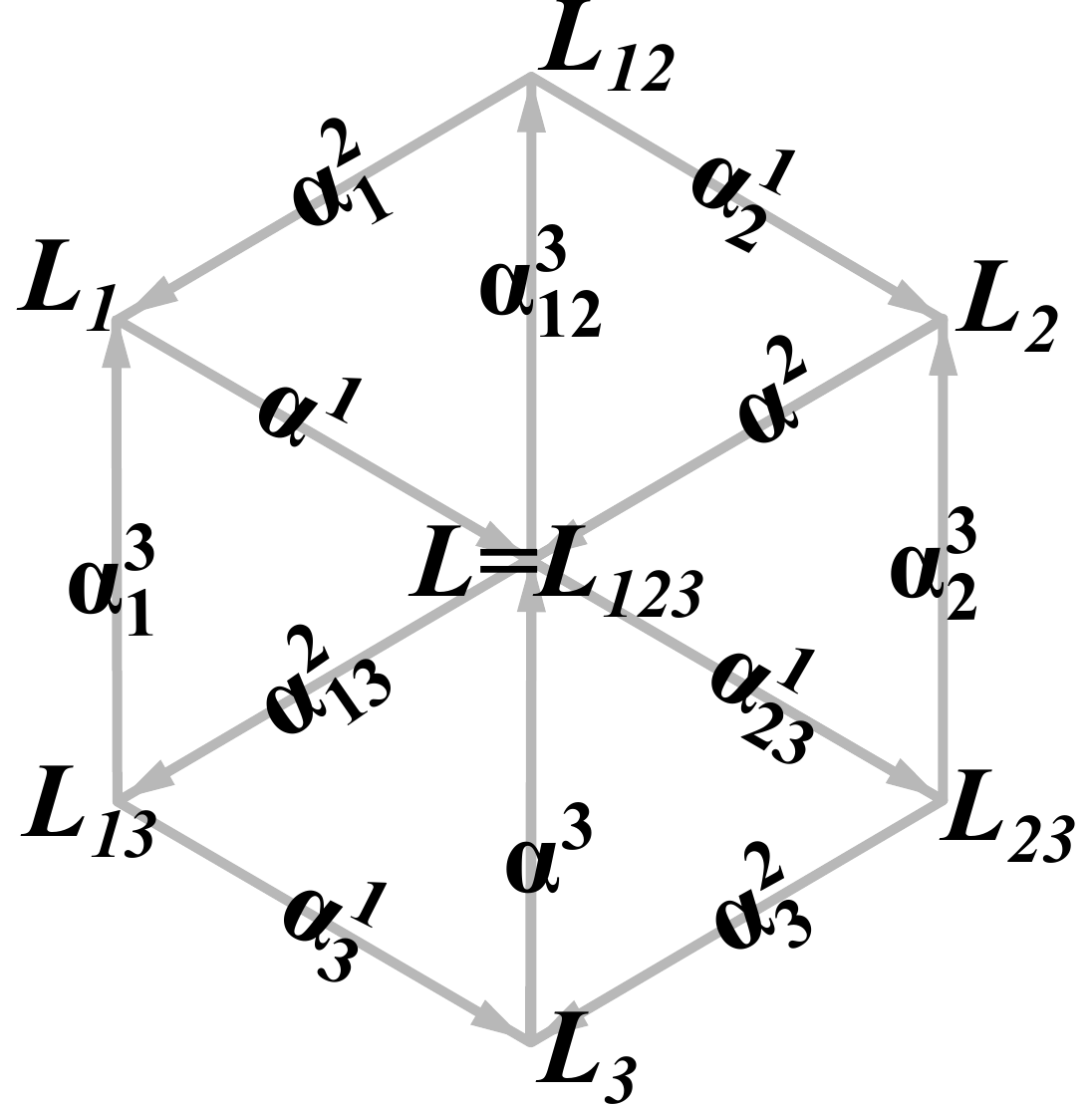}}
\caption{The cube of line bundles defined by 
$\hex(\chi^{-1})_{\widetilde{\mathcal{M}}}$}
\label{figure-13}
\end{figure}
We make this commutative cube skew-commutative by setting 
$\alpha^{i}_{(\dots)} = (-1)^{i} \alpha^{i}_{(\dots)}$.
The total complex $T^\bullet$ of the resulting cube is isomorphic to 
$\Psi(\mathcal{O}_0 \times \chi)$ in $D(Y)$
\cite[Prop.\ 4.6]{CautisLogvinenko}. 

Write $D^1_{23}$ for the vanishing divisor of map $\alpha^1_{23}$ of 
the cube in \eqref{eqn-abstract-cube-of-line-bundles} and
similarly for its other maps. The cohomology sheaves of 
$T^\bullet$ can be computed in terms of these vanishing divisors: 
\begin{lemma}
\label{lemma-cube_cohomology} Let $T^\bullet$ be the total complex of 
the skew-commutative cube in \eqref{eqn-abstract-cube-of-line-bundles}. 
Then: 
\begin{enumerate} 
\item 
\label{item-degree-zero-cohomology}
$H^0(T^\bullet) \cong \L \otimes \O_Z$ 
where $Z$ is the scheme theoretic intersection 
$D^1 \cap D^2 \cap D^3$. 

\item \label{item-degree-minus-one-cohomology}
For any permutation $\{I, J, K\}$ of $\{ 12, 13, 23 \}$ there
is a three-step filtration of $H^{-1}(T^\bullet)$ with successive 
quotients $\mathcal{F}''_I$, $\mathcal{F}'_J$
and $\mathcal{F}_K$. Here:
\begin{itemize} 
\item 
$\mathcal{F}_{12} = \O_Z \otimes \L_{12}(\gcd(D_1^2,D_2^1))$ 
where $Z$ is the scheme theoretic intersection of 
$\gcd(D_1^2,D_2^1)$ and the effective part of 
$D^3 + \lcm(D_3^1,D_3^2) - \tD_1^2 - D^1$ 
\item 
$\mathcal{F}_{13} = \O_Z \otimes \L_{13}(\gcd(D_1^3,D_3^1))$ 
where $Z$ is the scheme theoretic intersection 
of $\gcd(D_1^3,D_3^1)$ and the effective part of $D^2 +
\lcm(D_2^1,D_2^3) - \tD_3^1 - D^3$ 
\item 
$\mathcal{F}_{23} = \O_Z \otimes \L_{23}(\gcd(D_2^3,D_3^2))$ 
where $Z$ is the scheme theoretic intersection of 
$\gcd(D_2^3,D_3^2)$ and the effective part of 
$D^1 + \lcm(D_1^2,D_1^3) - \tD_2^3 - D^2$.
\end{itemize} 
and $\tD^i_j =  D^i_j - \gcd(D^i_j,D^j_i)$. 
To define $\mathcal{F}'_{J}$ 
we replace $\lcm(D^i_k, D^j_k)$ in its definition 
with $\lcm(D_k^i,\tD_k^j)$ where $j$ is chosen so that $K=\{jk\}$. 
To define $\mathcal{F}''_{I}$ we replace $\lcm(D^i_k, D^j_k)$ 
with $\lcm(\tD_k^i,\tD_k^j)$. 

\item 
\label{item-degree-minus-two-cohomology}
$H^{-2}(T^\bullet) \cong \L_{123}(D) \otimes \O_D $
where 
$D = \gcd(D_{23}^1, D_{13}^2, D_{12}^3)$ 
\item 
\label{item-all-cohomologies-vahish-except-for-degree-zero-to-minus-two}
$H^{-i}(T^\bullet) \cong 0$ for all $i \neq 0, \text{-}1$ and $\text{-}2$.  
\end{enumerate} 
\end{lemma}  
\begin{proof}
See \cite[Theorem 1.1 and Remark 3.4]{CrawQuinteroVelez-CohomologyOfWheelsOnToricVarieties}.
\end{proof}
 
\section{Sink-source graphs and non-compact exceptional divisors}
\label{section-sink-source-graphs-and-non-compact-except-divisors}

Let $Y$ be $G$-$\hilb(\mathbb{C}^3)$, $\widetilde{\mathcal{M}}$ be dual to 
the universal family of $G$-clusters and $E$ be a toric divisor on $Y$. 
In \cite[Prop.\  4.7]{CautisLogvinenko} we've classified the vertices 
of $Q(G)$ according to which of the arrows in the subquiver $\hex(\chi)$ 
on Fig.\ \ref{figure-15} vanish along $E$ in $\widetilde{\mathcal{M}}$. 
On Fig.\ \ref{figure-17} - \ref{figure-18} we list all 
possible cases, drawing in black the arrows which vanish and in grey 
the arrows which don't. There are four vertex types:
the \it charges\rm, the \it sources\rm, the \it sinks \rm and the \it tiles\rm.
The charge vertices always occur in $\mckquiv$ in straight lines 
propagating from a source vertex to a sink vertex. 
An $x_i$-oriented charge propagates along $x_i$-arrows of 
$\mckquiv$. A type $(1,0)$-charge propagates in the direction of the arrows, 
while a type $(0,1)$-charge propagates against it. 
A type $(a,b)$-source (resp.  sink) emits (resp.  receives) 
$a$ charges of type $(1,0)$ and $b$ charges of type $(0,1)$. 
\begin{figure}[h]
\begin{center}
\includegraphics[scale=0.050]{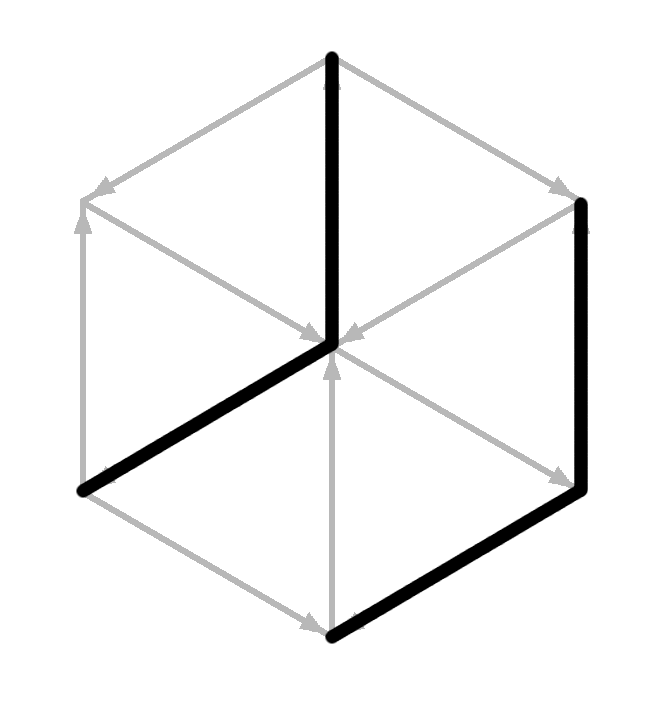}
\includegraphics[scale=0.050]{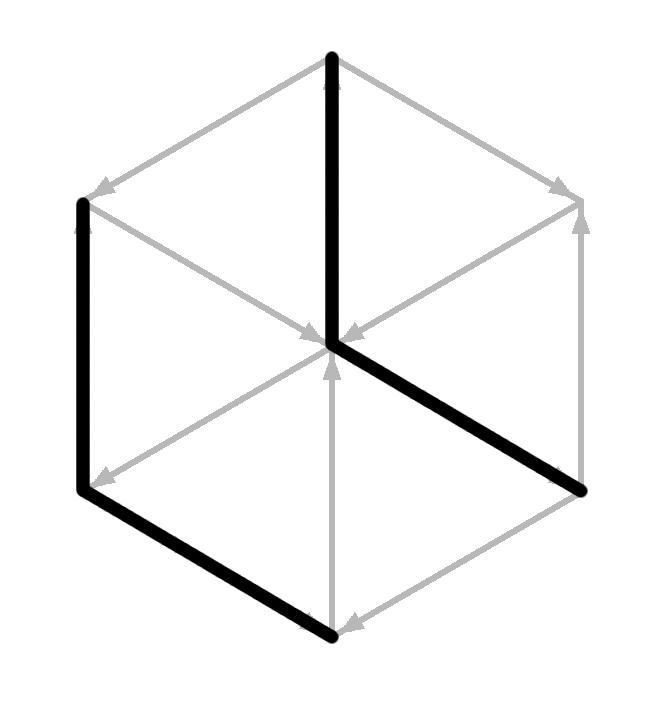}
\includegraphics[scale=0.050]{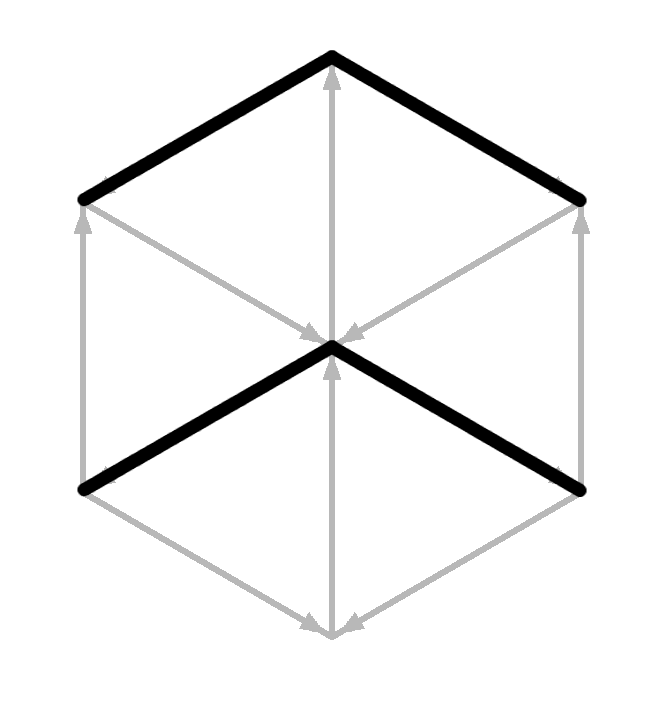}
\includegraphics[scale=0.050]{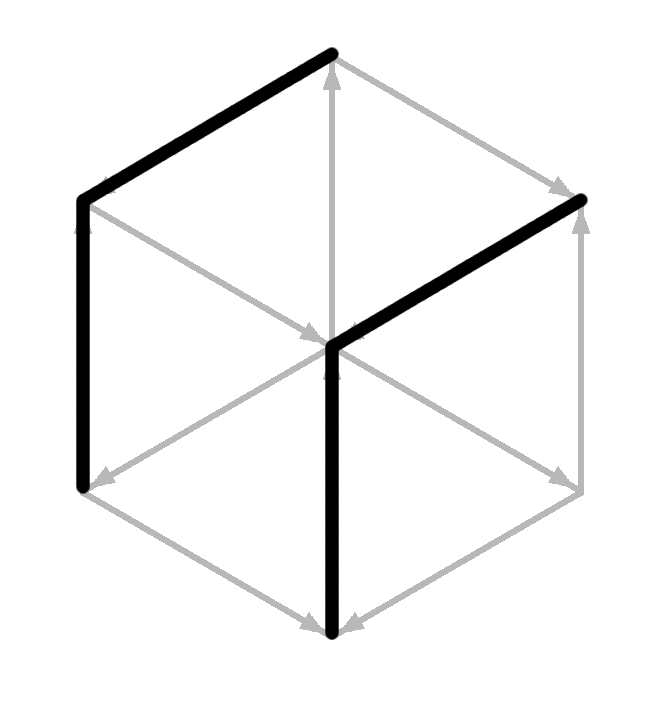}
\includegraphics[scale=0.050]{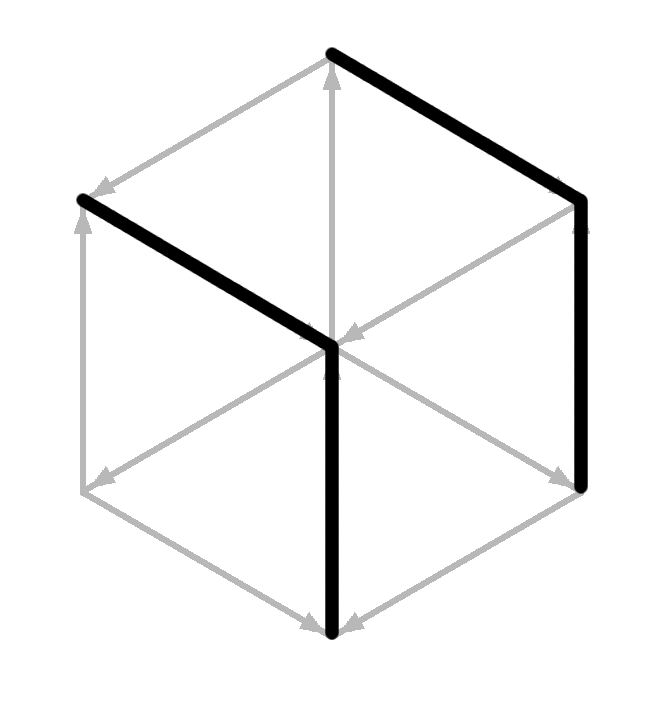}
\includegraphics[scale=0.050]{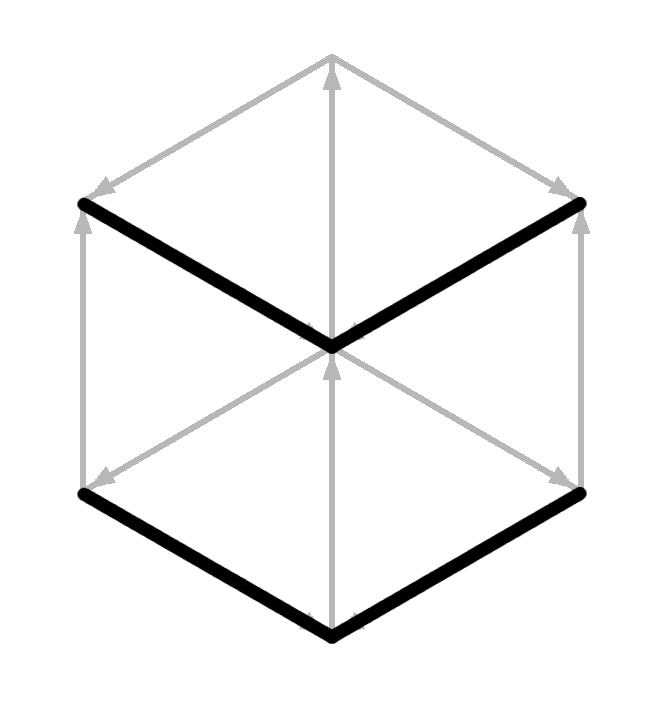}
\caption{\label{figure-17} The $x$-, $y$-, $z$-$(1,0)$-charges and
the $x$-, $y$-, $z$-$(0,1)$-charges}
\end{center} 
\begin{center}
\includegraphics[scale=0.050]{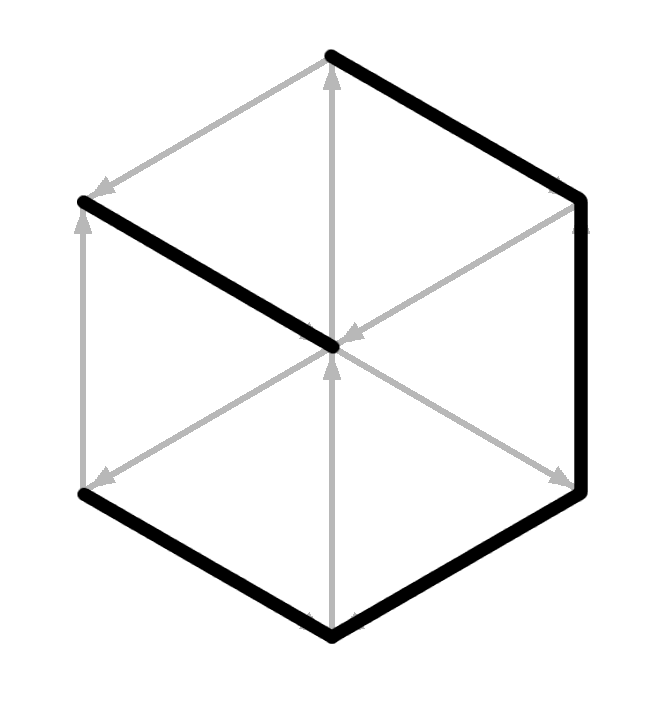}
\includegraphics[scale=0.050]{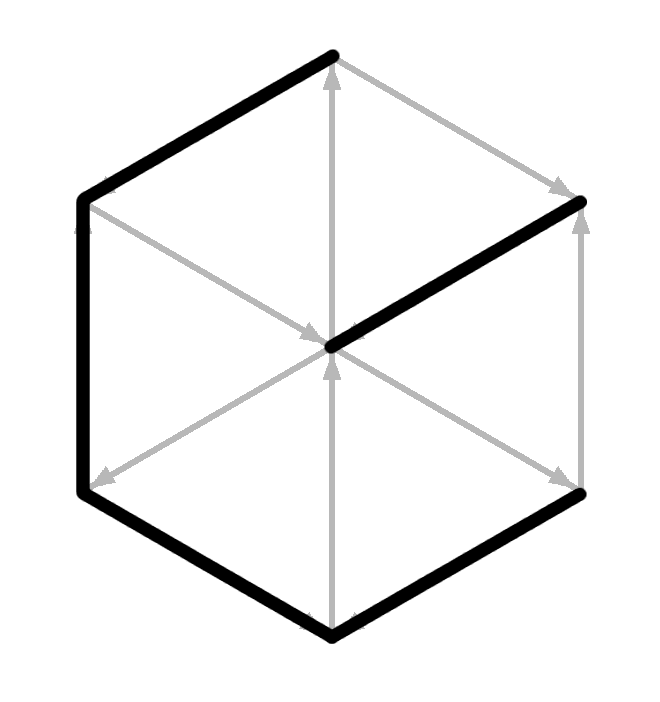}
\includegraphics[scale=0.050]{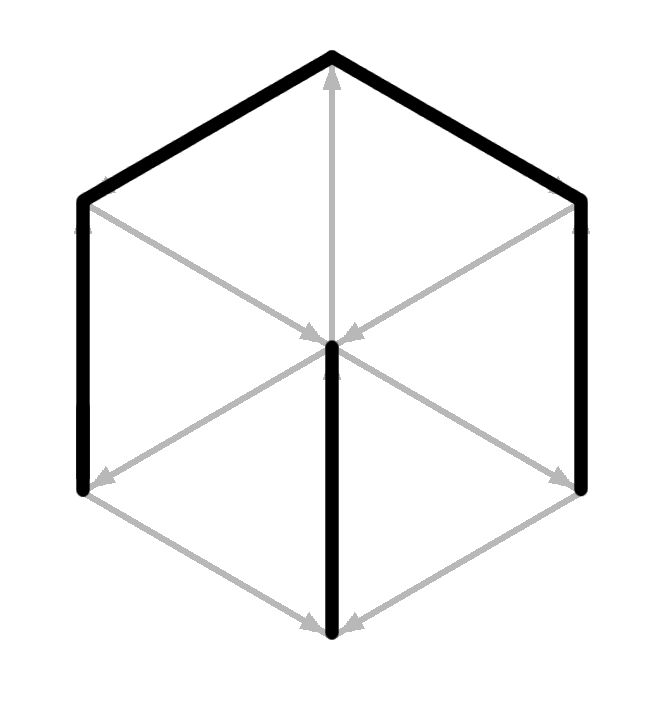}
\includegraphics[scale=0.050]{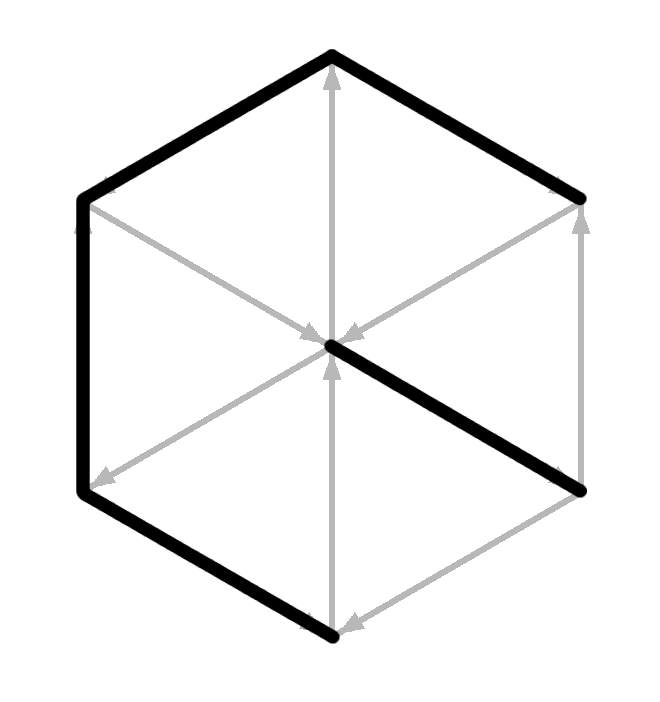}
\includegraphics[scale=0.050]{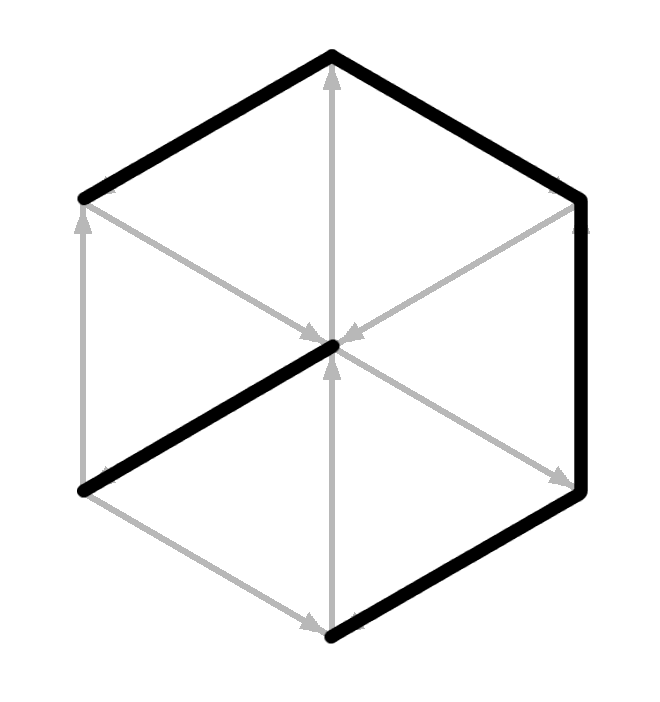}
\includegraphics[scale=0.050]{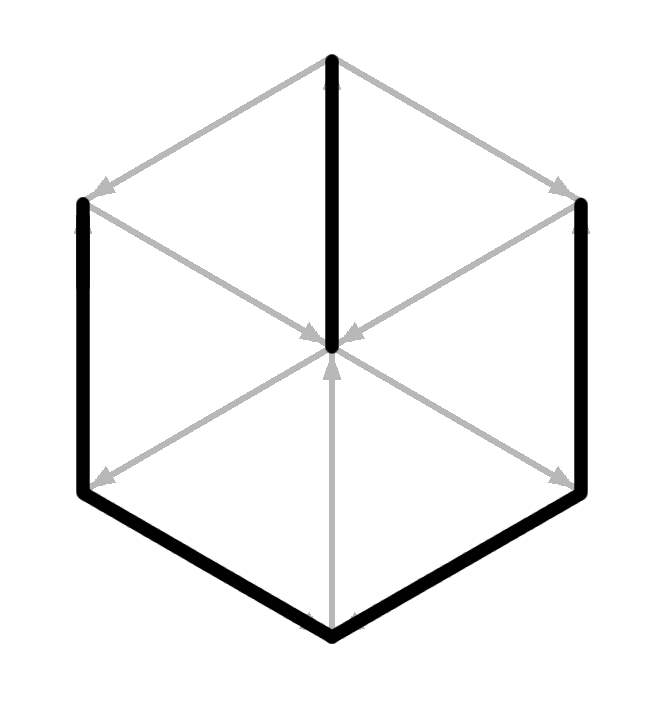}
\includegraphics[scale=0.050]{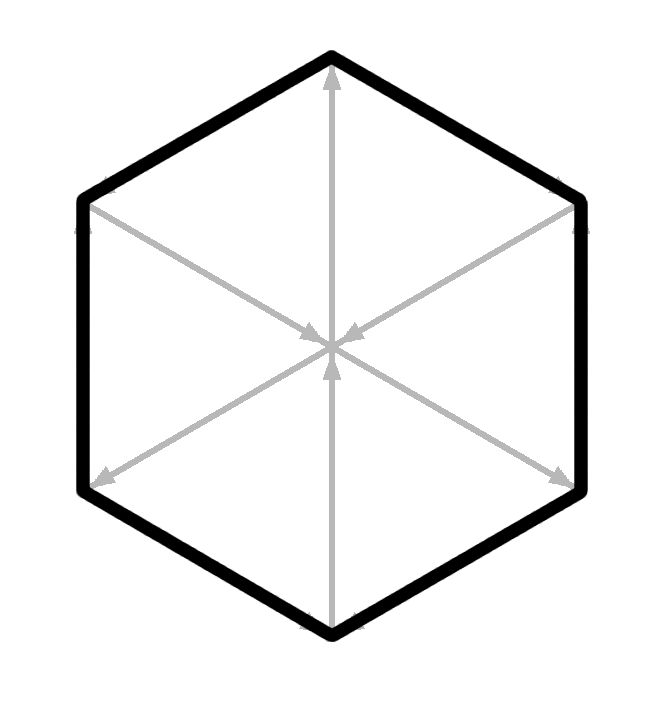}
\caption{\label{figure-19} The $x$-, $y$-, $z$-$(1,2)$-sources, 
the $x$-, $y$-, $z$-$(2,1)$-sources and the $(3,3)$-source.} \end{center} 
\begin{center} \includegraphics[scale=0.050]{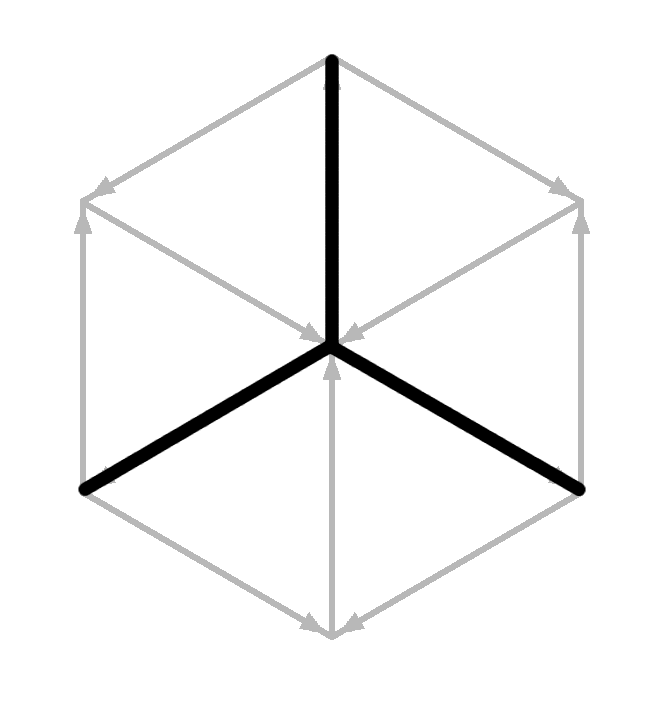}
\includegraphics[scale=0.050]{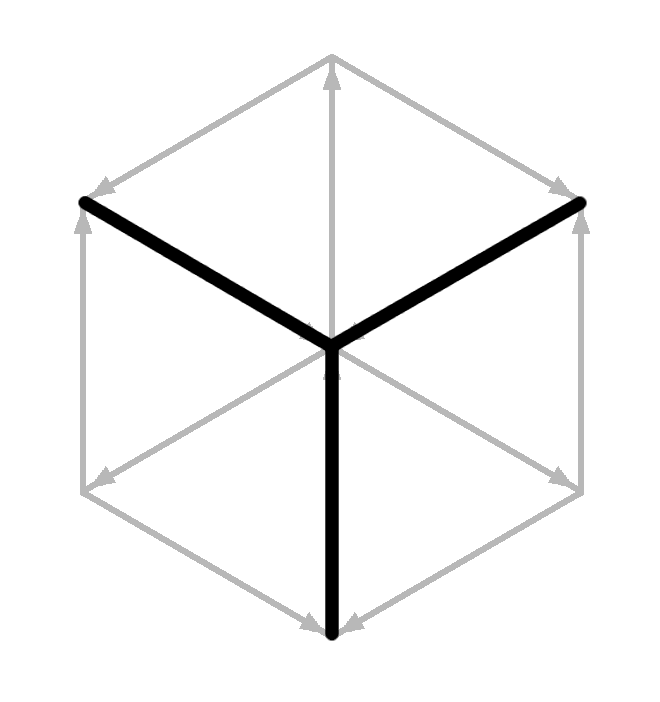}
\caption{\label{figure-16} The $(3,0)$-sink and the $(0,3)$-sink.}
\end{center} 
\begin{center}
\includegraphics[scale=0.050]{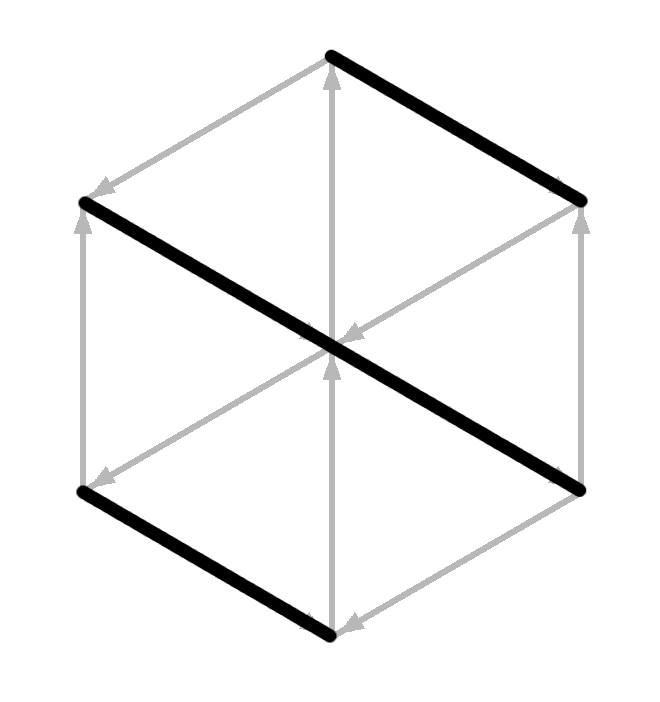}
\includegraphics[scale=0.050]{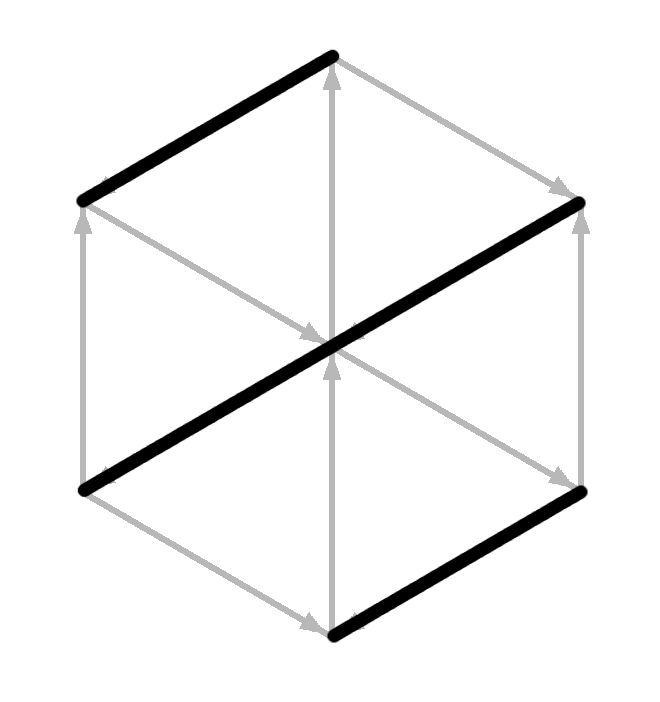}
\includegraphics[scale=0.050]{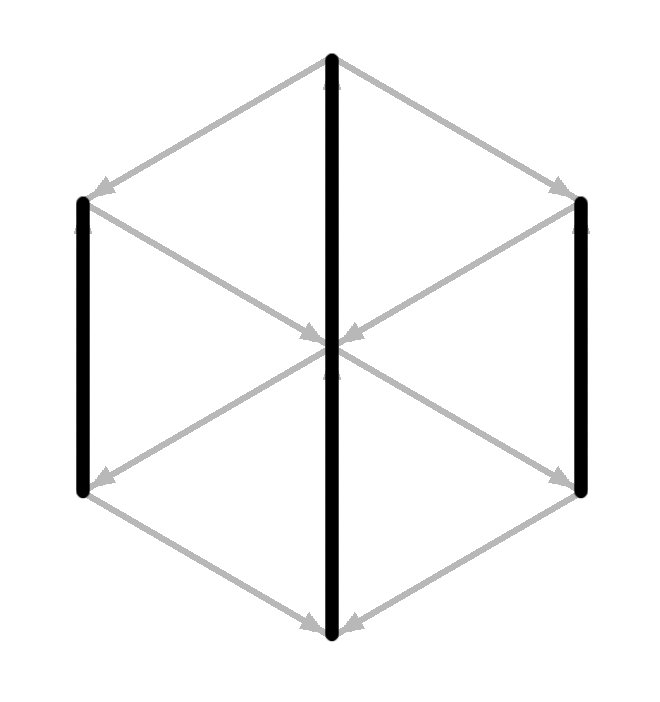}
\caption{\label{figure-18} The $x$-tile, the $y$-tile and the
$z$-tile.} \end{center} 
\end{figure} 

The \it sink-source graph \rm $SS_{\widetilde{\mathcal{M}}, E}$ is a graph 
drawn on top of $\mckquiv$ whose vertices are the sinks and the sources and 
whose edges are the charge lines. This graph subdivides the torus $T_H$ 
into several connected regions which are each $x$-, $y$- or $z$-tiled.
In particular, if $\sinksource_{\widetilde{\mathcal{M}},E}$ is empty 
then the whole of $T_H$ is either $x$-, $y$- or $z$-tiled.
Here we say that a region is, for example, $x$-tiled if all its 
internal vertices are $x$-tiles.  

The sink-source graph $SS_{\widetilde{\mathcal{M}}, E}$ 
completely determines the divisor $E$, because we can read 
off from $SS_{\widetilde{\mathcal{M}}, E}$ which arrows of $\mckquiv$ 
do and do not vanish along $E$, and then apply the following:
\begin{lemma}\label{lemma-number-of-diamonds}
Let $E$ be a torus-invariant divisor of $Y$. Suppose the total numbers 
of $x$-, $y$- and $z$-oriented arrows of $Q(G)$ which vanish along $E$ 
in $\widetilde{\mathcal{M}}$ are $a$, $b$ and $c$, respectively.
Then $E$ is the divisor defined by 
$e = \frac{1}{|G|}(a,b,c) \in \mathfrak{E} \subset \mathbb{Q}^3$.
\end{lemma}
\begin{proof}
See \cite[Lemma 2.5]{Logvinenko-ReidsRecipeAndDerivedCategories}. 
Although \cite{Logvinenko-ReidsRecipeAndDerivedCategories} assumes
throughout that $G$ is such that $\mathbb{C}^3/G$ has a single
isolated singularity at the origin, 
the proof of its Lemma 2.5 doesn't actually use this assumption. 
\end{proof}

If $E$ is a compact exceptional divisor of $Y$, then it was shown in
\cite[Prop.\ 3.1-3.3]{Logvinenko-ReidsRecipeAndDerivedCategories}
that there are only three possible shapes that $SS_{\widetilde{\mathcal{M}}, E}$ 
can have and these correspond precisely to $E$ being a $\mathbb{P}^2$, 
a rational scroll blown-up in $0$, $1$ or $2$ points or a del Pezzo
surface $dP_6$. Moreover, as demonstrated on Figures
\ref{figure-20}-\ref{figure-22}, the precise dimensions 
of $SS_{\widetilde{\mathcal{M}}, E}$ determine completely 
the triangulation $\Sigma$ locally around the corresponding 
point $e \in \mathfrak{E}$. In particular, they determine
monomial ratios which carve out the edges incident to $e$ in $\Sigma$. 
This provides a crucial link with Reid's recipe marking described in 
\S \ref{subsection-reids-recipe}. E.g. the characters with which 
Reid's recipe prescribes to mark $E$ are precisely 
the $(0,3)$-sink vertices of $SS_{\widetilde{\mathcal{M}}, E}$.  

\begin{figure}[!htb] \centering 
\subfigure[The sink-source graph] { \label{figure-20a} 
\includegraphics[scale=0.060]{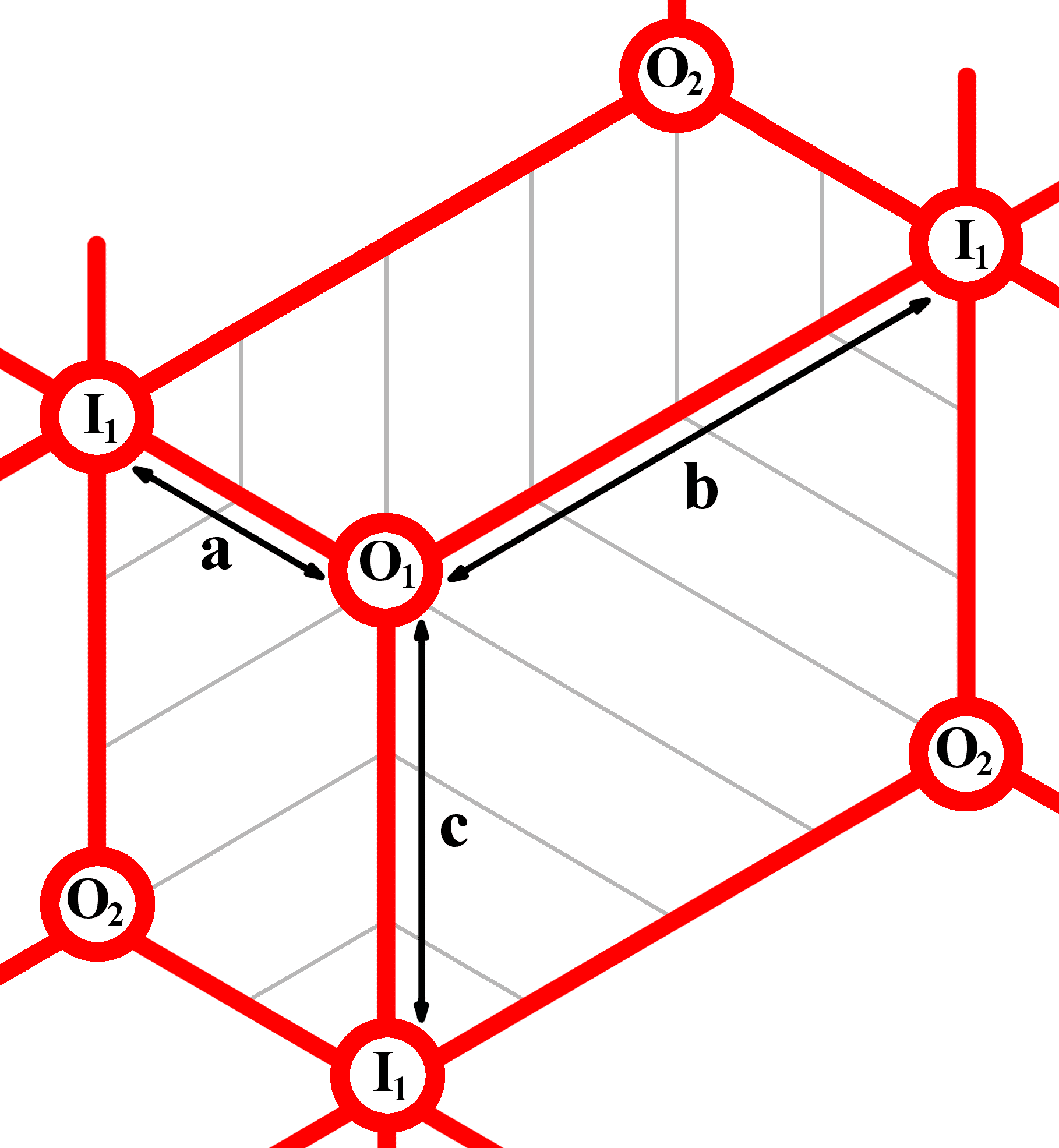}} \hspace{0.1cm}
\subfigure[Triangulation $\Sigma$ around $e$] { \label{figure-20b}
\includegraphics[scale=0.20]{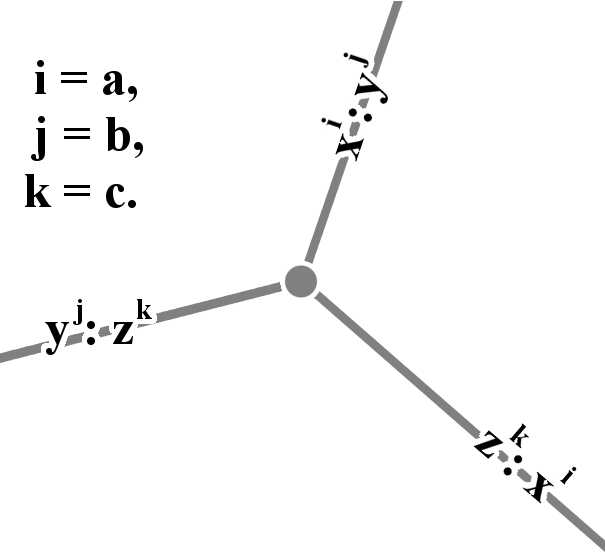}} 
\caption{One $(3,3)$-source} \label{figure-20}
\end{figure}
\begin{figure}[!htb] \centering 
\subfigure[The sink-source graph]{\label{figure-21a} 
\includegraphics[scale=0.07]{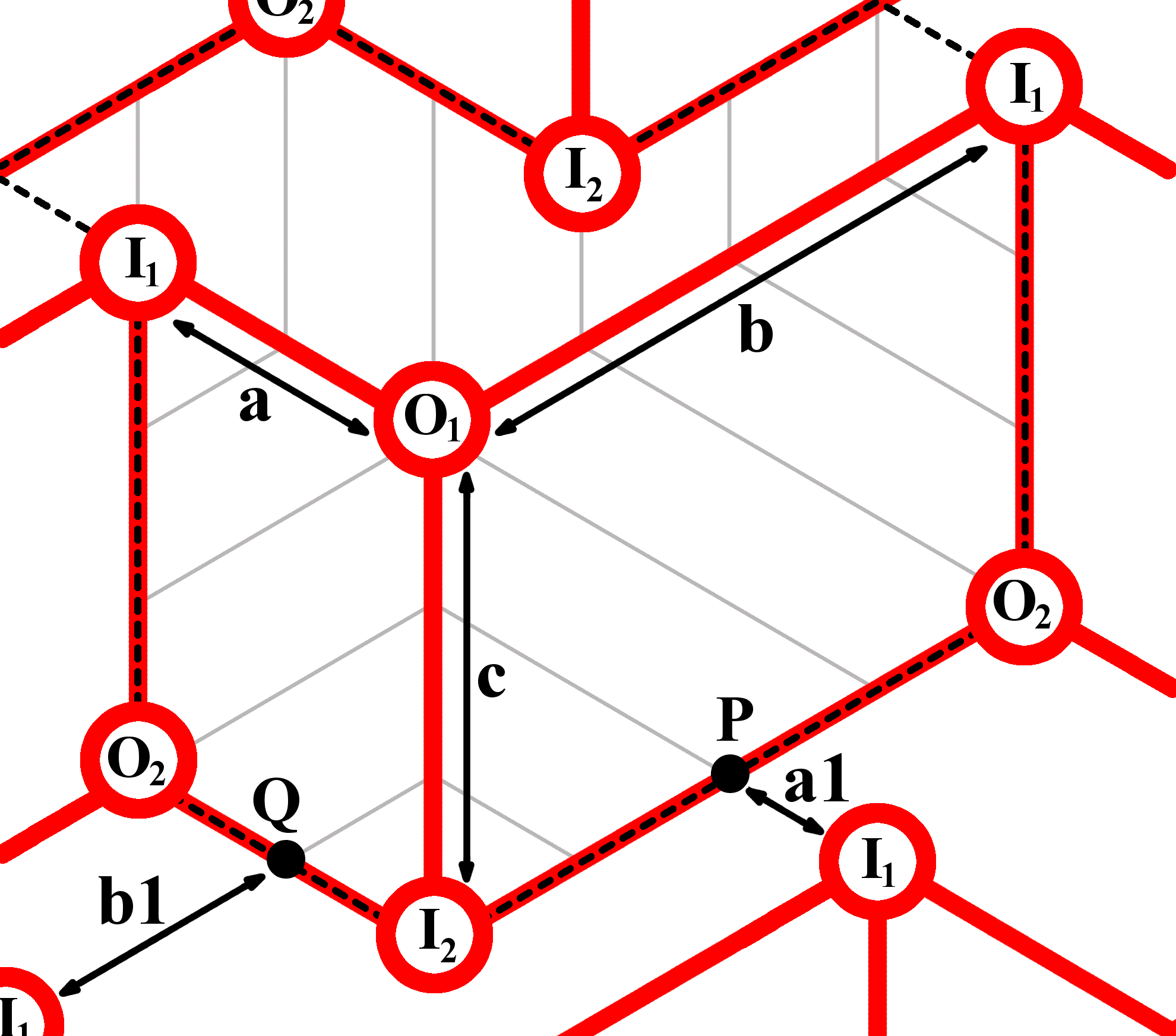}}  
\subfigure[Triangulation $\Sigma$ around $e$] { \label{figure-21b}
\includegraphics[scale=0.21]{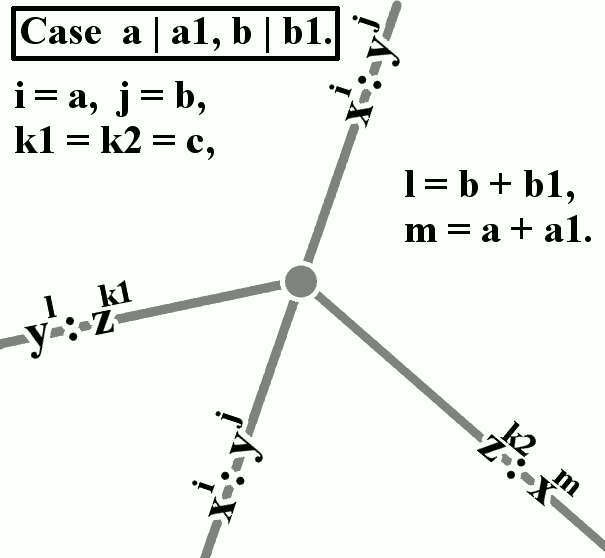}} 
\\
\subfigure[Triangulation $\Sigma$ around $e$]
{\label{figure-21c}
\includegraphics[scale=0.20]{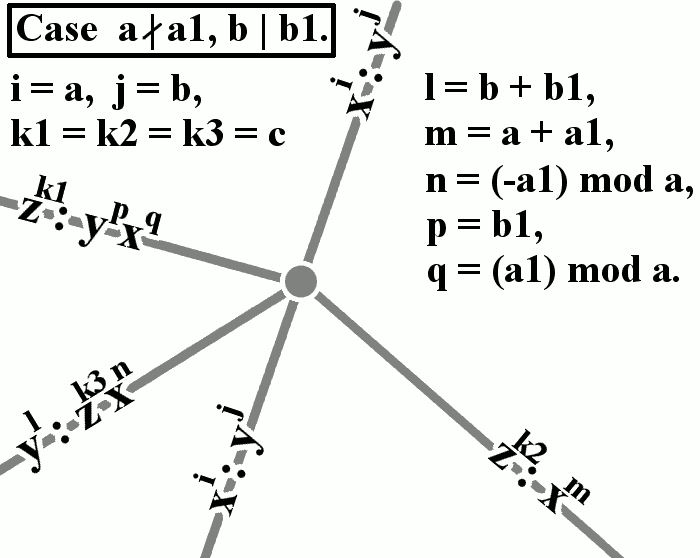}} 
\subfigure[Triangulation $\Sigma$ around $e$] { \label{figure-21d}
\includegraphics[scale=0.20]{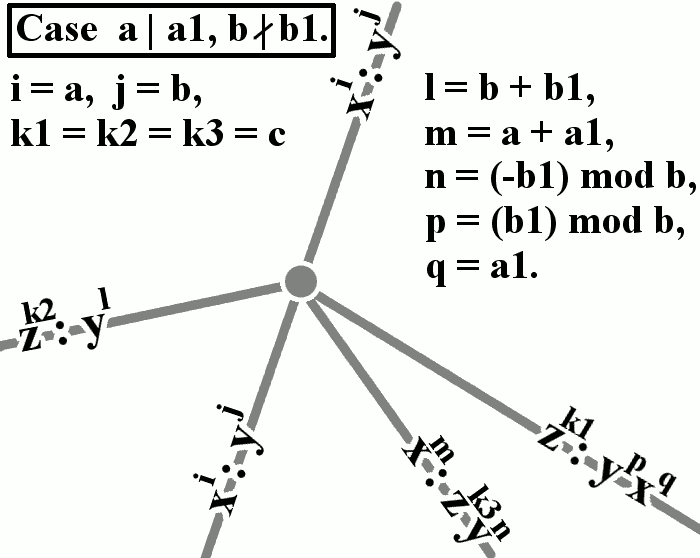}} 
\subfigure[Triangulation $\Sigma$ around $e$] { \label{figure-21e}
\includegraphics[scale=0.20]{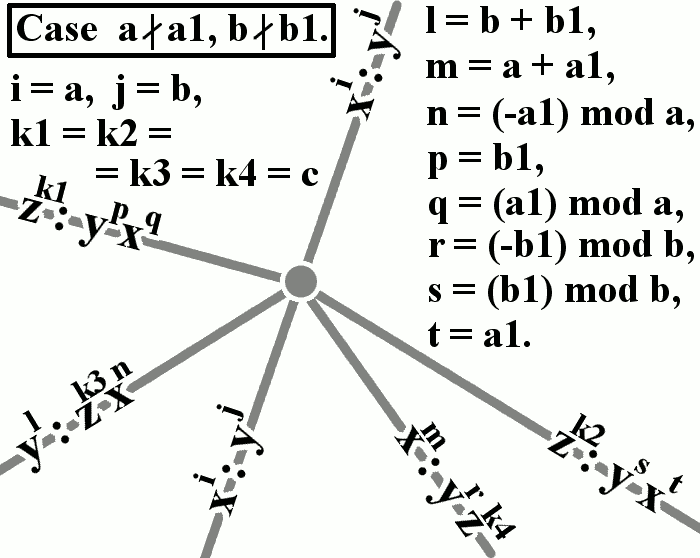}} 
\caption{One $(1,2)$-source and one $(2,1)$-source} \label{figure-21}
\end{figure}
\begin{figure}[!htb]
\centering 
\subfigure[Sink-source graph $\sinksource_{\widetilde{M},E}$] { \label{figure-22a} 
\includegraphics[scale=0.07]{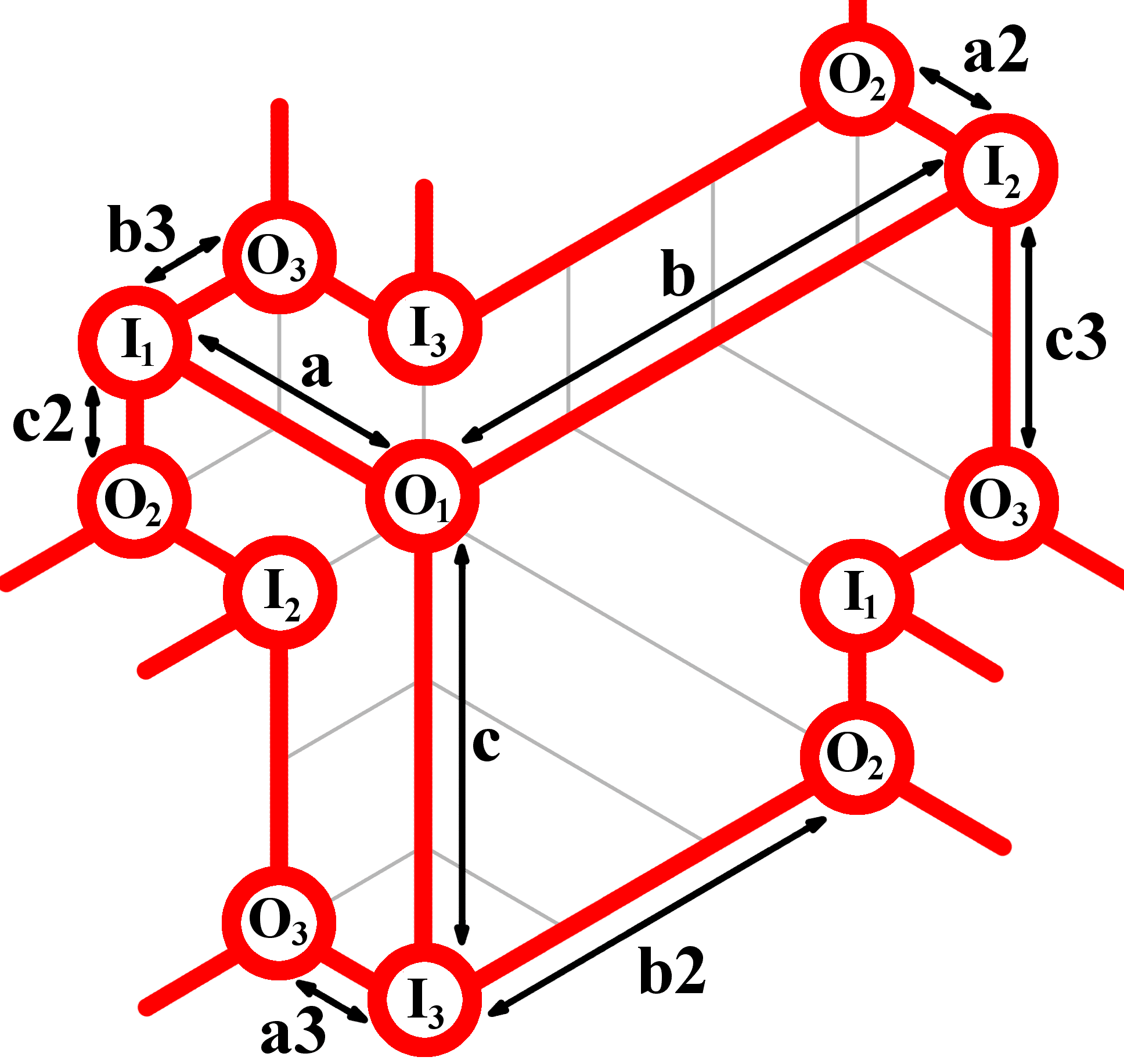}} \hspace{0.1cm}
\subfigure[Triangulation $\Sigma$ around $e$] 
{\label{figure-22b} 
\raisebox{0.5cm}{\includegraphics[scale=0.25]{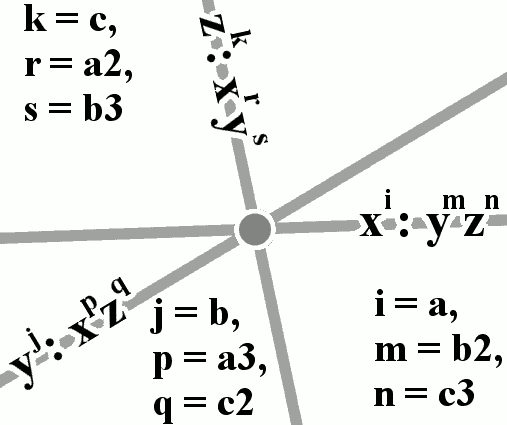}}}
\caption{Three $(2,1)$-sources} \label{figure-22}
\end{figure}

In this section, we extend the above results on sink-source
graphs to cover the non-compact exceptional divisors, which
appear when the singularities of $\mathbb{C}^3/G$ are not isolated: 

\begin{lemma}
\label{lemma-divisors-to-sink-source-graphs}
Let $E$ be an irreducible toric divisor on $Y$.
The sink-source graph $\sinksource_{\widetilde{\mathcal{M}},E}$ 
\begin{enumerate}
\item is empty if and only if $E$ is the strict
transform of one of the coordinate hyperplanes in $\mathbb{C}^3/G$. 
Moreover, $T_H$ is $x$-tiled (resp. $y$-tiled, $z$-tiled) if 
and only if $E$ is the strict transform of the $yz$-plane (resp. $xz$-plane, 
$xy$-plane). 

\item consists of one looping $(1,0)$-charge line (passing
through vertex $\chi_0$) and one looping $(0,1)$-charge line 
if $E$ is a non-compact exceptional divisor.

\item contains exactly one $(3,0)$-sink (given by vertex $\chi_0$)
if $E$ is a compact exceptional divisor.
\end{enumerate}
\end{lemma}
\begin{proof}
By definition $\sinksource_{\mathcal{M}, E}$ is empty and $T_H$ is,
for example, $x$-tiled if and only if every $x$-oriented arrow 
vanishes along $E$ in the associated representation
$Q(G)_{\widetilde{\mathcal{M}}}$, while none of $y$- or $z$-oriented ones do. 
On the other hand, we know that $\pi(E)$ is a closed torus-invariant 
subset of $\mathbb{C}^3/G$. 
This makes it either the origin, one of the three coordinate
lines, one of the three coordinate hyperplanes or the whole of 
$\mathbb{C}^3/G$. Among these, the property that $x^{|G|}$ vanishes 
along the subset, while $y^{|G|}$ and $z^{|G|}$ do not, uniquely identifies 
the $yz$-plane. This settles the first assertion. 

Suppose now $E$ is a compact exceptional divisor. 
Then its image in $\mathbb{C}^3/G$ can only be 
the origin. Therefore the argument in the proof of Prop.\ 4.14 
of \cite{CautisLogvinenko} demonstrates that 
$\sinksource_{\widetilde{\mathcal{M}},E}$ has a single $(3,0)$-sink given 
by vertex $\chi_0$. This settles the third assertion. 

Finally, suppose $E$ is a non-compact exceptional divisor. 
Then its image in $\mathbb{C}^3/G$ must be an irreducible toric curve. 
There are three of them on $\mathbb{C}^3/G$, corresponding to the three
coordinate axes of $\mathbb{C}^3$. Assume without loss of generality
that it's the $z$-axis. Let $p$ be any point on $E$ whose image
$\pi(p)$ in $\mathbb{C}^3/G$ is a point on the $z$-axis away from the
origin, i.e. $z^{|G|}(\pi(p)) \neq 0$. Since $\widetilde{\mathcal{M}}$
is a $\gnat$-family any $m \in \regring^G$ acts on 
$\widetilde{\mathcal{M}}_p$ by multiplication by $m(\pi(p))$. 
In particular, this is true for $G$-invariant monomial $z^{|G|}$. 
Since $z^{|G|}(\pi(p)) \neq 0$ it follows that 
$z \cdot s \neq 0$ for any non-zero section $s$ of $\widetilde{\mathcal{M}}_p$, 
i.e. no $z$-arrow vanishes along the whole of $E$ 
in the associated representation $Q(G)_{\widetilde{\mathcal{M}}}$. So 
every vertex $\chi \in Q(G)$ is either an $x$-tile, a $y$-tile, a
$z$-$(1,0)$-charge or a $z$-$(0,1)$-charge and thus 
$\sinksource_{\widetilde{\mathcal{M}},E}$ can only consist of looping $z$-charge
lines which do not intersect each other.

Start at any vertex in $T_H$ and move in the direction of $x$-arrows 
until we come full circle. By inspection, crossing any $z$-$(1,0)$-charge 
line we move from $y$-tiled region to $z$-tiled region and vice versa 
for crossing any $z$-$(0,1)$-charge line. Since we end up in the same
region we've started, there are as many $z$-$(0,1)$-charge lines 
as there are $z$-$(1,0)$-charge lines. By the first assertion of 
this lemma $\sinksource_{\widetilde{\mathcal{M}},E}$ is non-empty, 
hence there exists at least one $z$-$(1,0)$-charge line.

On the other hand, observe that any arrow which leaves 
a $z$-$(1,0)$-charge line must vanish along $E$. Therefore, if
you start at a vertex on a $z$-$(1,0)$-charge line the only
vertices you can reach by following only non-vanishing arrows 
are the other vertices on that charge line. But by the 
argument in \cite{CautisLogvinenko}, Prop.\ 4.12 there
must exist a path from every vertex of $Q(G)$ to $\chi_0$ which 
consists entirely of non-vanishing arrows. Therefore any 
$z$-$(1,0)$-charge line must contain $\chi_0$. 
Since the charge lines may not intersect we conclude 
that there exists at most one $z$-$(1,0)$-charge line. 
This settles the second assertion.
\end{proof}

\begin{figure}[!htb] \centering 
\includegraphics[scale=0.11]{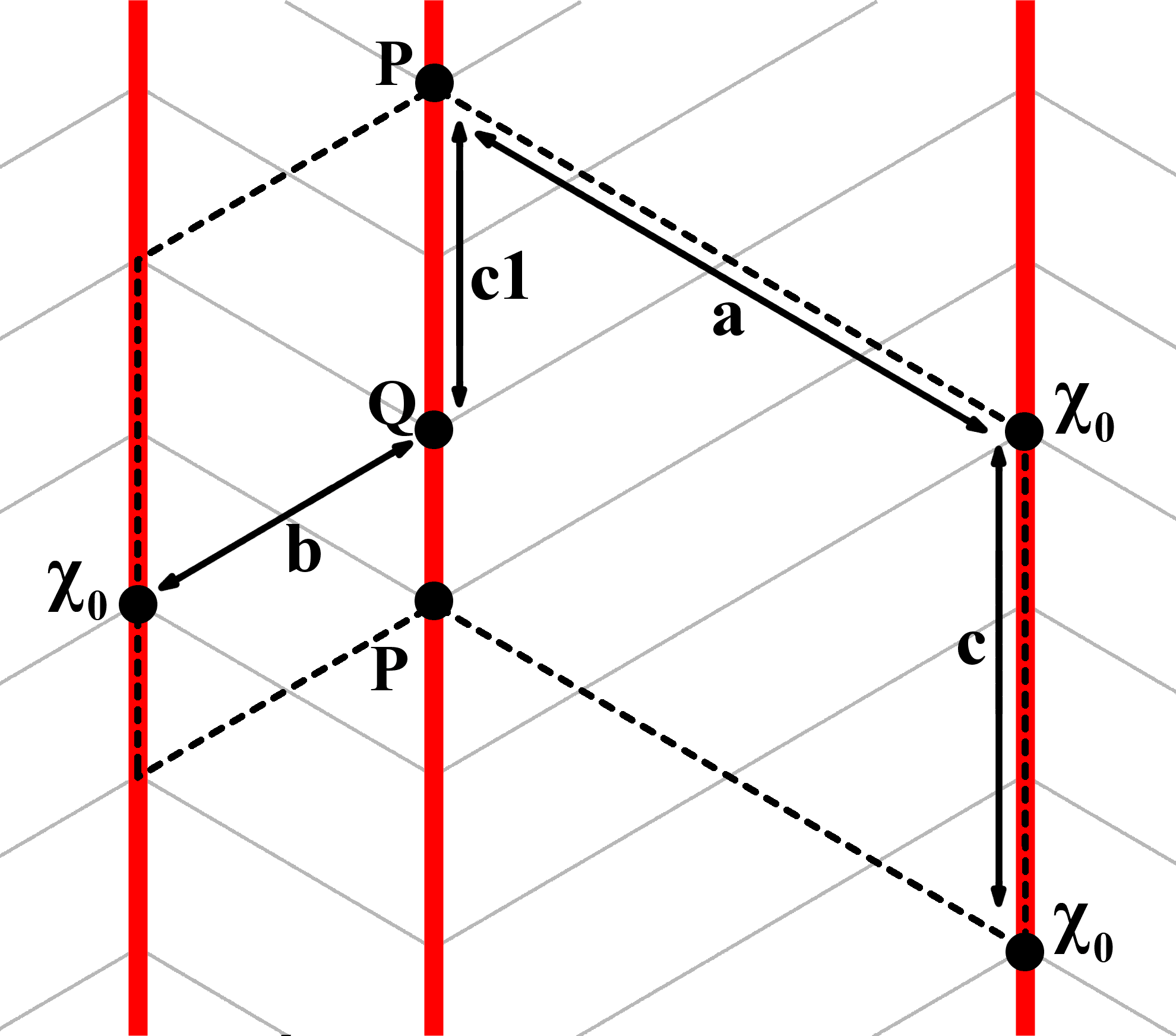}  
\caption{A single looping $(0,1)$-charge line sink-source graph}
\label{figure-1}
\end{figure}

\begin{figure}[!htb] \centering 
\subfigure[A valency 3 vertex] { \label{figure-2a}
\includegraphics[scale=0.29]{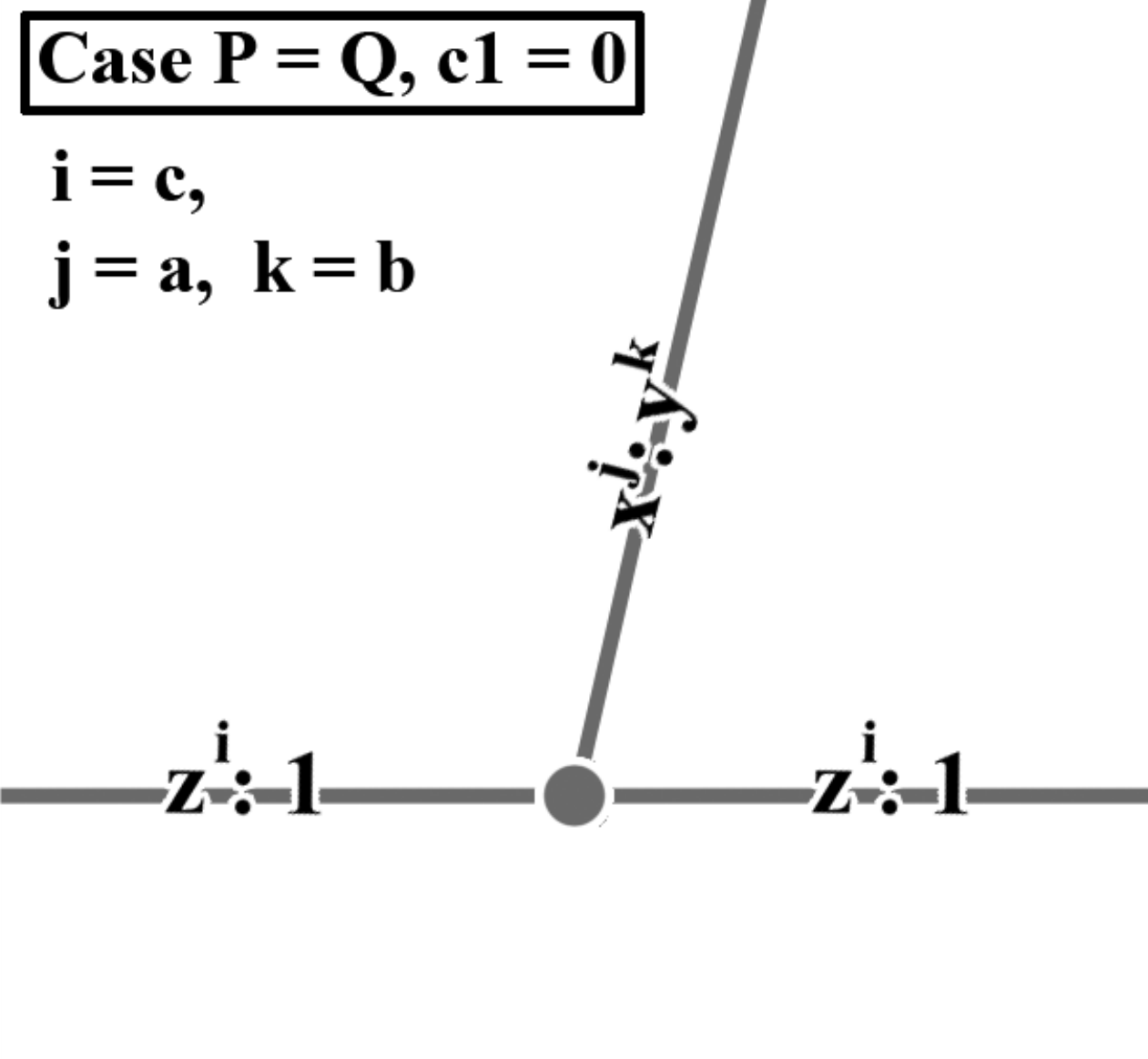}} 
\subfigure[A valency 4 vertex] { \label{figure-2b}
\includegraphics[scale=0.29]{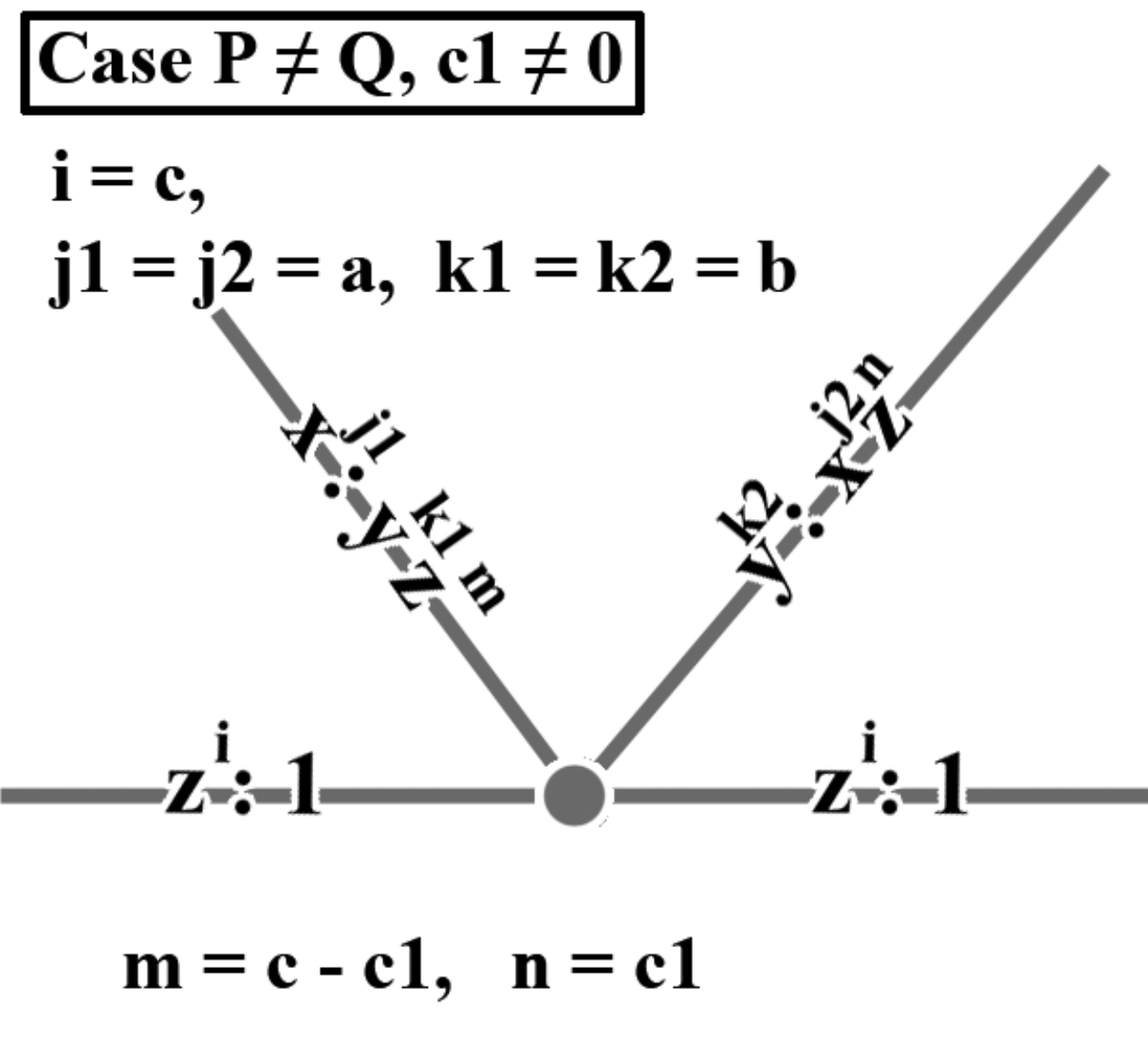}}
\caption{A side vertex of $\Delta$}
\label{figure-2}
\end{figure}

\begin{proposition}
\label{prps-one-looping-(0,1)-charge-line-to-non-compact-divisor}
Let $e \in \mathfrak{E}$ and let $E$ be the corresponding
toric divisor on $Y$. If the sink-source graph $SS_{\widetilde{\mathcal{M}}, E}$ 
is as depicted on Figure \ref{figure-1}, then the coordinates 
of $e$ in $L$ are $\frac{1}{|G|}(ac, bc, 0)$ and locally around 
$e$ triangulation $\Sigma$ looks as depicted on
Figure $\ref{figure-2}$. Moreover, the monomial ratios 
carving out the edges incident 
to $e$ can be computed in terms of the indicated lengths 
in $SS_{\widetilde{\mathcal{M}}, E}$ as shown on 
Figures $\ref{figure-2a}$-$\ref{figure-2b}$. 

If the shape of $SS_{\widetilde{\mathcal{M}}, E}$ is a rotation of Figure
\ref{figure-1} by $\frac{2\pi}{3}$ or $\frac{4\pi}{3}$ one permutes
$x$, $y$ and $z$ accordingly in all of the above.
\end{proposition}

To explain the notation of Figure \ref{figure-1}: let $P$
be the vertex where we first meet the looping $(1,0)$-charge line 
if we follow the line of $x$-arrows backwards from $\chi_0$.
Let $Q$ be same but for a line of $y$-arrows. Let $a$ and $b$
be the lengths (in arrows) of the resulting paths 
from $P$ and $Q$ to $\chi_0$, as indicated. Let $c$ be the number 
of distinct characters which occur on the $(0,1)$-charge line, i.e. 
the length of the path of $z$-arrows 
which starts at $\chi_0$ and ends when it first comes back 
to $\chi_0$.  Then the dotted line on Figure \ref{figure-1} 
gives a choice of a fundamental domain of $Q(G)$. Finally, 
$c_1$ is the length a path of $z$-arrows from $Q$ to $P$.
If $P$ and $Q$ coincide we set $c_1 = 0$. 

\begin{proof}
The proof proceeds very similarly to the proofs of Props. 3.1-3.3 of 
\cite{Logvinenko-ReidsRecipeAndDerivedCategories}, except we
work with $\sinksource_{\widetilde{\mathcal{M}}, E}$ instead of 
$\sinksource_{\mathcal{M}, E}$. Suppose there
is an edge incident to $e$ which is carved by a ratio of form 
$z^{k'} : x^{i'} y^{j'}$ for some $i', j', k' \geq 0$. Let $\chi$
denote the common character of $z^{k'}$ and $x^{i'} y^{j'}$. 
Any path which starts at $\chi^{-1}$ and consists of   
$k'$ $z$-arrows or of $i'$ $x$-arrows and $j'$
$y$-arrows terminates at $\chi_0$. By the dual of 
the argument which begins the proof of Prop.\ 3.1 of 
\cite{Logvinenko-ReidsRecipeAndDerivedCategories} any such 
path may not contain arrows that vanish along $E$. 
In particular, the $x$-arrow which leaves $\chi^{-1}$ cannot vanish 
along $E$ unless $i' = 0$ and the $y$-arrow which leaves $\chi^{-1}$ 
cannot vanish unless $j' = 0$. On the other hand, observe that
$\chi^{-1}$ must lie on the same looping $z$-$(1,0)$-charge line 
as $\chi_0$ since we can reach the latter by taking a path of 
$z$-arrows from the former. Therefore both the $x$-arrow and the $y$-arrow 
which leave $\chi^{-1}$ vanish along $E$. 
We conclude that $\chi = \chi_0$, $i' = j' = 0$ and $k' = c$. 
In other words, the ratio in question can only be $z^c \colon 1$. 

Suppose now there is an edge incident to $e$ which is carved by 
a ratio of the form $x^{i'} : y^{j'} z^{k'}$ for some $i', j', k' \geq 0$.
We may further assume that $i' \neq 0$ and $j' \neq 0$ as we've
already dealt with that case. Let $\chi$ denote the common 
character of $x^{i'}$ and $y^{j'} z^{k'}$. Then, arguing as above, 
no arrow in the path of $i'$ $x$-arrows from $\chi^{-1}$ to $\chi_0$
vanishes along $E$. Therefore $\chi^{-1}$ must lie somewhere on the 
path of $x$-arrows between $P$ and $\chi_0$ (see Figure
\ref{figure-1}). On the other hand, no arrow in any path of 
$j'$ $y$-arrows and $k'$ $z$-arrows vanishes along $E$. 
In particular, since $j' \neq 0$ the $y$-arrow which leaves $\chi^{-1}$ 
doesn't vanish. From Figure \ref{figure-1} we can see that 
the only possibilty is $\chi^{-1} = P$, $i' = a$, $j' = b$ and
$k' = (- c_1) \mod c$. So the ratio in question can only
be $x^a : y^b z^{(-c_1 \mod c)}$. A similar argument for a ratio 
of form $y^{j'} : x^{i'} z^{k'}$ shows that 
the only possibility is $\chi^{-1} = Q$ and the ratio $y^b : x^a z^{c_1}$. 

By Lemma \ref{lemma-number-of-diamonds} we have $e = \frac{1}{G}(a, b, c)$ 
where $a$, $b$ and $c$ are the numbers of $x$-, $y$- and $z$-oriented arrows 
which vanish along $E$ in $SS_{\widetilde{\mathcal{M}}, E}$. The choice
of a fundamental domain indicated by the dotted line    
on Figure \ref{figure-1} demonstrates that $e = \frac{1}{G}(ac,bc,0)$. 
 
Now suppose $P$ and $Q$ coincide, i.e. $c_1 = 0$. 
Then the above shows that the only ratios which can mark 
an edge incident to $e$ are
$z^c : 1$ and $x^a : y^b$, so the triangulation $\Sigma$
must look locally around $e$ as depicted on Figure $\ref{figure-2a}$. 
Suppose $P$ and $Q$ do not coincide, i.e. $c_1 \neq 0$. Then the only
ratios which can mark an edge incident to $e$ are
$z^c : 1$, $x^a : y^b z^{c - c_1}$ and $y^b : x^a z^{c_1}$ and  
so the triangulation $\Sigma$ must look locally around $e$ as 
depicted on Figure $\ref{figure-2b}$. 
\end{proof}

With Lemma \ref{lemma-divisors-to-sink-source-graphs} and 
Prop.\ \ref{prps-one-looping-(0,1)-charge-line-to-non-compact-divisor}
we obtain immediately a refined version of Theorem 3.1 of 
\cite{Logvinenko-ReidsRecipeAndDerivedCategories} which takes
into account every vertex in $\mathfrak{E}$ and not just those
which lie in the interior of $\Delta$ and correspond to compact 
exceptional divisors:
\begin{theorem}
\label{theorem-sink-source-graph-to-divisor-type-correspondence}
Let $e \in \mathfrak{E}$. Then one of the following must hold: 
\begin{footnotesize} 
\begin{align*}
\begin{array}{|c|c|c|c|c|} 
\hline
\sinksource_{\widetilde{\mathcal{M}}, E_e} & 
e &
\text{Triangulation } \Sigma &
\text{Reid's recipe} &
E_e \\
& &
\text{ locally around } e
&  
\text{(\cite{Craw-AnexplicitconstructionoftheMcKaycorrespondenceforAHilbC3}, \S3) } 
&
\\
\hline
\text{Empty} & 
\text{A corner } &
\text{A corner of $\Delta$} &
\text{ -- } &
\text{The strict transform of} 
\\
& \text{vertex of $\Delta$} 
& & &
\text{a coordinate hyperplane}
\\
\hline
\text{Fig.\ $\ref{figure-1}$} & 
\text{A side } & 
\text{Fig.\ $\ref{figure-2}(a)$-$(b)$} &
\text{ -- } &
\mathbb{P}^1 \times \mathbb{C}, \text{ blown-up}
\\
(\text{up to rotation}^{\dagger})
& 
\text{vertex of $\Delta$}
& 
(\text{up to rotation}^{\dagger})
& & 
\text{ in $0$ or $1$ points}
\\
\hline
\text{Fig.\ \ref{figure-20a}} & 
\text{An interior } & 
\text{Fig.\ \ref{figure-20b}} &
\text{Case $1$} & 
\mathbb{P}^2 \\ 
&
\text{vertex of $\Delta$}
& & &
\\
\hline
\text{Fig.\ \ref{figure-21a}} & 
\text{An interior } & 
\text{Fig.\ \ref{figure-21b} - \ref{figure-21e}} &
\text{Cases $2$-$3$} & 
\text{A surface scroll, blown-up} \\ 
(\text{up to rotation}^{\dagger})
& 
\text{vertex of $\Delta$}
& 
(\text{up to rotation}^{\dagger})
& & 
\text{in $0$, $1$ or $2$ points} 
\\
\hline
\text{Fig.\ \ref{figure-22a}} & 
\text{An interior } & 
\text{Fig.\ \ref{figure-22b}} &
\text{Case $4$} & 
\text{Del Pezzo surface $dP_6$} \\
& 
\text{vertex of $\Delta$}
& & & 
\\
\hline
\end{array} 
\end{align*} 
$\dagger$: The sink-source graph $SS_{\widetilde{\mathcal{M}}, E}$ may also be 
a rotation of the diagram by an angle of $\frac{2 \pi}{3}$ or $\frac{4 \pi}{3}$. 
In this case the corresponding diagram of triangulation $\Sigma$ 
locally around $e$ should be rotated by the same angle and 
$x$, $y$ and $z$ should be permuted.
\end{footnotesize} 
\end{theorem}

\section{CT-subdivisions}
\label{section-CT-subdivisions} 

Fix a character $\chi \in G^\vee$. Section  
\ref{section-psi-O-chi-and-skew-commutative-cubes-of-line-bundle}
describes the transform $\Psi(\mathcal{O}_0 \otimes \chi)$ 
in terms of the vanishing divisors of the skew-commutative cube 
of line bundles $\hex(\chi^{-1})_{\widetilde{\mathcal{M}}}$.  
Each of these vanishing divisors is a reduced union 
of irreducible toric divisors on $Y$ \cite[\S4.2]{CautisLogvinenko}. 
For an irreducible toric divisor $E$, we may ask which maps in 
the cube corresponding to $\hex(\chi^{-1})_{\widetilde{\mathcal{M}}}$ 
vanish along $E$. The answer is encoded in the vertex type of $\chi^{-1}$ 
in $\sinksource_{\widetilde{\mathcal{M}}, E}$. 
The table that translates between the two is provided on 
Fig.\ \ref{figure-13}. Therefore, to compute all the vanishing divisors 
of $\hex(\chi^{-1})_{\widetilde{\mathcal{M}}}$ it suffices
to know the vertex type of $\chi^{-1}$ 
for each toric divisor on $Y$. This turns out 
to be directly linked to the following notion introduced in 
\cite[\S5]{Craw-AnexplicitconstructionoftheMcKaycorrespondenceforAHilbC3}:
\begin{defn}
Given a monomial $m \in \mathbb{C}[x,y,z]$ define $\convhull(m)$ to be 
the union of all the basic triangles in $\Sigma$ whose $G$-graph
contains $m$. 
\end{defn}

The union of the nonempty $\convhull(m)$ as $m$ ranges over all the
monomials of weight $\chi$ provides a subdivision of $\Delta$. 
We need the following coarsening of it. Let $Cz^\bullet$ be the union 
of the basic triangles in $\Sigma$ in whose $G$-graph $\chi$ is 
represented by $z^c$ for $c > 0$. Similarly for 
$Cx^\bullet$ and $Cy^\bullet$. Let $Tx^\bullet y^\bullet$ 
be the union of all the basic triangles in $\Sigma$ 
in whose $G$-graph $\chi$ is represented by $x^a y^b$ for $a,b > 0$.
Similarly for $Ty^\bullet z^\bullet$ and $Tx^\bullet z^\bullet$. 
Together these six give a subdivision of $\Delta$ which we call 
the \em CT-subdivision\rm.  

It was proved in 
\cite[Lemma 5.3]{Craw-AnexplicitconstructionoftheMcKaycorrespondenceforAHilbC3}
that when non-empty $\convhull(m)$ is a convex region of
$\Delta$ which contains $e_x$ (resp. $e_y$, $e_z$) if and only if 
$x$ (resp. $y$, $z$) doesn't divide $m$. So if
non-empty:
\begin{itemize}
\item $Cx^\bullet$ is a convex area containing 
side $e_y e_z$ of $\Delta$. 
\item $Cy^\bullet$ is a convex area containing 
side $e_x e_z$ of $\Delta$. 
\item $Cz^\bullet$ is a convex area containing 
side $e_x e_y$ of $\Delta$. 
\item $Ty^\bullet z^\bullet$ is a union of convex areas 
each containing corner vertex $e_x$ of $\Delta$. 
\item $Tx^\bullet z^\bullet$ is a union of convex areas
each containing corner vertex $e_y$ of $\Delta$. 
\item $Tx^\bullet y^\bullet$ is a union of convex areas
each containing corner vertex $e_z$ of $\Delta$. 
\end{itemize}

\begin{figure}[!htb] 
\begin{center}
\includegraphics[scale=0.11]{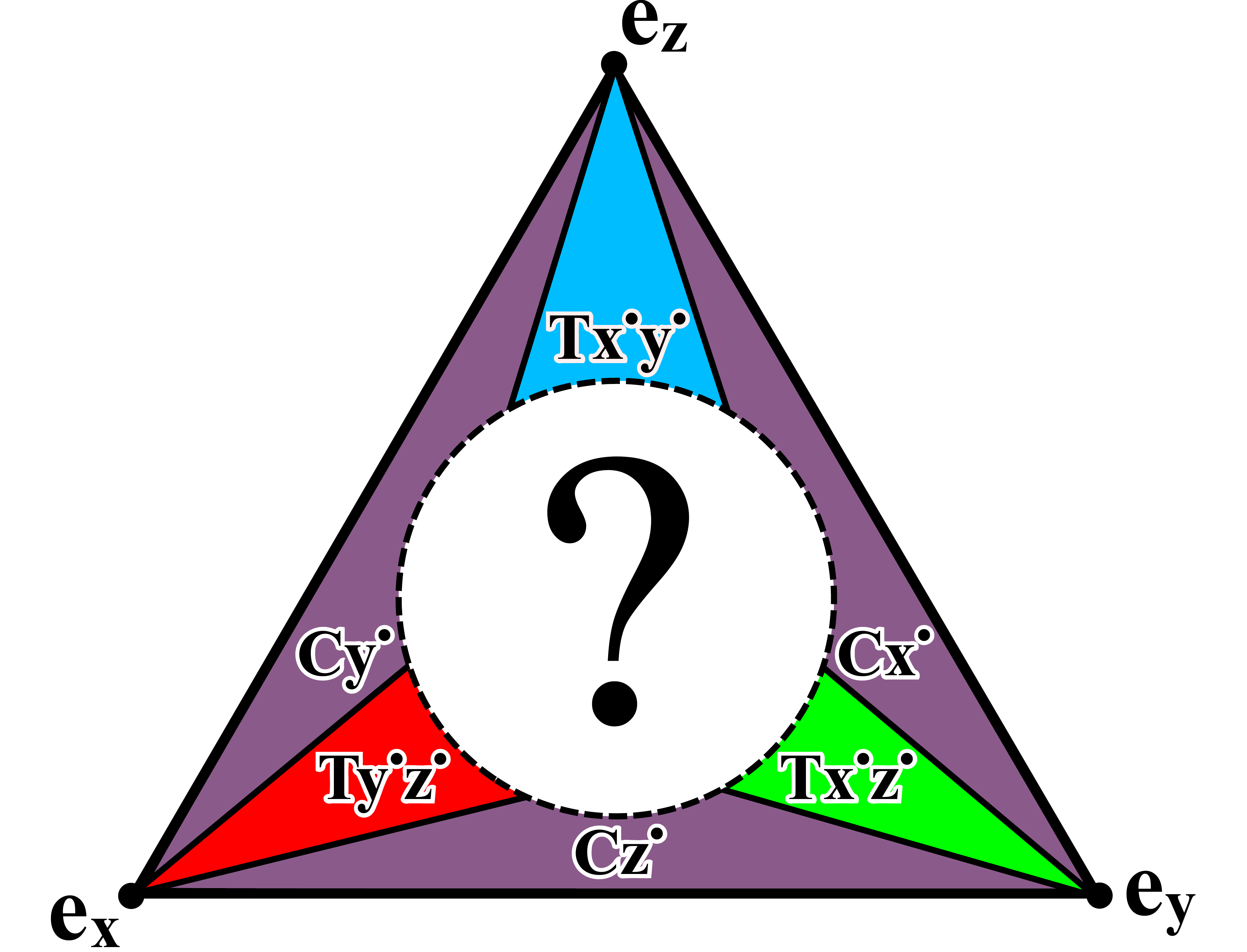} \end{center}
\caption{\label{figure-28} CT-subdivision of $\Delta$ 
 }
\end{figure}

Thus a CT-subdivision looks, in general, as depicted on 
Figure \ref{figure-28}. However any of the six areas can be empty. 
We want to think of such cases as degenerations of the general picture
above. When $Cx^\bullet$, $Cy^\bullet$ or $Cz^\bullet$ is empty 
we think of them as having degenerated away completely. But
for $Ty^\bullet z^\bullet$, $Tx^\bullet z^\bullet$ or 
$Ty^\bullet z^\bullet$ we distinguish two cases.  
When $Tx^\bullet y^\bullet$ is empty, the areas $Cx^\bullet$ 
and $Cy^\bullet$ share a boundary that is necessarily 
a line segment $l$ out of $e_z$. If $\chi$ doesn't mark $l$ we think 
of $Tx^\bullet y^\bullet$ as having degenerated into an 
infinitesimally thin strip along $l$. In other words, 
we do not consider the points of $l$ as lying on 
the boundary of $Cx^\bullet$ and $Cy^\bullet$ but rather as lying on 
the degeneration of $Tx^\bullet y^\bullet$ still wedged in 
between the two. If, on the other hand, $\chi$ does mark $l$ 
we think of $Tx^\bullet y^\bullet$ as having degenerated to 
just the vertex $e_z$, allowing $Cx^\bullet$ and $Cy^\bullet$ 
to share a boundary along $l$.

We formalise this as follows: 

\begin{defn}
\begin{enumerate}
\item Define $Cz^\bullet$ to be the union of all the basic triangles
in $\Sigma$ in whose $G$-graph $\chi$ is represented by 
$z^c$ for $c > 0$. Define $Cx^\bullet$ and $Cy^\bullet$ similarly. 
\item Define $Tx^\bullet y^\bullet$ as follows: 
\begin{itemize}
\item If non-empty, define $Tx^\bullet y^\bullet$ to be 
the union of the basic triangles in $\Sigma$ in whose 
$G$-graph $\chi$ is represented by $x^a y^b$ for $a,b > 0$. 
\item Otherwise, we consider the boundary $l$ between
$Cx^\bullet$ (or $e_z e_y$ if $Cx^\bullet$ empty) and $Cy^\bullet$
(or $e_z e_x$ if $Cy^\bullet$ empty). 
\begin{itemize}
\item If $\chi$ doesn't mark $l$, define $Tx^\bullet y^\bullet$
to be $l$. 
\item If $\chi$ does mark $l$, define $Tx^\bullet y^\bullet$ to be $e_z$. 
\end{itemize}
\end{itemize}
Define $Ty^\bullet z^\bullet$ and $Tx^\bullet z^\bullet$ similarly. 
\end{enumerate}
\end{defn}

Our reason for adopting these conventions becomes clear in the
course of proving the following:
\begin{lemma}
\label{lemma-non-vanishing-paths-to-CT-areas}
Let $e$ be a vertex of $\Sigma$ and $E$ be the corresponding
toric divisor on $Y$. Then an area of the CT-subdivision for $\chi$ contains
$e$ if and only if $\mckquiv$ has a path from $\chi^{-1}$ to $\chi_0$ 
which doesn't vanish along $E$ in $\mckquiv_{\widetilde{M}}$ and 
which represents monomial of the same type as this area. 
E.g. there is a non-vanishing $x^i y^j$ path of arrows with $i,j > 0$ 
if and only if $e$ belongs to $Tx^\bullet y^\bullet$. 
\end{lemma}
\begin{proof}
\em ``Only If'' direction: \rm 

Suppose $e$ belongs to a C-area or a non-degenerate T-area
of the CT-subdivision for $\chi$. Then $m$ represents $\chi$
in the $G$-graph of some basic triangle $\tau$ which contains $e$.
A monomial $m$ represents $\chi$ in the $G$-graph of $\tau$ if and 
only if $m . 1 \neq 0$ in the $G$-cluster parametrised by 
the toric fixed point $p_\tau$ of $\tau$. Equivalently, 
all the paths from $0$ to $\chi$ which correspond to $m$ 
must not vanishing at $p_\tau$ in the associated 
representation $\mckquiv_{M}$. Dually, 
all the paths from $\chi^{-1}$ to $0$ which correspond to $m$ 
must not vanish at $p_\tau$ in $\mckquiv_{\widetilde{M}}$. 
Finally, since $E$ contains $p_\tau$, if a path doesn't vanish at $p_\tau$ 
it certainly doesn't vanish along $E$. 

It remains to deal with degenerate $T$-areas.  
Suppose, without loss of generality, that $e$ belongs to 
a degenerate $Tx^\bullet y^\bullet$-area. Then either $e = e_z$, 
or one of $Cx^\bullet$ and $Cy^\bullet$ is empty and $e$
lies on the corresponding side of $\Delta$, or 
$e$ lies on the border of $Cx^\bullet$ and $Cy^\bullet$ and 
$\chi$ doesn't mark this border. If $e = e_z$, then every path 
of $x^\bullet y^\bullet$-type is non-vanishing in 
$\sinksource_{\widetilde{\mathcal{M}}, E}$. If  $e$ lies on a side 
of $\Delta$, the existence of the desired non-vanishing path
can be readily verified on (appropriately rotated) 
Fig.\ \ref{figure-1}-\ref{figure-2}. 
Finally, let $e$ lie on the border of $Cx^\bullet$ and $Cy^\bullet$. 
This border is a line $l$ carved out by $x^a : y^b$ and $\chi$ 
is represented by $x^{ka}$ and $y^{kb}$ in the $G$-graphs to $e_y$-side 
and to $e_x$-side of $l$. In particular, $x^{ka}$ and $y^{kb}$ define 
non-vanishing paths from $\chi^{-1}$ to $\chi_0$ in 
$\sinksource_{\widetilde{\mathcal{M}}, E}$. 
Since $\chi$ doesn't mark $l$ we have $k > 1$,
and any one of $x^{(k-1)a}y^b, \dots, x^a y^{(k-1)b}$
defines the desired non-vanishing path of $x^\bullet y^\bullet$-type. 

\em ``If'' direction: \rm 

A non-vanishing path from $\chi^{-1}$ to
$\chi_0$ must be contained completely within one (or more) of
the tiled regions $\sinksource_{\widetilde{\mathcal{M}}, E}$ subdivides
$T_H$ into. E.g. a non-vanishing $x^\bullet y^\bullet$ path has to be 
contained within the $z$-tiled region. 
The possible sink-source graph shapes listed in Theorem 
\ref{theorem-sink-source-graph-to-divisor-type-correspondence}
give us a total of twelve tiled regions to consider. 
Five of these are contractible and we deal with them first. 
Within a contractible region of $T_H$ any two paths 
between the same pair of vertices of $Q(G)$ lift to 
the same monomial modulo $xyz$. Therefore within 
a contractible tiled region of $T_H$ all 
non-vanishing paths from a given vertex to $\chi_0$ 
correspond to the same monomial. It is then easy 
to verify that this monomial does indeed occur 
in the $G$-graph of at least one of the basic triangles 
containing $e$.
 
For example, if $\sinksource_{\widetilde{\mathcal{M}}, E}$ is 
as depicted on Figure \ref{figure-22a} then the non-vanishing
paths from the vertices in the $z$-tiled region to $\chi_0$
yield the monomials 
\begin{align}
\label{eqn-non-vanishing-paths-in-z-tiled-region-of-dP6} 
\left\{x^i y^j \;|\; 0 \leq i \leq a, 0 \leq j \leq b_3 \text{ or }
0 \leq i \leq a_2, 0 \leq j \leq b \right\}. 
\end{align}
On the other hand, we see on Figure \ref{figure-22b} that one of the 
six triangles containing $e$ has $\frac{y^{b_2}z^{c_3}}{x^a}$, 
$\frac{x^{a_3}z^{c_2}}{y^b}$ and
$\frac{x^{a_2 + 1}y^{b_3 + 1}}{z^{c-1}}$ as the coordinates of 
its affine chart. Its $G$-graph is then readily seen
to contain all the monomials in 
\eqref{eqn-non-vanishing-paths-in-z-tiled-region-of-dP6}. 
The other contractible regions are dealt with similarly. 

We now proceed to deal with non-contractible tiled regions. 

Suppose, that $\sinksource_{\widetilde{\mathcal{M}}, E}$ is 
as depicted on Figure \ref{figure-21a} with no rotations
necessary. Then it has a single non-contractible region:
the $z$-tiled one. Suppose that $\chi^{-1}$ is
some vertex within this region or on its boundary. 
Let $x^\alpha y^{\bar{\beta}}$ be the monomial of the path 
which goes along $x$-arrows from $\chi^{-1}$ until 
it hits $y$-$(1,0)$-charge line $I_1 O_1$ 
and then follows the $y$-arrows of this charge line to $\chi_0$.
Similarly let $x^{\bar{\alpha}} y^{\beta}$ be the monomial of 
the path along the $y$-arrows and then $x$-$(1,0)$-charge line $I_1 O_1$. 
Note that $0 \leq \bar{\alpha} < a$ and $0 \leq \bar{\beta} < b$. 
Next, note that the monomials of 
any two paths from $\chi^{-1}$ to $\chi_0$ 
must differ by a power of $\frac{x^a}{y^b}$. It follows 
that $\alpha = a k + \bar{\alpha}$ and $\beta = b k + \bar{\beta}$
for some $k \geq 0$. Hence 
the non-vanishing paths from $\chi^{-1}$ to $\chi_0$ 
correspond to the following monomials:
\begin{align}
\label{eqn-monomials-corresponding-to-non-vanishin-paths-in-a-non-contractible-region} 
 x^{\alpha} y^{\bar{\beta}}, 
   x^{\alpha - a} y^{\bar{\beta} + b}, 
   \dots, 
   x^{\bar{\alpha} + a} y^{\beta - b}, 
   x^{\bar{\alpha}} y^\beta. 
\end{align}

Now note that in all the triangles to the $e_y$ side of 
line $x^a \colon y^b$ the ratio $\frac{y^b}{x^a}$ is 
a regular function. Every monomial in 
\eqref{eqn-monomials-corresponding-to-non-vanishin-paths-in-a-non-contractible-region}
is a product of $x^{\alpha} y^{\bar{\beta}}$ and a power of
$\frac{y^b}{x^a}$, so only $x^{\alpha} y^{\bar{\beta}}$ can 
appear in the $G$-graphs of the triangles to the $e_y$ side 
of line $x^a \colon y^b$. Similarly, only $x^{\bar{\alpha}} 
y^{\beta}$ can appear in the $G$-graphs of the triangles to the $e_x$ side
of line $x^a \colon y^b$. On the other hand, observe that one
of the triangles around $e$ has
$\frac{y^b}{x^a}$ and $\frac{z^c}{x^{a_1} y^{b_1 \mod b}}$
as two of its coordinates. 
Hence $x^{a_1}y^{b - 1}$ is in its $G$-graph. 
Since $0 \leq \alpha \leq a_1$ and $0 \leq \bar{\beta} \leq b - 1$
(see Figure \ref{figure-21a})
it follows that $x^{\alpha} y^{\bar{\beta}}$ must also be in its
$G$-graph. Similarly, the basic triangle with
$\frac{x^a}{y^b}$ and $\frac{z^c}{x^{a_1 \mod a} y^{b_1}}$ 
as coordinates has 
$x^{\alpha} y^{\bar{\beta}}$ in its $G$-graph. 
Thus, of the monomials in 
\eqref{eqn-monomials-corresponding-to-non-vanishin-paths-in-a-non-contractible-region}
only $x^{\alpha} y^{\bar{\beta}}$ and $x^{\bar{\alpha}}
y^{\beta}$ do appear the $G$-graphs around $e$.  
The remaining ones are the ``ghost'' monomials, 
the ones which appear as non-vanishing paths
in $\sinksource_{\widetilde{\mathcal{M}}, E}$, but not 
in the $G$-graphs around $e$. 

Observe, however, that all the ``ghost'' monomials in
\eqref{eqn-monomials-corresponding-to-non-vanishin-paths-in-a-non-contractible-region}
are of $x^\bullet y^\bullet$ type. So a problem 
arises only when neither 
$x^{\alpha} y^{\bar{\beta}}$ nor $x^{\bar{\alpha}} y^{\beta}$
are of $x^\bullet y^\bullet$ type, i.e. when $\bar{\alpha} = \bar{\beta} =
0$. It is here that our peculiar choice of conventions for
degeneration of $Ty^\bullet z^\bullet$, $Tx^\bullet z^\bullet$
and $Tx^\bullet y^\bullet$ comes into play. 
Suppose $\bar{\alpha} = \bar{\beta} = 0$. The monomials
of the non-vanishing paths from $\chi^{-1}$ to $\chi_0$ are then 
$$ x^{ka}, x^{(k-1)a} y, \dots, xy^{(k-1)b}, y^{kb} $$
for some $k \geq 1$.
By above, in the $G$-graphs on $e_x$ side of line $x^a \colon y^b$
character $\chi$ is represented by $y^{ka}$ and in the $G$-graphs on $e_y$ 
side -- by$x^{kb}$. So line $x^a \colon y^b$ is the border between 
$Cx^\bullet$ and $Cy^\bullet$ in the CT-subdivision for $\chi$, 
and $Tx^\bullet y^\bullet$, which is usually wedged in between the two, 
is empty. If $k > 1$ then $\chi$ doesn't actually mark line $x^a \colon y^b$. 
By our conventions this means that $Tx^\bullet y^\bullet$ 
degenerates to an infinitesimally thin strip still wedged in 
between $Cx^\bullet$ and $Cy^\bullet$. Vertex $e$ lies on 
$x^a \colon y^b$, and therefore belongs to this degeneraton  
of $Tx^\bullet y^\bullet$. This accounts for the ``ghost'' 
monomials $x^{(k-1)a}y, \dots, x y^{(k-1)b}$. 
On the other hand, when $k = 1$ the non-vanishing paths from
$\chi^{-1}$ are given by $x^a$ and $y^b$, both of which appear
in $G$-graphs around $e$, i.e. there are no ``ghost'' monomials.
This matches the fact that $\chi$ does mark the line $x^a \colon y^b$ 
through $e$, which by our conventions means that 
$Tx^\bullet y^\bullet$ degenerates away from $e$ to just 
$e_z$. 

The non-contractible regions which occur when
$\sinksource_{\widetilde{\mathcal{M}}, E}$ is 
as depicted on Fig.\ \ref{figure-20a}, as depicted on 
Fig.\ \ref{figure-1} or empty are dealt with similarly.
\end{proof}

\begin{proposition}
\label{prps-vertex-type-to-CT-subdivision-role-correspondence}
Let $\chi$ be a non-trivial character of $G$. Let $E$ be a toric divisor 
on $Y$ and $e$ be the corresponding vertex of $\Sigma$. 
The vertex type of $\chi^{-1}$ in $\sinksource_{\widetilde{\mathcal{M}}, E}$ and
the role of $e$ in the CT-subdivision for $\chi$ are related
in the following way:

%\begin{footnotesize} 
\begin{align}
\label{table-vertex-type-to-CT-subdivision-role-correspondence}
\begin{array}{|m{0.35in} m{1in}|m{3.5in}|} 
\hline
\multicolumn{2}{|c|}{\text{Vertex type of $\chi^{-1}$}} & 
\text{The role of $e$ in the CT-subdivision of $\Delta$:}
\\
\multicolumn{2}{|c|}{\text{in $\sinksource_{\widetilde{\mathcal{M}},E}$}} & 
\text{\begin{footnotesize}($e$ belongs to the following areas and none
other)\end{footnotesize}}
\\
\hline
\includegraphics[scale=0.04, trim=0 0 0 -10]{pics/figure-17e} & \text{$z$-$(1,0)$-charge} &
\text{ An internal vertex of $Cz^\bullet$ }
\\
\hline
\multirow{2}{*}{\includegraphics[scale=0.04, trim=0 0 0 -10]{pics/figure-18c}} &
\multirow{2}{*}{\text{$z$-tile}}
&
\text{ $\bullet$ An internal vertex of $Tx^\bullet y^\bullet$ }
\\
& & \text{ $\bullet$ A vertex on the border of 
$Tx^\bullet y^\bullet$ with $Cx^\bullet$ or $Cy^\bullet$ }
\\
\hline
\includegraphics[scale=0.04, trim=0 0 0 -10]{pics/figure-19f} &
\text{$z$-$(2,1)$-source} &
\text{ A vertex on the border of $Cx^\bullet$ and $Cy^\bullet$ but not $Cz^\bullet$ }
\\
\hline
\multirow{2}{*}{\includegraphics[scale=0.04, trim=0 0 0
-10]{pics/figure-19e}} 
& 
\multirow{2}{*}{\text{$z$-$(1,2)$-source}} &
\text{ A vertex on the border of $Cz^\bullet$ and $Tx^\bullet y^\bullet$ }
\\ 
& & 
\text{\begin{footnotesize} (plus, possibly, $Cx^\bullet$ or
$Cy^\bullet$ or both) \end{footnotesize} }
\\
\hline
\multirow{2}{*}{\includegraphics[scale=0.04, trim=0 0 0
-10]{pics/figure-17f}} 
& 
\multirow{2}{*}{\text{$z$-$(0,1)$-charge}} &
\text{ A vertex on the border of $Tx^\bullet z^\bullet$ and 
$Ty^\bullet z^\bullet$ }
\\
& & 
\text{\begin{footnotesize} (plus, possibly, one or two of $Cx^\bullet$,
$Cy^\bullet$ or $Cz^\bullet$) \end{footnotesize} }
\\
\hline
\multicolumn{3}{|c|}{\text{ (similarly for the $x$- and $y$-
vertex types)}}
\\ 
\hline
\multirow{2}{*}{\includegraphics[scale=0.04, trim=0 0 0
-10]{pics/figure-16b}} 
& 
\multirow{2}{*}{\text{$(0,3)$-sink}} &
\text{ A vertex on the border of $Ty^\bullet z^\bullet$, 
$Tx^\bullet z^\bullet$ and $Tx^\bullet z^\bullet$
}
\\ 
& & 
\text{\begin{footnotesize} (plus, possibly, one, two or three
of $Cx^\bullet$, $Cy^\bullet$ or $Cz^\bullet$) \end{footnotesize} }
\\
\hline
\includegraphics[scale=0.04, trim=0 0 0 -10]{pics/figure-19g} &
\text{$(3,3)$-source} &
\text{ A vertex on the border of $Cx^\bullet$, $Cy^\bullet$
and $Cz^\bullet$ }
\\
\hline
\end{array} 
\end{align} 
%\end{footnotesize} 
\end{proposition}
\begin{proof}
The proof works by directly verifying the data of
Table \eqref{table-vertex-type-to-CT-subdivision-role-correspondence}
for every possible shape of the sink-source graph 
$\sinksource_{\widetilde{\mathcal{M}}, E}$ as listed in Theorem 
\ref{theorem-sink-source-graph-to-divisor-type-correspondence}. 
$\sinksource_{\widetilde{\mathcal{M}}, E}$ subdivides $T_H$ into
three tiled regions. Within each region we locate all 
non-vanishing paths from $\chi^{-1}$ to $\chi_0$. 
By Lemma \ref{lemma-non-vanishing-paths-to-CT-areas} 
these determine all the areas of the CT-subdivision for $\chi$ 
which $e$ belongs to. 

Suppose, for example, $\chi^{-1}$ is a $z$-$(1,0)$-charge.
Then it lies on the border of the $x$-tiled and 
the $y$-tiled regions. So we look in each region for 
non-vanishing paths from $\chi^{-1}$ to $\chi_0$. 
Inspecting the possible sink-source graph shapes listed in Theorem 
\ref{theorem-sink-source-graph-to-divisor-type-correspondence}
we see in each of the two regions a unique non-vanishing path:
the same path which follows the $z$-$(1,0)$-charge 
line from $\chi^{-1}$ to the $(3,0)$-sink at $\chi_0$. 
Hence both the regions contribute the same monomial of type 
$z^\bullet$ and so $e$ has to be an internal vertex of 
$Cz^\bullet$. Suppose, on the other hand, that
$\chi^{-1}$ is $z$-tile. Then it lies in the interior
of the $z$-tiled region. Inspecting all the possible 
sink-source graph shapes we see that from any $z$-tile 
there is always a non-vanishing path to $\chi_0$ of 
$x^\bullet y^\bullet$ type, and sometimes there is
also a path of $x^\bullet$-type, or of $y^\bullet$-type, or both. 
We conclude that $\chi^{-1}$ is either an internal vertex
of $Tx^\bullet y^\bullet$ or it lies on the border of 
$Tx^\bullet y^\bullet$ with $Cx^\bullet$, or $Cy^\bullet$
or both. Now suppose $\chi^{-1}$ is a $z$-$(0,1)$-charge. 
Then it lies on the border of $x$-tiled and $y$-tiled
regions. Inspecting the four possible sink-source graph shapes 
we see the $x$-tiled region will always contain a non-vanishing 
path of type $y^\bullet z^\bullet$: take the path which goes 
along $y$-arrows until it hits $z$-$(1,0)$-charge line, and 
then follows that charge line back to $\chi_0$, possibly 
doing a single loop along it if $\sinksource_{\widetilde{\mathcal{M}}, E}$ 
is as depicted on Fig.\ \ref{figure-1}. The $x$-tiled region 
can also contain non-vanishing paths of type $y^\bullet$ and 
$z^\bullet$, but we note that it can only contain both
if $\sinksource_{\widetilde{\mathcal{M}}, E}$ is as depicted on 
Fig.\ \ref{figure-1} or the rotation of Fig.\ \ref{figure-21}
by $\frac{2\pi}{3}$ clockwise. We further note that in both
these cases there doesn't also exist
a non-vanishing path of type $x^\bullet$ from 
$\chi^{-1}$ to $\chi_0$. Repeating the same argument for 
the $y$-tiled region, we conclude that $\chi^{-1}$ lies on the border 
of $Ty^\bullet z^\bullet$, $Tx^\bullet z^\bullet$ and, possibly,
one or two of the $C$-areas. The remaining vertex types are dealt 
with similarly. 
\end{proof}

We now prove several corollaries of 
Prop.\ \ref{prps-vertex-type-to-CT-subdivision-role-correspondence} 
which relate the geometry of $CT$-subdivisions
to the calculation of the cohomology sheaves
of the total complex $T^\bullet$ of 
the skew-commutative cube
of line bundles $\hex(\chi^{-1})_{\widetilde{\mathcal{M}}}$. 
Denote by $D^i$, $D^i_j$ and $D^i_{jk}$ the vanishing divisors
of the maps in the cube, as per
\S\ref{section-psi-O-chi-and-skew-commutative-cubes-of-line-bundle}. 
By Lemma \ref{lemma-cube_cohomology} we have
$\mathcal{H}^0(T^\bullet) = \mathcal{L}^{-1}_\chi
\otimes \mathcal{O}_{D^1 \cap D^2 \cap D^3} $
and correspondingly:
\begin{corollary}
\label{cor-D1-D2-D3-via-CT-subdivision}
Let $\chi \in G^\vee$ be non-trivial, let $E$ be a toric divisor
on $Y$ and let $e$ be the corresponding vertex of $\mathfrak{E}$. 
Then:
\begin{itemize}
\item  $E \in D^1$ if and only if $e \in Ty^\bullet z^\bullet$.
\item  $E \in D^2$ if and only if $e \in Tx^\bullet z^\bullet$.
\item  $E \in D^3$ if and only if $e \in Tx^\bullet y^\bullet$.
\end{itemize}
In particular, $\mathcal{H}^{0}(T^\bullet) \neq 0$ if and only 
if there exists a basic triangle in $\Sigma$ which touches 
each of $Ty^\bullet z^\bullet$, $Tx^\bullet z^\bullet$ and 
$Tx^\bullet y^\bullet$. 
\end{corollary}
\begin{proof}
By inspection of Table 
\eqref{table-vertex-type-to-CT-subdivision-role-correspondence}
and Fig.\ \ref{figure-28}. 
\end{proof}

On the other hand, by Lemma \ref{lemma-cube_cohomology} 
sheaf $\mathcal{H}^{-1}(T^\bullet)$ admits a three-step filtration
whose successive quotients are supported on:
\begin{itemize}
\item intersection of $\gcd(D_2^3,D_3^2)$ and the effective part of 
$D^1 + \lcm(D_1^2, D_1^3) - \tD_2^3 - D^2$.
\item intersection 
of $\gcd(D_1^3,D_3^1)$ and the effective part of $D^2 +
\lcm(D_2^1,D_2^3) - \tD_3^1 - D^3$ 
\item intersection of $\gcd(D_1^2,D_2^1)$ and the effective part of 
$D^3 + \lcm(D_3^1,D_3^2) - \tD_1^2 - D^1.$ 
\end{itemize}

Correspondingly:
\begin{corollary}
\label{cor-gcd-D_2^3-D_3^2-etc-via-CT-subdivisions}
Let $\chi \in G^\vee$ be non-trivial, let $E$ be a toric divisor
on $Y$ and let $e$ be the corresponding vertex of $\mathfrak{E}$. 
Then:
\begin{itemize}
\item  $E \in \gcd(D_2^3,D_3^2)$ if and only if $e \in 
Cx^\bullet \setminus \left(Tx^\bullet z^\bullet \cup Tx^\bullet y^\bullet \right)$.
\item  $E \in \gcd(D_1^3,D_3^1)$ if and only if $e \in Cy^\bullet 
\setminus \left(Ty^\bullet z^\bullet \cup Tx^\bullet y^\bullet \right)$.
\item  $E \in \gcd(D_1^2,D_2^1)$ if and only if $e \in Cz^\bullet
\setminus \left(Ty^\bullet z^\bullet \cup Tx^\bullet z^\bullet\right)$.
\end{itemize}
\end{corollary}
\begin{proof}
By inspection of Table 
\eqref{table-vertex-type-to-CT-subdivision-role-correspondence}
and Fig.\ \ref{figure-28}.
\end{proof}

\begin{corollary}
\label{cor-Z_23-Z_13-Z_12-via-CT-subdivisions}
Let $\chi \in G^\vee$ be non-trivial, let $E$ be a toric divisor
on $Y$ and let $e$ be the corresponding vertex of $\mathfrak{E}$. 
Then  
\begin{itemize}
\item $E$ lies in the effective part of 
$D^1 + \lcm(D_1^2, D_1^3) - \tD_2^3 - D^2$
if and only if 
$$e \in \left(Cy^\bullet \cup Ty^\bullet z^\bullet \cup
Cz^\bullet\right) \setminus \left(Tx^\bullet z^\bullet \cup Tx^\bullet
y^\bullet\right).$$
\item $E$ lies in the effective part of 
$D^2 + \lcm(D_2^1,D_2^3) - \tD_3^1 - D^3$
if and only if 
$$e \in \left(Cx^\bullet \cup Tx^\bullet z^\bullet \cup
Cz^\bullet\right) \setminus \left(Ty^\bullet z^\bullet \cup Tx^\bullet
y^\bullet\right).$$
\item $E$ lies in the effective part of 
$D^3 + \lcm(D_3^1,D_3^2) - \tD_1^2 - D^1$ if and only 
if
$$e \in \left(Cx^\bullet \cup Tx^\bullet y^\bullet \cup
Cy^\bullet\right) \setminus \left(Ty^\bullet z^\bullet \cup Tx^\bullet
z^\bullet\right).$$
\end{itemize} 
\end{corollary}
\begin{proof}
By inspection of Table 
\eqref{table-vertex-type-to-CT-subdivision-role-correspondence}
and Fig.\ \ref{figure-28}.
\end{proof}
\begin{remark}
\label{remark-Hminus1-wedge-or-snapped-leg}
It follows that for $\mathcal{H}^{-1}(T^\bullet)$ 
to be non-zero a $C$-area must have a direct border with either 
the $T$-area opposite it or the two other $C$-areas. 
Moreover, this border must have internal vertices: 
those which don't also belong to one of the two remaining $T$-areas. 
In other words, one of $C$-areas must contain within its boundary 
a tip of a ``wedge'' and/or squashed legs of $T$
and these must be more than one edge long in total.
\end{remark}

Let now $T$ be the union of $Ty^\bullet z^\bullet$, $Tx^\bullet
z^\bullet$
and $Ty^\bullet z^\bullet$. We now analyse the geometry of $T$ in all
possible cases. We will see that in non-degenerate cases, $T$ is a
concave triangle (see Figure~\ref{figure-29}), while in degenerate cases we find
that $T$ is a union of points, lines and regions (see
Figures~\ref{figure-30} and \ref{figure-31}) and, in particular, $T$
need not be connected.

In general $T$ is a triangle with vertices $e_x$, $e_y$ and $e_z$ and
three sides which are the boundaries of the three $C$-areas. If a
$C$-area is empty we take the corresponding side of $\Delta$ as the
side of $T$. The three $T$-areas are each connected. Each of them
contains precisely one of the vertices of $T$ and is a union of convex
pieces all containing that vertex.

If all three $T$-areas are non-degenerate we have a dichotomy depicted on Figure \ref{figure-29}:
\begin{center}
\parbox{0.8\textwidth}{$\bullet$ \em ``Meeting point'': \rm 
The three $T$-areas meet in a unique point $P \in T$.}

%\vspace{0.2in}
\begin{center} or \end{center}

\vspace{0.1in}
\parbox{0.8\textwidth}{$\bullet$ \em ``Wedge'': \rm One of the 
$T$-areas wedges in between the other two and
disconnects them by touching the opposite side 
of $T$ in more than a point.}
\end{center}

\begin{figure}[!htb] \centering 
\subfigure[``Meeting point'' subdivision] { \label{figure-29a}
\includegraphics[scale=0.10]{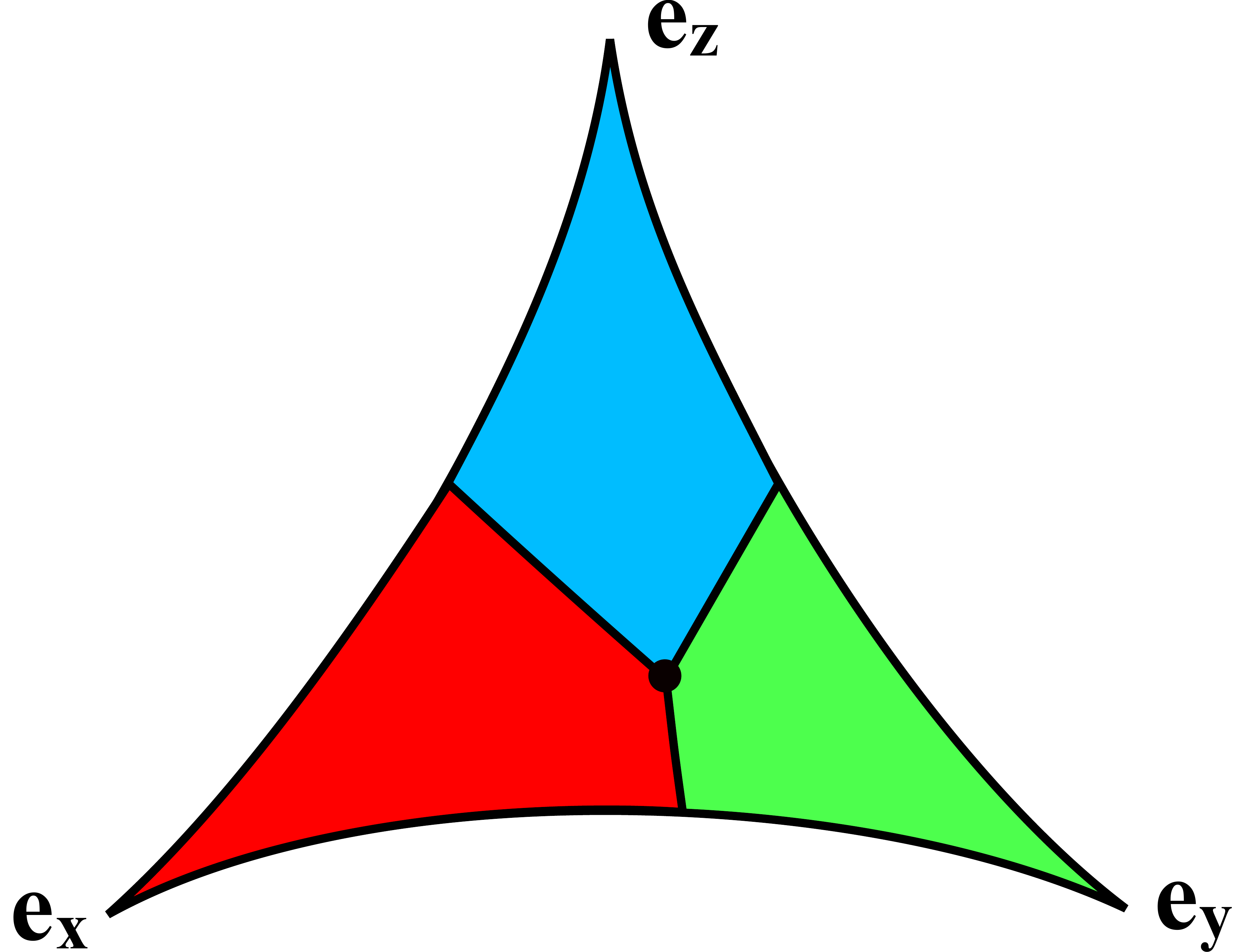}} 
\subfigure[``Wedge'' subdivision] { \label{figure-29b}
\includegraphics[scale=0.10]{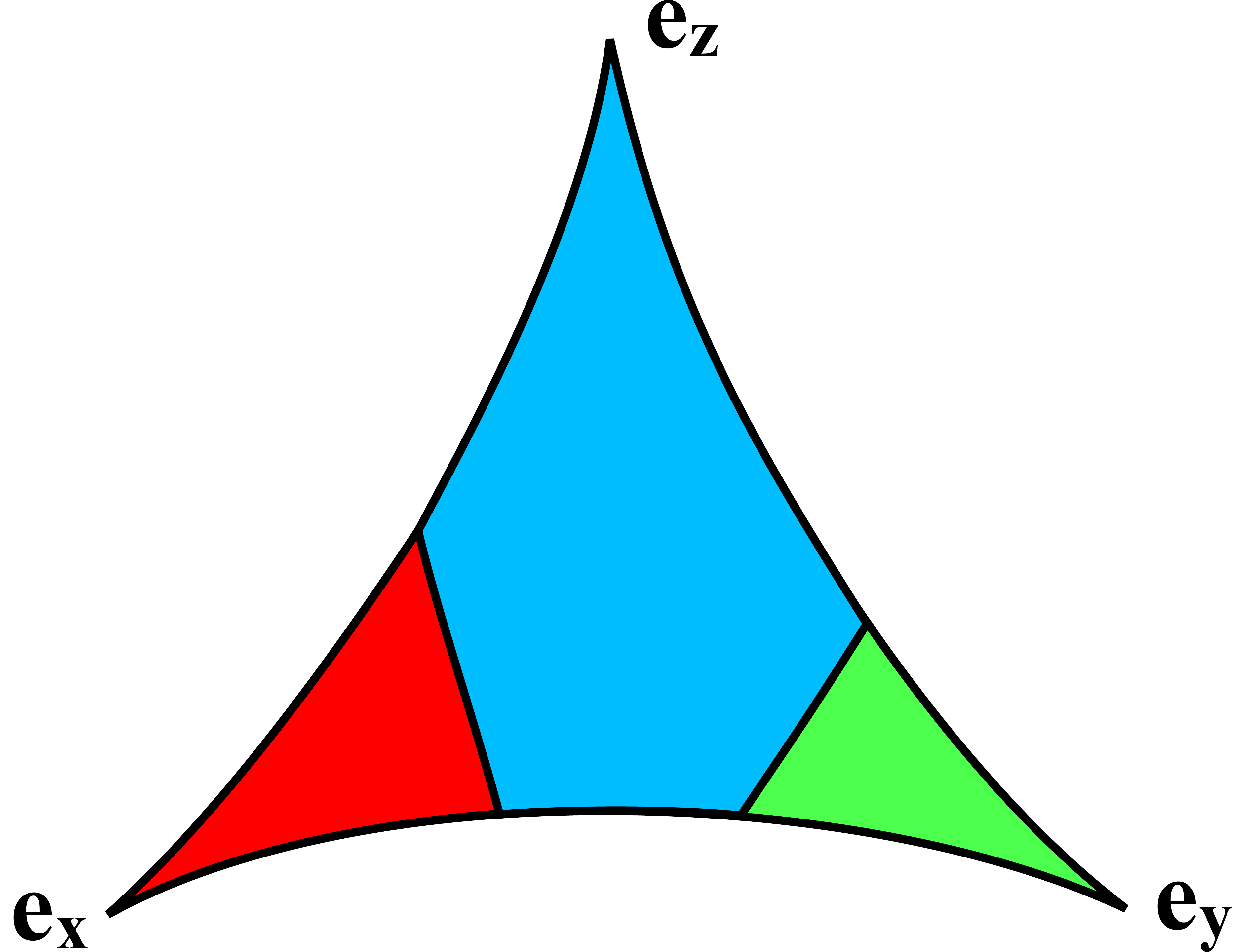}}
\caption{Concave triangle $T$ (non-degenerate case)}
\label{figure-29}
\end{figure}

When a $T$-area degenerates, the neighboring pair 
of $C$-areas squash the leg of $T$ 
between them into a straight line. The $T$-area then
either stretches out into an infinitesimally thin strip 
along the squashed leg or degenerates to just the vertex 
at the tip of the leg, snapped off from the 
rest of $T$.  

On Figures \ref{figure-30} and \ref{figure-31} we 
depict all the degenerations which are possible in view of
Prop.\ \ref{prps-vertex-type-to-CT-subdivision-role-correspondence}.  
All of them can be seen to actually occur in examples.
On Figure \ref{figure-30} we list those degenerations where 
there is a meeting point and on Figure \ref{figure-31} those 
where there isn't. Note that in all the latter cases either
one or both of the following must occur: 
\begin{itemize}
\item \em ``Wedge'': \rm one $T$-area touches the opposite side
of $T$ in more than a point. 
\item \em ``Snap'': \rm a pair of $C$-areas squash a leg of $T$ 
and snap a $T$-area off. 
\end{itemize}

\begin{figure}[!htb] \centering 
\subfigure[One line] { \label{figure-30a}
\includegraphics[scale=0.10]{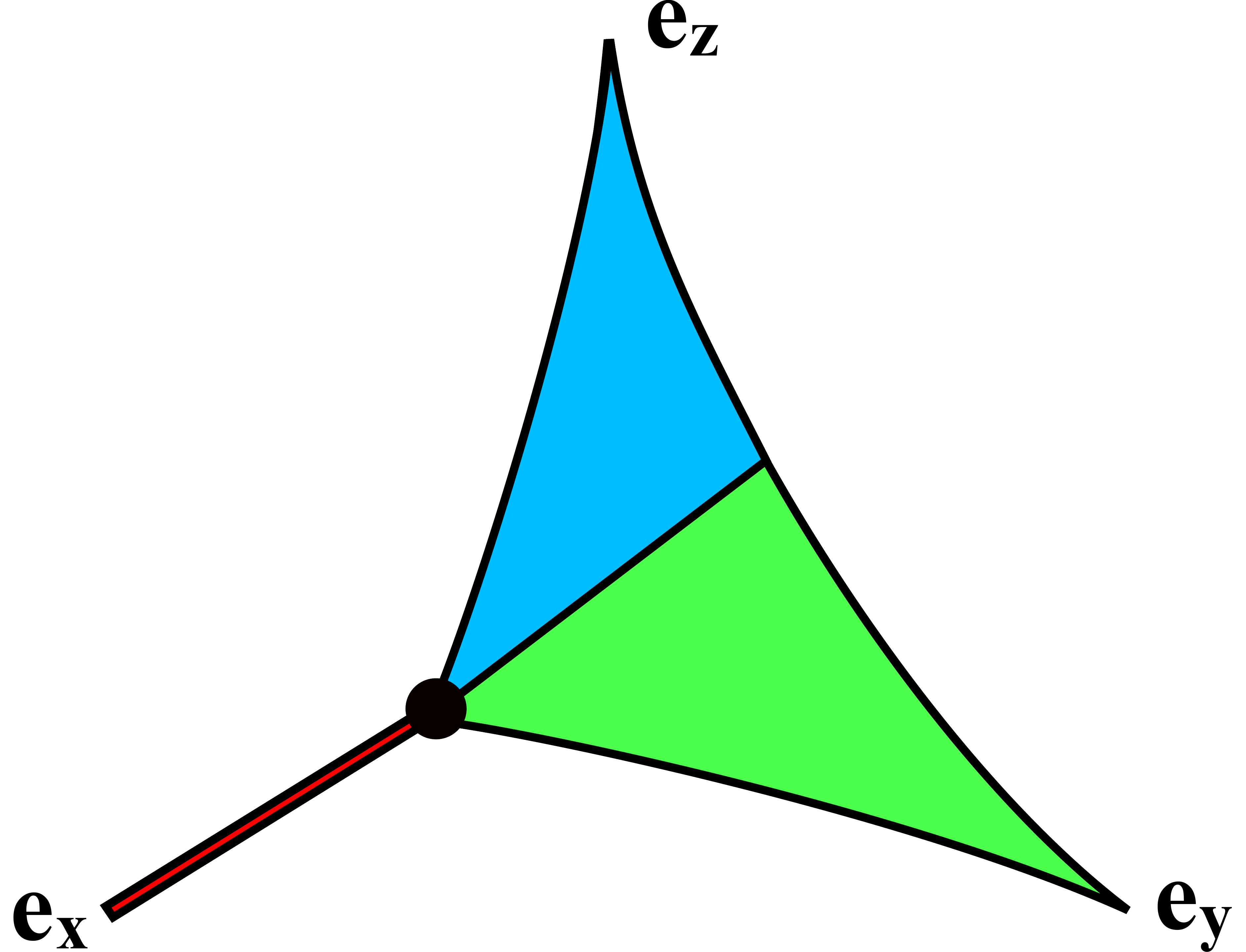}} 
\subfigure[Three lines] { \label{figure-30b}
\includegraphics[scale=0.10]{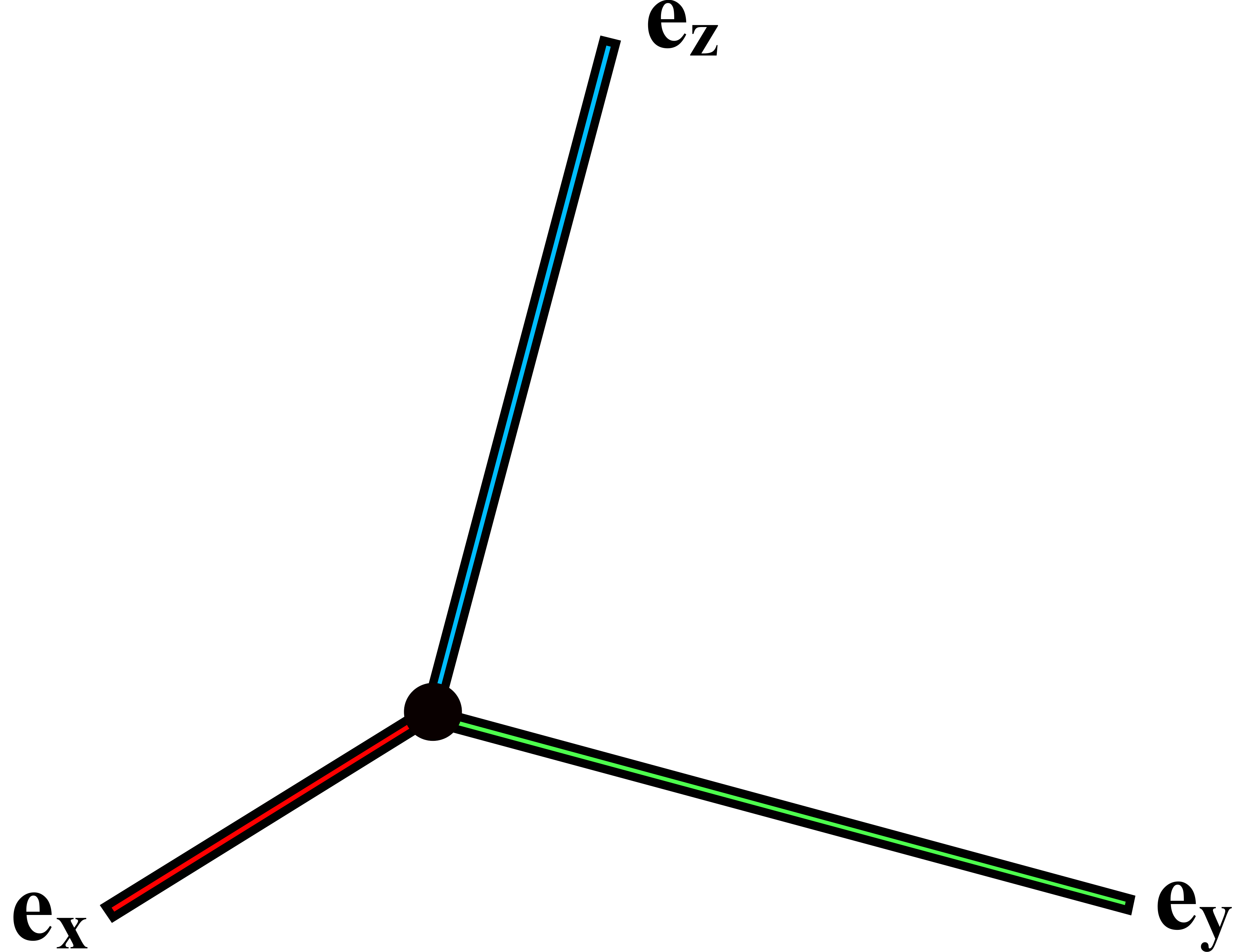}}
\caption{``Meeting point'' degenerations}
\label{figure-30}
\end{figure}

\begin{figure}[!htb] \centering 
\subfigure[One line] { \label{figure-31a}
\includegraphics[scale=0.07]{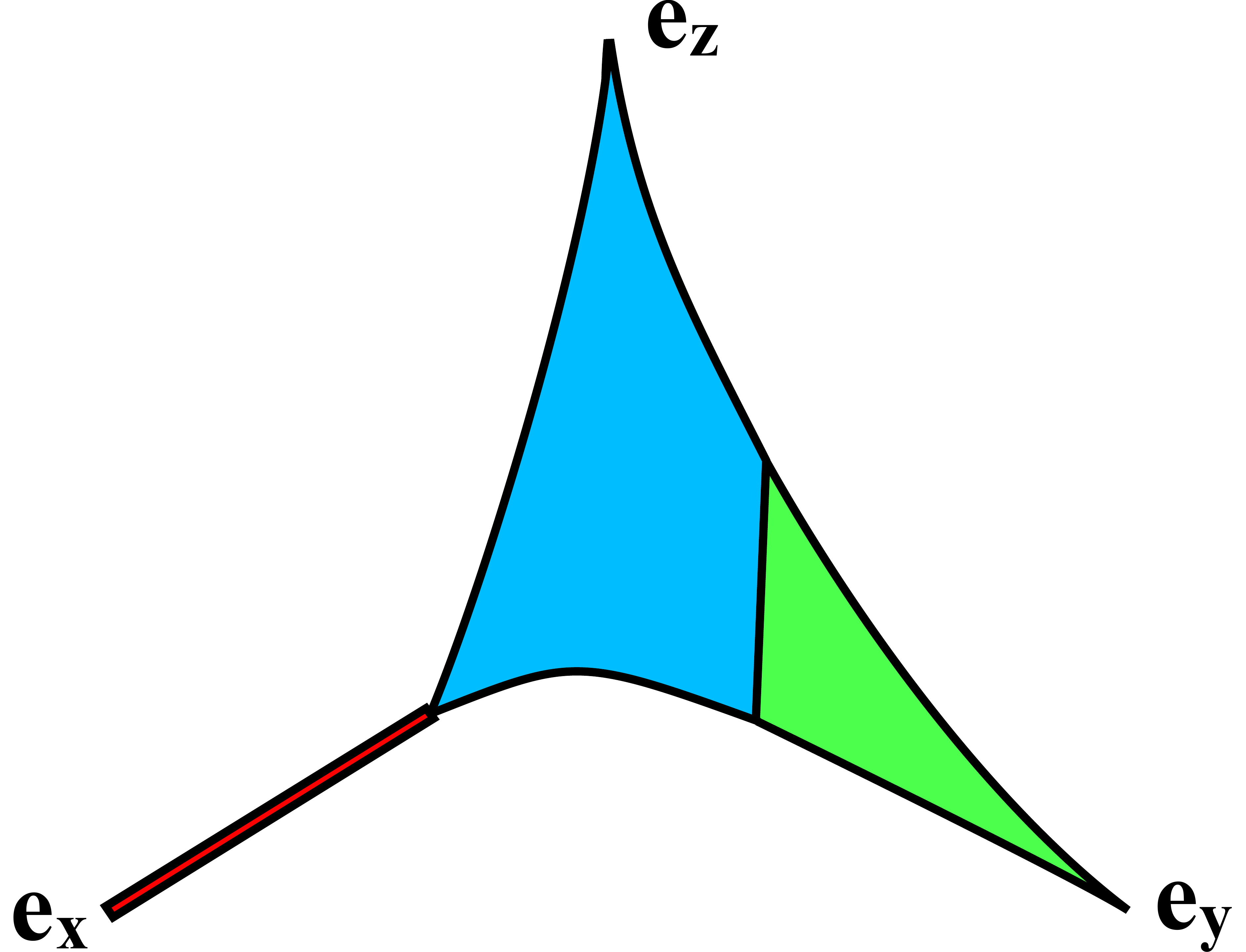}} 
\subfigure[One point] { \label{figure-31h}
\includegraphics[scale=0.07]{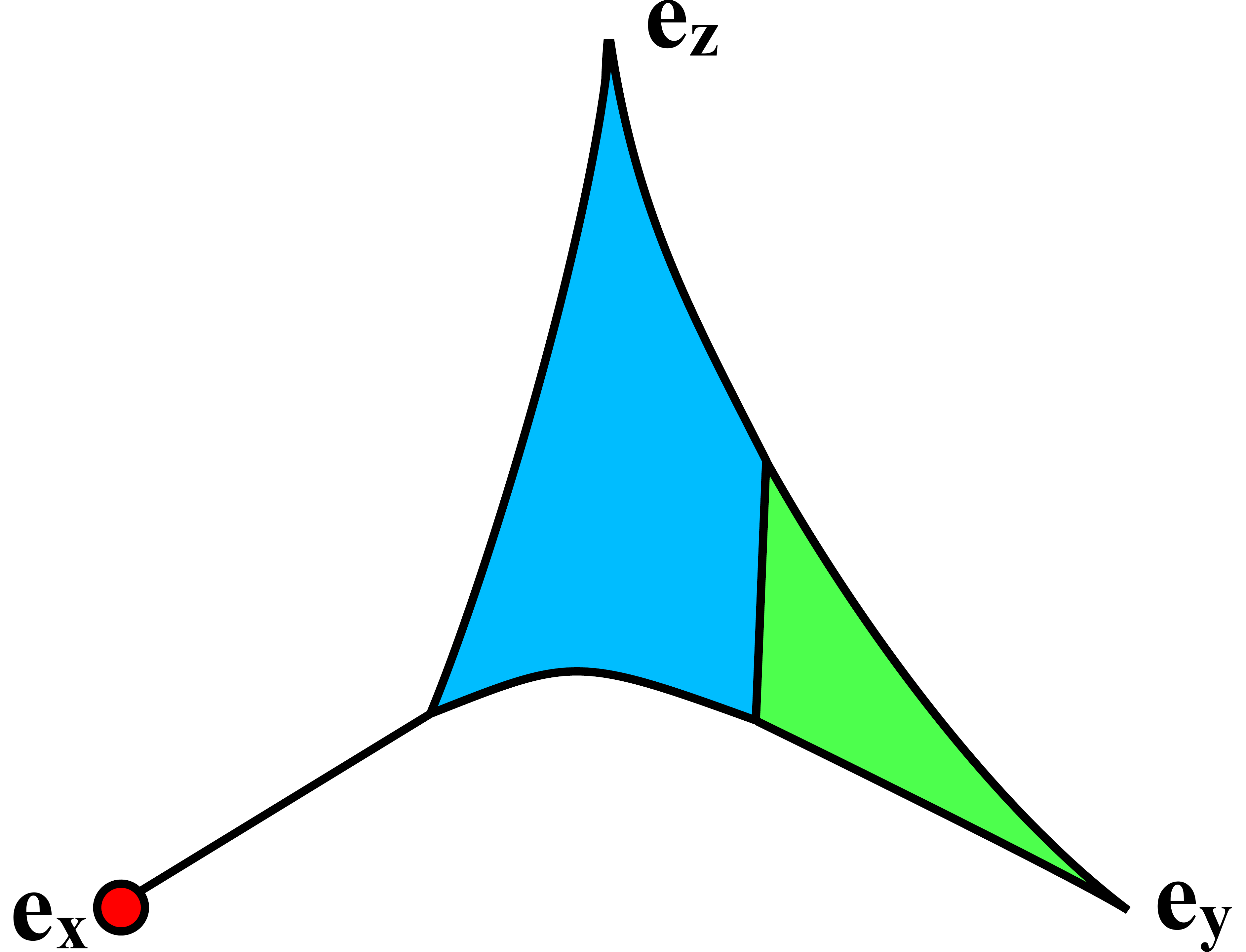}} 
\\
\subfigure[Two lines] { \label{figure-31b}
\includegraphics[scale=0.07]{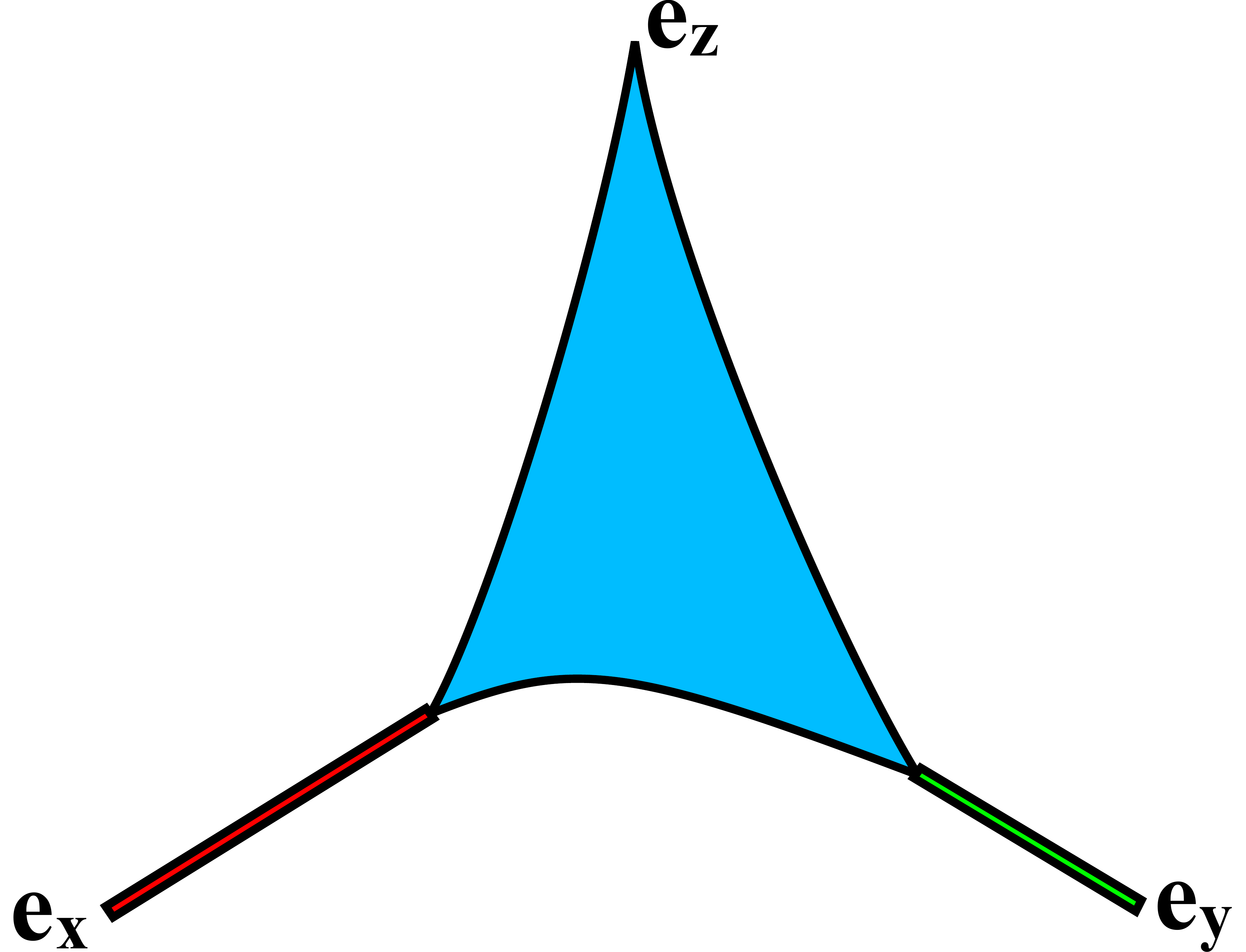}}
\subfigure[One line and one point] { \label{figure-31c}
\includegraphics[scale=0.07]{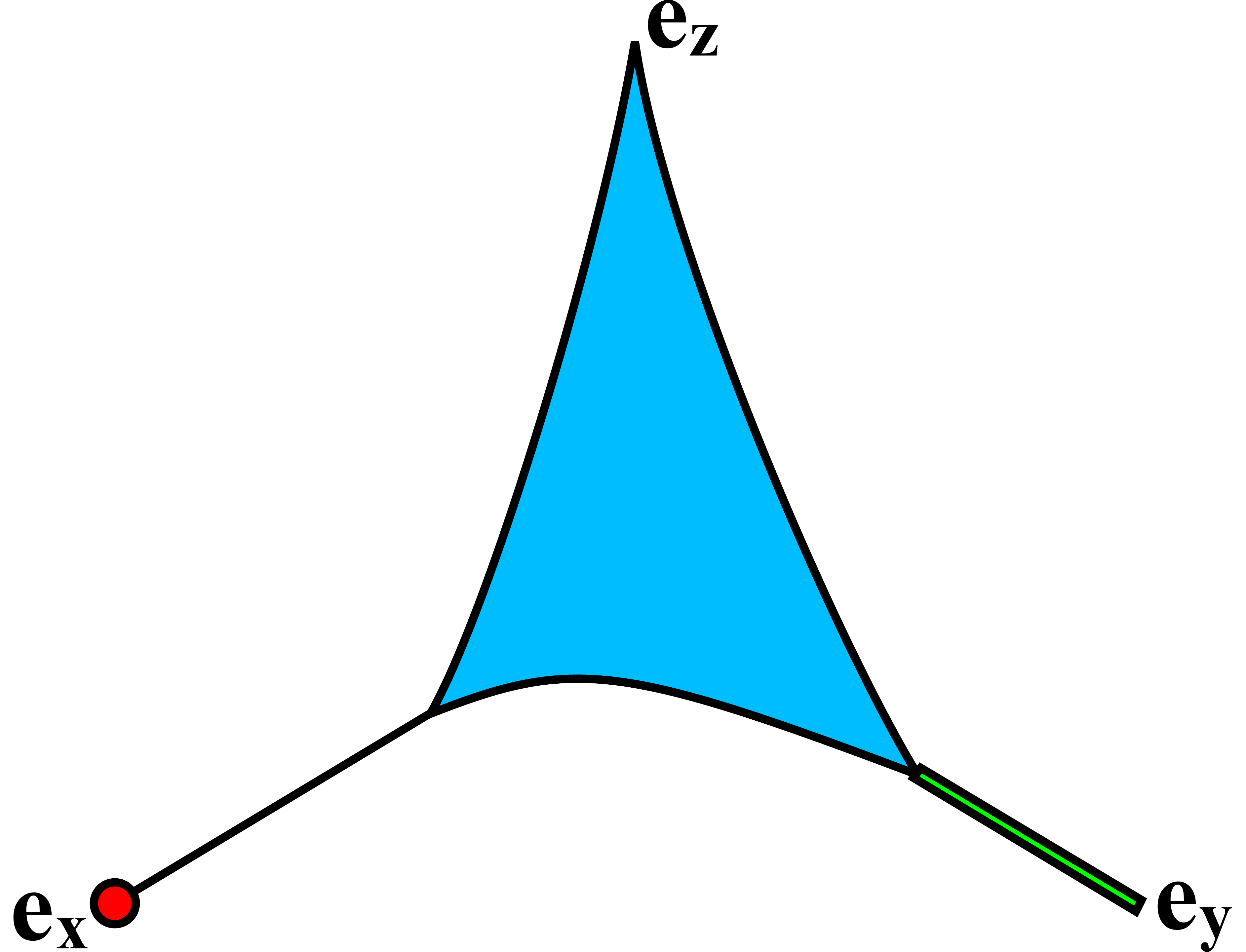}}
\subfigure[Two points] { \label{figure-31d}
\includegraphics[scale=0.07]{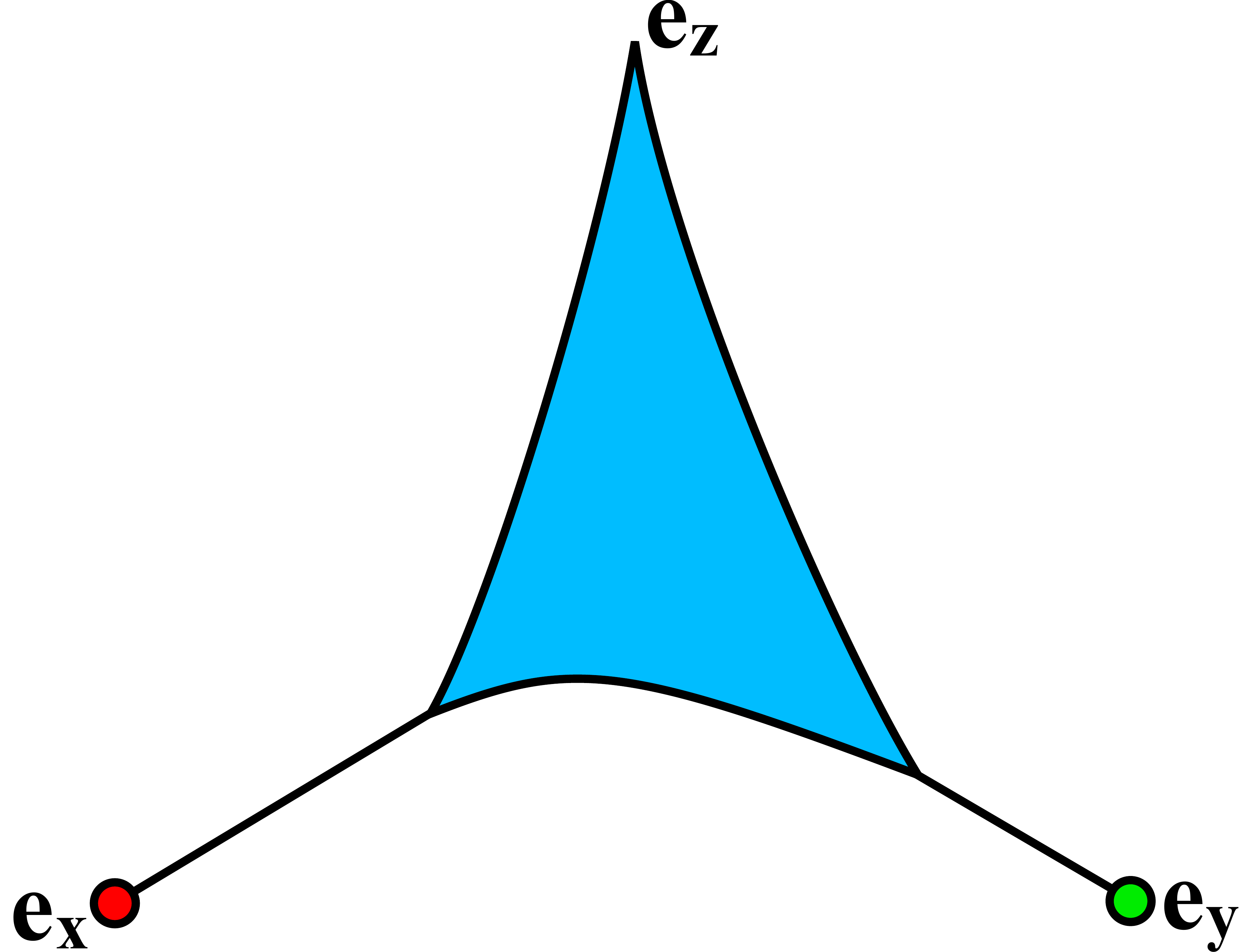}}
\\
\subfigure[Two lines and one point] { \label{figure-31e}
\includegraphics[scale=0.07]{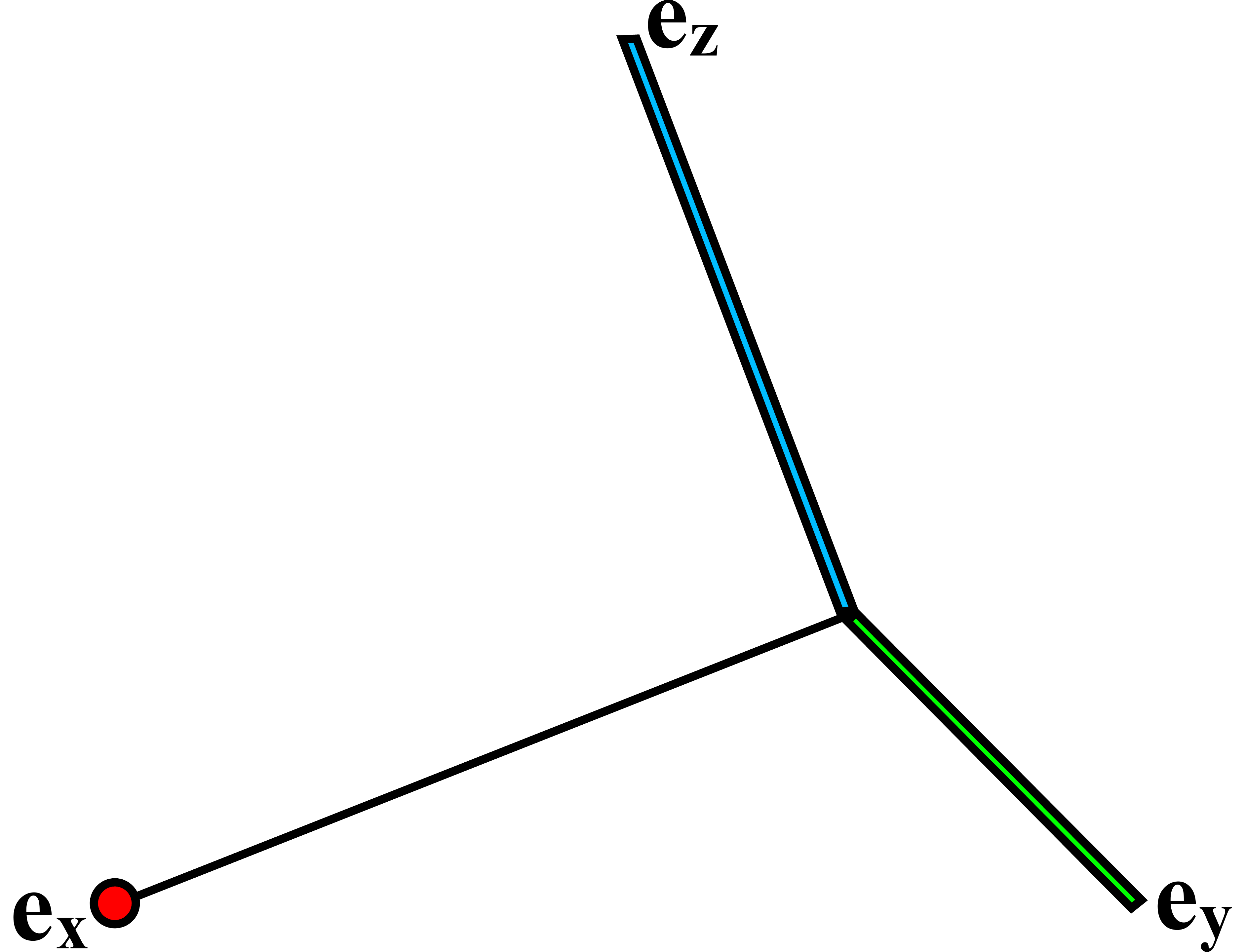}}
\subfigure[One line and two points] { \label{figure-31f}
\includegraphics[scale=0.07]{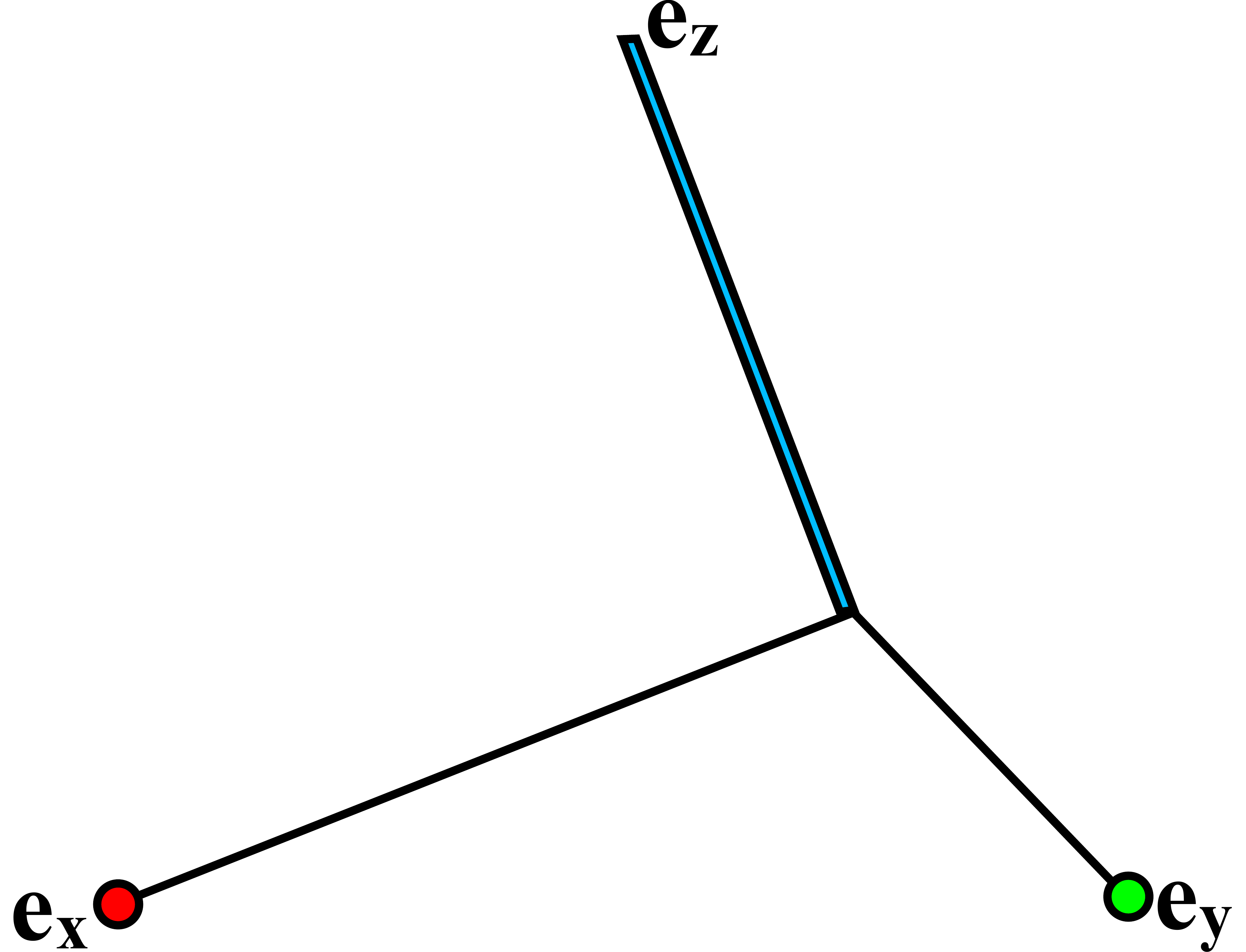}}
\subfigure[Three points] { \label{figure-31g}
\includegraphics[scale=0.07]{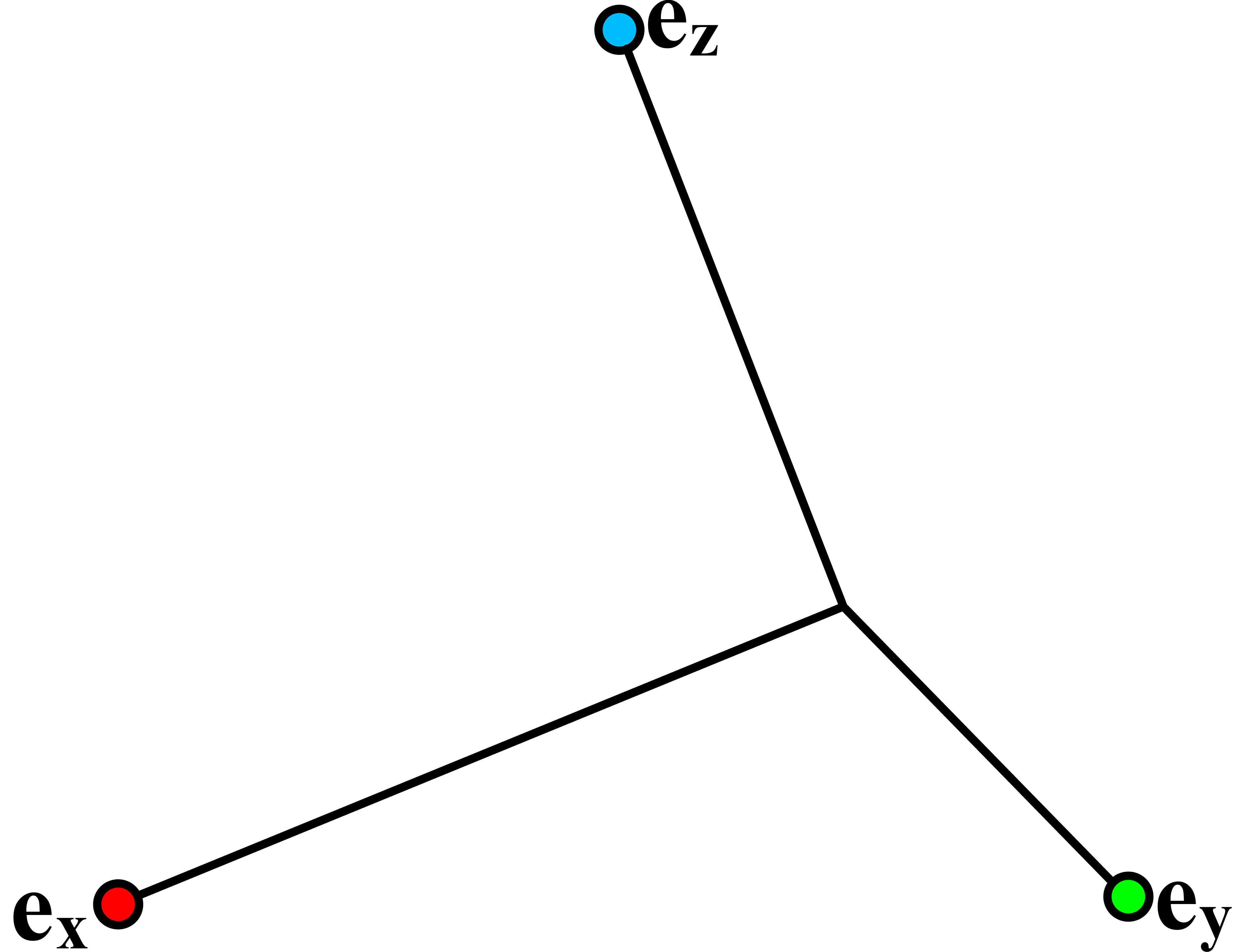}}
\caption{``Wedge'' and ``Snap'' degenerations}
\label{figure-31}
\end{figure}

We can now see the reason only one of 
$\mathcal{H}^0(T^\bullet)$ and  $\mathcal{H}^{-1}(T^\bullet)$ 
can be non-zero. In light of the dichotomy above, 
Cor.~\ref{cor-D1-D2-D3-via-CT-subdivision} implies 
that if $\mathcal{H}^{0}(T^\bullet) \neq 0$ then 
the CT-subdivision either contains a ``meeting point'' or at least 
the three $T$-areas are close enough for there to be a basic triangle 
which touches all three. While Remark 
\ref{remark-Hminus1-wedge-or-snapped-leg} implies that
$\mathcal{H}^{-1}(T^\bullet) \neq 0$ 
only when two $T$-areas become disconnected and
the gap is at least two edges along the side of $T$ they share.

To make the argument above into a proof of Theorem 
\ref{theorem-transform-is-a-sheaf-for-non-trivial-chi} it remains
to deal with the following issue. Suppose two $T$-areas
were disconnected by a ``wedge'' whose tip is more than two edges
long. Then a basic triangle certainly couldn't connect them 
along this tip. But, a priori, there could be a basic triangle 
somewhere else in the ``wedge'' which spans it border-to-border 
and thus touches all three $T$-areas. However, it turns out that 
any basic triangle which connects up the $T$-areas can only do so along 
the tip of the ``wedge''. To prove this assertion, we first observe 
the following:

\begin{corollary}
\label{cor-marked-by-chi-lies-on-wedge-or-a-tip}
Let $\chi \in G^\vee$ be non-trivial. An edge $ef$ of $\Sigma$
is marked by $\chi$ if and only if it belongs to 
either the tip of a ``wedge'' or to a squashed leg with 
a snapped-off $T$-area at the tip.
\end{corollary}
\begin{proof}
Consider the monomial ratio $m : m'$ which carved out $ef$. 
By \cite[Cor.\ 2.3]{Logvinenko-ReidsRecipeAndDerivedCategories}
$m$ and $m'$ represent $\chi$ in the $G$-graphs of 
the two basic triangles which contain $ef$. Therefore edge $ef$ 
either lies on the border of 
$Cx^\bullet$ and $Ty^\bullet z^\bullet$, 
or on the border of
$Cy^\bullet$ and $Tx^\bullet z^\bullet$, 
or on the border of
$Cz^\bullet$ and $Tx^\bullet y^\bullet$,
which places it on the tip of a ``wedge''
or it lies on the border of 
$Cx^\bullet$ and $Cy^\bullet$, or on the border
of $Cy^\bullet$ and $Cz^\bullet$, 
or on the border of $Cz^\bullet$ and $Cx^\bullet$,
which places it along a squashed leg of $T$. And 
since $\chi$ marks $ef$ it is the type of degeneration
where a $T$-area is reduced to just 
the vertex at the tip of leg. 
\end{proof}

Recall that $D^1$, $D^2$ and $D^3$ are the vanishing divisors
of the maps which in the dual family $\widetilde{\mathcal{M}}$
represent the $x$-, $y$- and $z$-arrows entering $\chi^{-1}$. 
Hence they are also the vanishing divisors of the maps which 
in the $G$-cluster $\mathcal{M}$ represent
the $x$-, $y$- and $z$-arrows leaving $\chi$. 
So a point $p \in Y$ lies in $D^1 \cap D^2 \cap D^3$ 
if and only if the whole of $\mathbb{C}[x,y,z]$ acts by zero 
on the $\chi$-eigenspace of the corresponding $G$-cluster. 
It is then said that $\chi$ \em lies in the socle \rm of
$\mathcal{Z}_p$. It was shown in 
\cite[\S6]{Craw-AnexplicitconstructionoftheMcKaycorrespondenceforAHilbC3}
and \cite[Lemma 9.1]{Craw-Ishii-02} that $\chi$ marks a vertex 
of $\Sigma$ if and only if $\chi$ is in the socle of every $G$-cluster
in the corresponding divisor. We now improve upon that observation:

\begin{proposition}
\label{prps-chi-in-the-socle}
Let $\tau$ be a basic triangle of $\Sigma$. 
If $\chi$ lies in the socle of the torus-invariant $G$-cluster
defined by $\tau$, then $\chi$ marks either a vertex or an edge of
$\tau$. Moreover $\chi$ lies in the socle of every $G$-cluster 
in the corresponding divisor or curve of $Y$. 
\end{proposition}
\begin{proof}
Assume that $\tau$ lies in a regular corner triangle $\Lambda$ of side $r$, 
the case of a meeting of champions triangle (cf.  \cite[\S3.1]{CrawReid})
is similar. Assume that $\tau$ is an `up' triangle, 
the `down' case is similar. Up to permutation of $x, y, z$, 
the edges of $\tau$ are cut out
by $x^{d-i} : y^{b+i}z^i$, $y^{e-j} : x^{a+j}z^j$ and
$z^{f-k} : x^ky^{c+k}$ for some $0\leq i,j,k\leq r-1$ with
$i+j+k=r-1$, where $d-a=e-b-c=f=r$.  The shape of the $G$-graph of 
$\tau$ and the position of $\tau$ within $\Lambda$ are 
determined by which of $i,j,k$ are zero, and we distinguish several cases:
\smallskip
 
 \textsc{Case 1}:  If $i,j,k>0$ then each vertex of $\tau$ lies in the
interior of $\Lambda$, the $G$-graph is shown in
\cite[Figure 9]{Craw-AnexplicitconstructionoftheMcKaycorrespondenceforAHilbC3}, and
the socle of the torus-invariant $G$-cluster has basis
\begin{equation}
 \label{eqn:socle}
x^{d-i-1}y^{c+k}, \quad x^{d-i-1}z^j, \quad  x^ky^{e-j-1}, \quad y^{e-j-1}z^i, \quad x^{a+j}z^{f-k-1}, \quad y^{b+i}z^{f-k-1}.
 \end{equation}
The proof of \cite[Lemma 9.1]{Craw-Ishii-02} shows that 
Reid's recipe labels the vertices 
of $\tau$ by the characters $\chi$ of these 6 monomials, so 
the first statement holds in Case 1. 

 \smallskip

 \textsc{Case 2}:  If precisely one of $i,j,k$ is zero, say $k=0$,
then the edge of $\tau$ cut out by $z^{f} : y^{c}$ lies on an edge of
$\Lambda$. Since $k=0$ we have $x^ky^{e-j-1} \vert
y^{e-j-1}z^i$ and $x^{d-i-1}z^j \vert x^{a+j}z^{f-1}$, so the second
and third monomials from \eqref{eqn:socle} are not in the socle. If
$c>0$ then the remaining four monomials from \eqref{eqn:socle} are 
a basis of the socle. Moreover, by
\cite[(4,6)]{Craw-AnexplicitconstructionoftheMcKaycorrespondenceforAHilbC3}
the characters of $x^{a+j}z^{f-1}$ and 
$y^{b+i}z^{f-1}$ mark the vertex of $\tau$ in the interior of $\Lambda$, 
while the characters of $x^{d-i-1}y^{c}$ and $y^{e-j-1}z^i$ 
mark the vertices of $\tau$ on the edge of $\Lambda$ as required. 
If $c=0$ then the edge of $\tau$ cut out by $z^{f} :
y^{c}$ lies in an edge of $\Delta$ and the socle has basis
$x^{a+j}z^{f-1}$ and $y^{b+i}z^{f-1}$ because $x^{d-i-1}y^{c}\vert
x^{a+j}z^{f-1}$ and $y^{e-j-1}z^i \vert y^{b+i}z^{f-1}$. The
corresponding pair of characters marks the vertex of $\tau$ in the
interior of $\Lambda$. Thus the first statement holds in Case 2.
 
 \smallskip

 \textsc{Case 3}: If precisely two of $i,j,k$ are zero, then
we distinguish three subcases:
 
 \smallskip
 \one\;\; If $j=k=0$, then the edges of $\tau$ cut out by $x^a : y^e$
and $z^{f} : y^{c}$ lie in edges of $\Lambda$. Since
$i=f-1$ we have $y^{b+i}z^{f-k-1} \vert y^{e-j-1}z^i$, so the sixth
monomial from \eqref{eqn:socle} does not lie in the socle, and neither
do the second and third monomials because $k=0$. If $a, c>0$ then the
first, fourth and fifth monomials from \eqref{eqn:socle} are a
basis of the socle and the corresponding characters mark the three 
vertices of $\tau$. If $a=0$ and $c\neq 0$ then the edge cut
out by $x^a : y^e$ lies in an edge of $\Delta$, the socle
has basis $y^{e-1}z^i$ and the corresponding character marks the
unique vertex of $\tau$ that lies inside $\Delta$. The case
$c=0$ and $a\neq 0$ is similar. If $a=c=0$, then the unique monomial
in the socle $y^{e-1}z^{f-1}$ is equal to $y^{b+r-1}z^{r-1}$, and the
corresponding character marks the unique edge of $\tau$ which lies
inside $\Delta$. 
 
 \smallskip
 \two\;\; If $i=k=0$, then the edges of $\tau$ cut out by $x^d : y^b$
and $z^{f} : y^{c}$ lie in edges of $\Lambda$. If $b, c>0$
then the socle has basis $x^{d-1}y^c$, $x^{d-1}z^{f-1}$, $y^{e-r}$ and
$y^{b}z^{f-1}$. The characters of three of these monomials mark the
vertices of $\tau$, while $x^{d-1}z^{f-1} = x^{a+j} z^j$ and 
so its character marks 
the unique edge of $\tau$ that does not lie in an edge of $\Lambda$.  
The degenerate cases where $b = 0$ or $c=0$, or both, 
are similar to those in Case \one\ above.
 
 \smallskip
 \three\;\; If $i=j=0$, then the edges of $\tau$ cut out by $x^d :
y^b$ and $x^{a} : y^{e}$ lie in edges of $\Lambda$ and one
vertex of $\tau$ is $e_z$. If $a, b >0$ then the socle has basis
$x^{r-1}y^{e-1}$ and $x^{d-1}y^{c+r-1}$ and the corresponding
characters mark the two vertices of $\tau$ that are not $e_z$.
The degenerate cases with $a = 0$ or $b=0$, or both, are
similar to those in Case \one\ above. When $a=b=0$
the character of the unique monomial 
$x^{d-1}y^{c+r-1}$ in the socle marks an edge of $\tau$. 
This completes the proof of the first statement in Case 3.

 \smallskip
  
 \textsc{Case 4}:  If $i=j=k=0$, then $\Lambda$ is 
a corner triangle with vertex $e_z$ and side $r=1$, i.e.
$\tau$ is the whole of $\Lambda$. Assume that $a, b>0$,
otherwise edges of $\tau$ lie in edges of $\Delta$ and the
calculation simplifies further. The socle has basis $x^ay^c$ and
$y^{b+c}$ and the corresponding monomials mark
the two of vertices of $\tau$ which are not $e_z$.  
This proves the first statement in Case 4.
 \smallskip
 
This completes the proof of the first statement. The second 
statement follows immediately from 
\cite[Lemma~9.1]{Craw-Ishii-02} whenever $\chi$ marks a vertex of 
$\tau$. It remains to prove it when $\chi$ marks 
an edge of $\tau$. The only nondegenerate case where this happens is Case
3\two\ in which the relevant edge of $\tau$ is cut out by
$y^{e-r+1}:x^{d-1}z^{f-1}$. The degenerate cases are similar. 
So let $\tau^\prime$ denote the basic triangle adjacent to $\tau$ 
that shares this edge. The $G$-graph of $\tau^\prime$ is obtained 
from that of $\tau$ by performing the $G$-igsaw transform in the sense of
\cite{Nak00} using $y^{e-r+1}/x^{d-1}z^{f-1}$. This removes 
the monomial $x^{d-1}z^{f-1}$ from the $G$-graph of $\tau$ and
replaces it by the monomial $y^{e-r+1}$. Since $y^{e-r}$ lies in the
socle of the $G$-graph of $\tau$, it is immediate that $y^{e-r+1}$
lies in the socle of the $G$-graph of $\tau^\prime$. Therefore 
$\chi$ lies in the socle of both torus-invariant $G$-clusters in
the curve in $Y$ defined by the edge it marks. 
This proves the second statement. 
\end{proof}
 
In the language of CT-subdivisions Prop.~\ref{prps-chi-in-the-socle} 
states that any basic triangle of $\Sigma$ which touches
all three $T$-areas must contain a vertex or an edge
marked by $\chi$ which also touches all three $T$-areas. In 
light of Cor.~\ref{cor-marked-by-chi-lies-on-wedge-or-a-tip}, it
proves our prior assertion that a basic triangle connecting up
two T-areas disconnected by a ``wedge'' must lie along the tip 
of the ``wedge''.  In fact, it proves:

\begin{proof}[Proof of Theorem
\ref{theorem-transform-is-a-sheaf-for-non-trivial-chi}]

The transform $\Psi(\mathcal{O}_0 \otimes \chi)$ is isomorphic
to the total complex $T^\bullet$ of the skew-commutative
cube of line bundles $\hex(\chi^{-1})_{\widetilde{\mathcal{M}}}$. 
By Lemma \ref{lemma-cube_cohomology}\eqref{item-all-cohomologies-vahish-except-for-degree-zero-to-minus-two}
we have $\mathcal{H}^i(T^\bullet) = 0$ unless $i = 0, -1, -2$.
Since by Lemma \ref{lemma-divisors-to-sink-source-graphs}
a non-trivial $\chi$ could never be a $(3,0)$-sink,
we also have $\mathcal{H}^{-2}(T^\bullet) = 0$
by Lemma \ref{lemma-cube_cohomology}\eqref{item-degree-minus-two-cohomology}. 
To prove the theorem it therefore suffices to prove that for
any non-trivial $\chi$ only one of 
$\mathcal{H}^0(T^\bullet)$ and $\mathcal{H}^{-1}(T^\bullet)$ is non-zero. 
The assertion about being a pushforward from its support  
follows from the formulas for $\mathcal{H}^0$
and $\mathcal{H}^{-1}$ in Lemma \ref{lemma-cube_cohomology}. 

Suppose that $\mathcal{H}^{0}(T^\bullet) \neq 0$. Then by 
Lemma \ref{lemma-cube_cohomology}\eqref{item-degree-zero-cohomology}
and Cor.\ \ref{cor-D1-D2-D3-via-CT-subdivision}
there exists a basic triangle in $\Sigma$ whose 
vertices touch all three $T$-areas. 
By Prop.\ \ref{prps-chi-in-the-socle} such triangle
must contain a vertex or an edge marked by $\chi$ 
which also touches all three $T$-areas. 

Suppose there is a vertex $e$ which touches 
all three $T$-areas. Then the CT-subdivision 
for $\chi$ is of the ``meeting point'' type 
and looks as depicted on Figure \ref{figure-29a} 
or Figure \ref{figure-30}. In particular, it is 
clear that there can be no ``wedges'' or squashed legs
of $T$ snapping off a $T$-area. We conclude that
$\mathcal{H}^{-1}(T^\bullet) = 0$ since by
Lemma \ref{lemma-cube_cohomology}\eqref{item-degree-minus-one-cohomology}, 
Cor.\ \ref{cor-gcd-D_2^3-D_3^2-etc-via-CT-subdivisions} and 
Cor.\ \ref{cor-Z_23-Z_13-Z_12-via-CT-subdivisions}
the support of $\mathcal{H}^{-1}(T^\bullet)$ consists
of all the vertices which are interior to
the tip of a ``wedge'', interior to a squashed leg 
snapping off a $T$-area, or are a meeting point of 
the two. 
 
Suppose there is a $\chi$-marked edge $ef$ which touches 
all three $T$-areas. Any edge marked by $\chi$ belongs
by Cor.\ \ref{cor-marked-by-chi-lies-on-wedge-or-a-tip}
to the tip of a ``wedge'' or a squashed leg of $T$
which snaps off a $T$-area.  So the CT-subdivision 
for $\chi$ looks as depicted on Figure \ref{figure-29b} or 
on Figure \ref{figure-31}. In particular, it is clear
that if a ``wedge'' contains $ef$ in its tip, 
then $ef$ must be the whole of the tip and there are 
no snapped off $T$-areas.  Similarly, if a $T$-area is snapped off 
by a squashed leg of $T$ which contains $ef$, then $ef$ is the whole of this
leg and there are no ``wedges'' or other snapped off $T$-areas. 
In either case we have  $\mathcal{H}^{-1}(T^\bullet) = 0$
for the same reason as above. 
\end{proof}

\section{Derived Reid's recipe}

In this section we compute the transform $\Psi(\mathcal{O}_0 \otimes \chi)$ 
for each $\chi \in G^\vee$ based on the role that $\chi$ plays
in Reid's recipe. Taken together, 
the results in this section give a proof of 
Theorem \ref{theorem-derived-reids-recipe}. Throughout we use the notation 
from \S\ref{section-psi-O-chi-and-skew-commutative-cubes-of-line-bundle}.
Moreover, to simplify our notation in this section, 
we frequently use same letters to denote both vertices of $\Sigma$ and 
the corresponding toric divisors on $Y$.

\subsection{Refined list of the roles $\chi$ can play in Reid's recipe}

First, we need a refinement of the list given in 
\cite[\S4]{Craw-AnexplicitconstructionoftheMcKaycorrespondenceforAHilbC3}
of the roles $\chi$ can play in Reid's recipe:
\begin{proposition}
\label{prps-four-configurations-for-chi-marked-edges} 
For any character $\chi \in G^\vee$ precisely one of the following holds:
\begin{enumerate}
\item \label{item-marks-a-divisor} 
$\chi$ marks a single vertex $e$ of $\Sigma$.  
\item \label{item-marks-a-single-chain} 
$\chi$ marks a single concave chain of edges of $\Sigma$ 
contained within the boundary of a unique $C$-area. 
If this area is, for example, $Cz^\bullet$,
then as illustrated on Fig.\ \ref{figure-23a} and \ref{figure-23b}
the chain consists of one or more straight-line segments 
$P_1 P_2$, \dots, $P_{m} P_{m+1}$ such that: 
\begin{itemize}
\item $e_x P_1$ and $P_{m+1} e_y$
form the boundary of $Cz^\bullet$ with $Ty^\bullet z^\bullet$
and $Tx^\bullet z^\bullet$, respectively. Either may be degenerate, 
cf.\ Fig.\ \ref{figure-23b}.
\item 
$P_2, \dots, P_{m}$ lie on a succession of lines out of $e_z$
\item 
each $P_{i} P_{i+1}$ is carved out by ratio $z^{c} \colon x^{a_{i}} y^{b_{i}}$ 
for some $a_{i}, b_{i}, c \in \mathbb{Z}$ with
\begin{align}
\label{eqn-ratio-coefficients-for-single-e_z-concave-chain}
0 \le a_{1} \le \dots \le a_{m} \quad\text{ and }\quad b_{1} \ge 
\dots \ge b_{m} \ge 0.
\end{align}
\end{itemize}

\begin{figure}[!htb] \centering 
\subfigure[General case] { \label{figure-23a}
\includegraphics[scale=0.11]{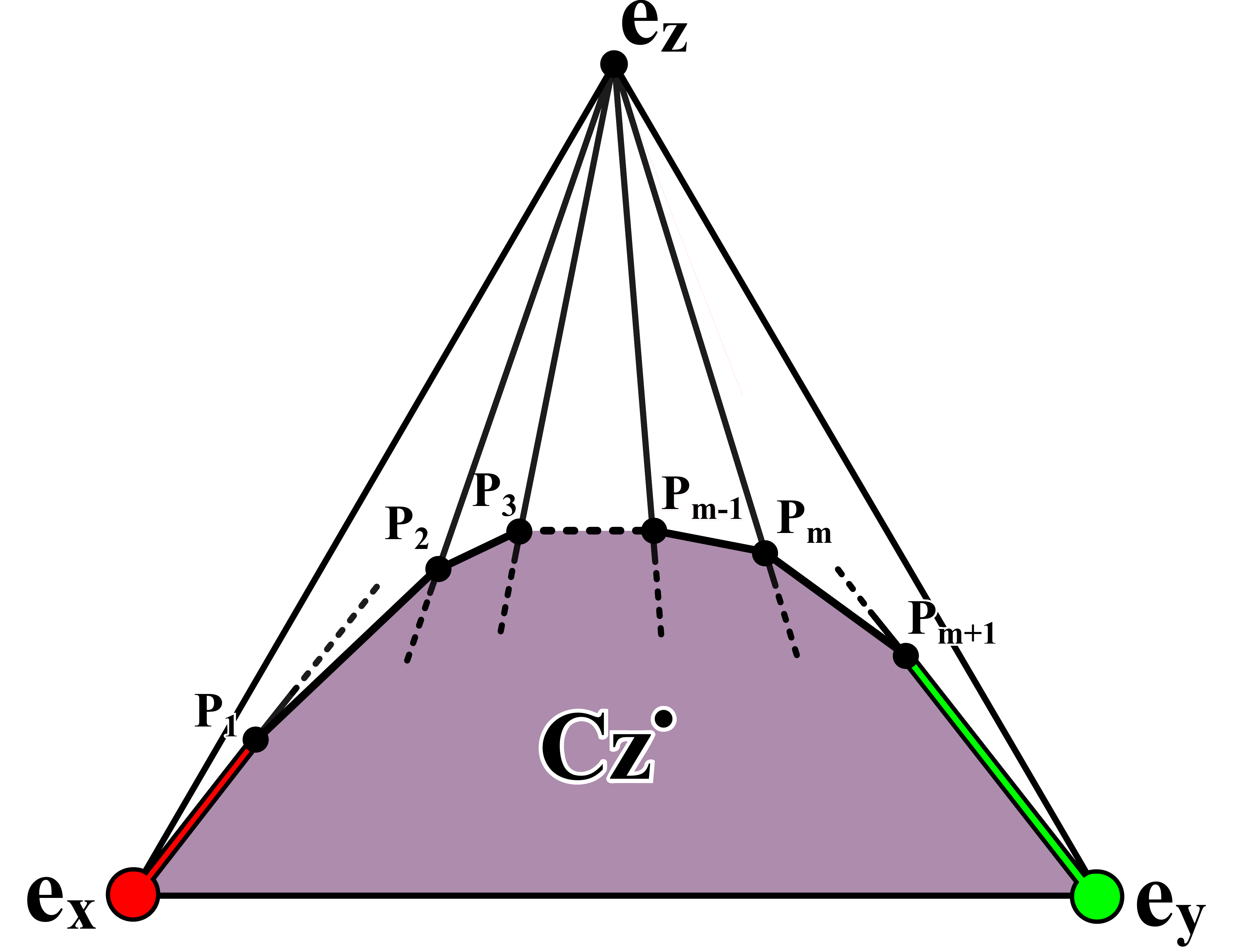}} 
\subfigure[Both $e_x P_1$ and $P_{m+1}e_y$ reduced to points]{ \label{figure-23b}
\includegraphics[scale=0.11]{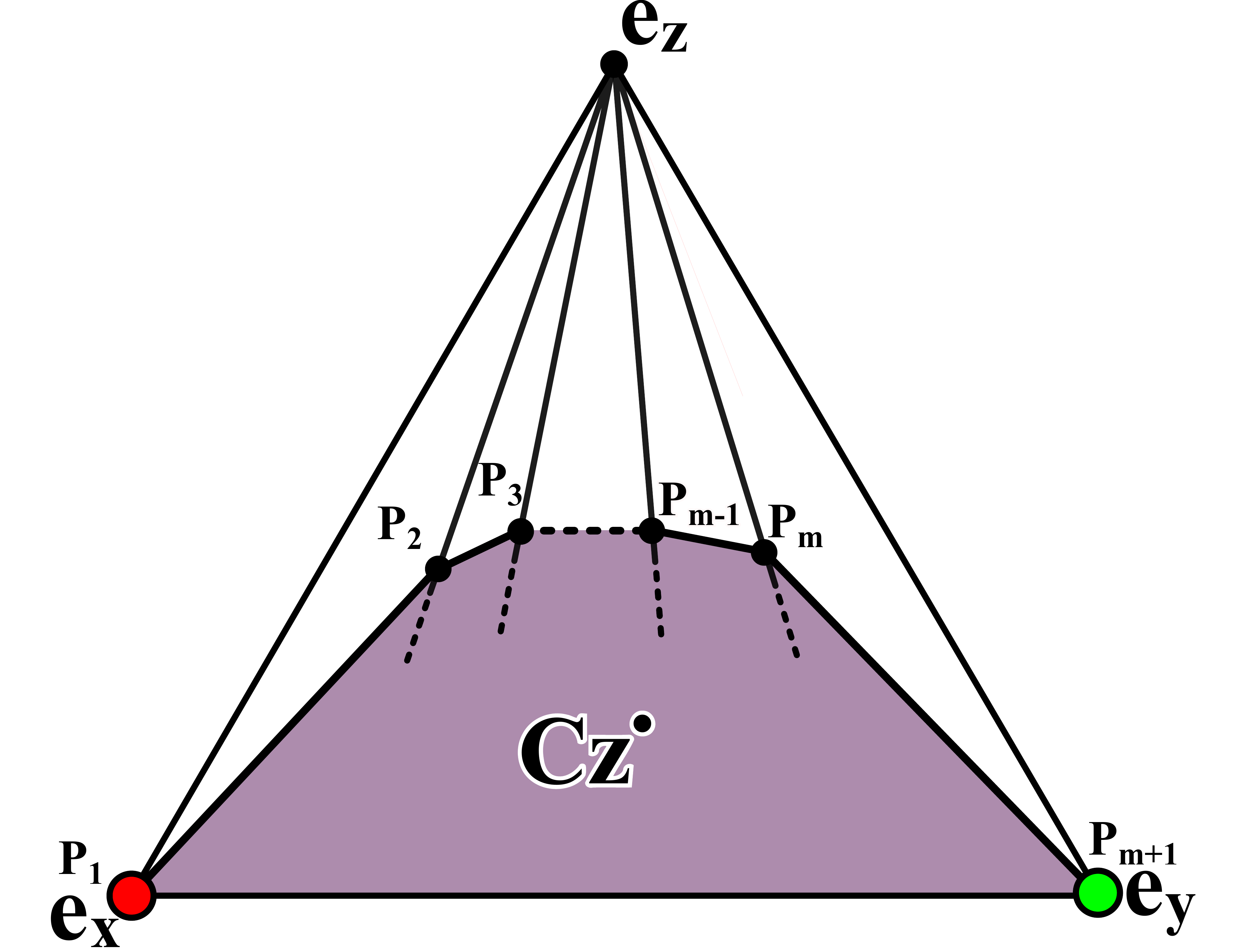}}
\caption{A single $Cz^\bullet$-concave chain}
\label{figure-23}
\end{figure}

\item \label{item-marks-a-meeting-of-champions}

\em ``Meeting of champions'': \rm 
$\chi$ marks three lines in $\Sigma$ which start at 
the three corners of $\Delta$ and meet at some internal vertex $P$. 
The monomial ratios which carve out these lines are 
$x^{a} \colon y^{b}$, $y^{b} \colon z^{c}$ and $z^{c} \colon x^{a}$ 
for some $a,b,c > 0$ and the CT-subdivision for $\chi$ is as 
depicted on Fig.\ \ref{figure-24a}.

\begin{figure}[!htb] \centering 
\subfigure[A ``meeting of champions'']
{ \label{figure-24a} 
\includegraphics[scale=0.11]{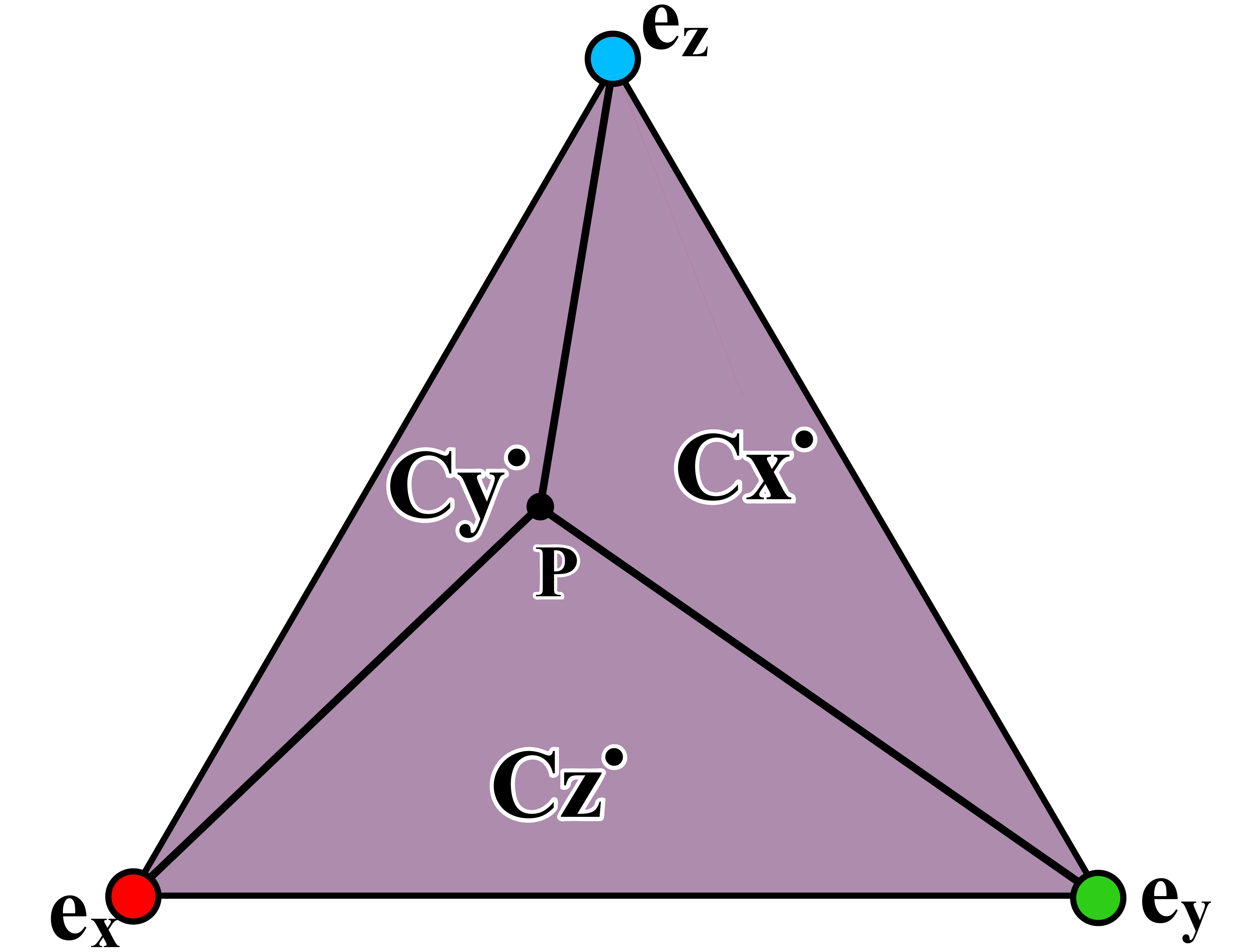}}
\subfigure[A ``long side'']
{\label{figure-24b} 
\includegraphics[scale=0.11]{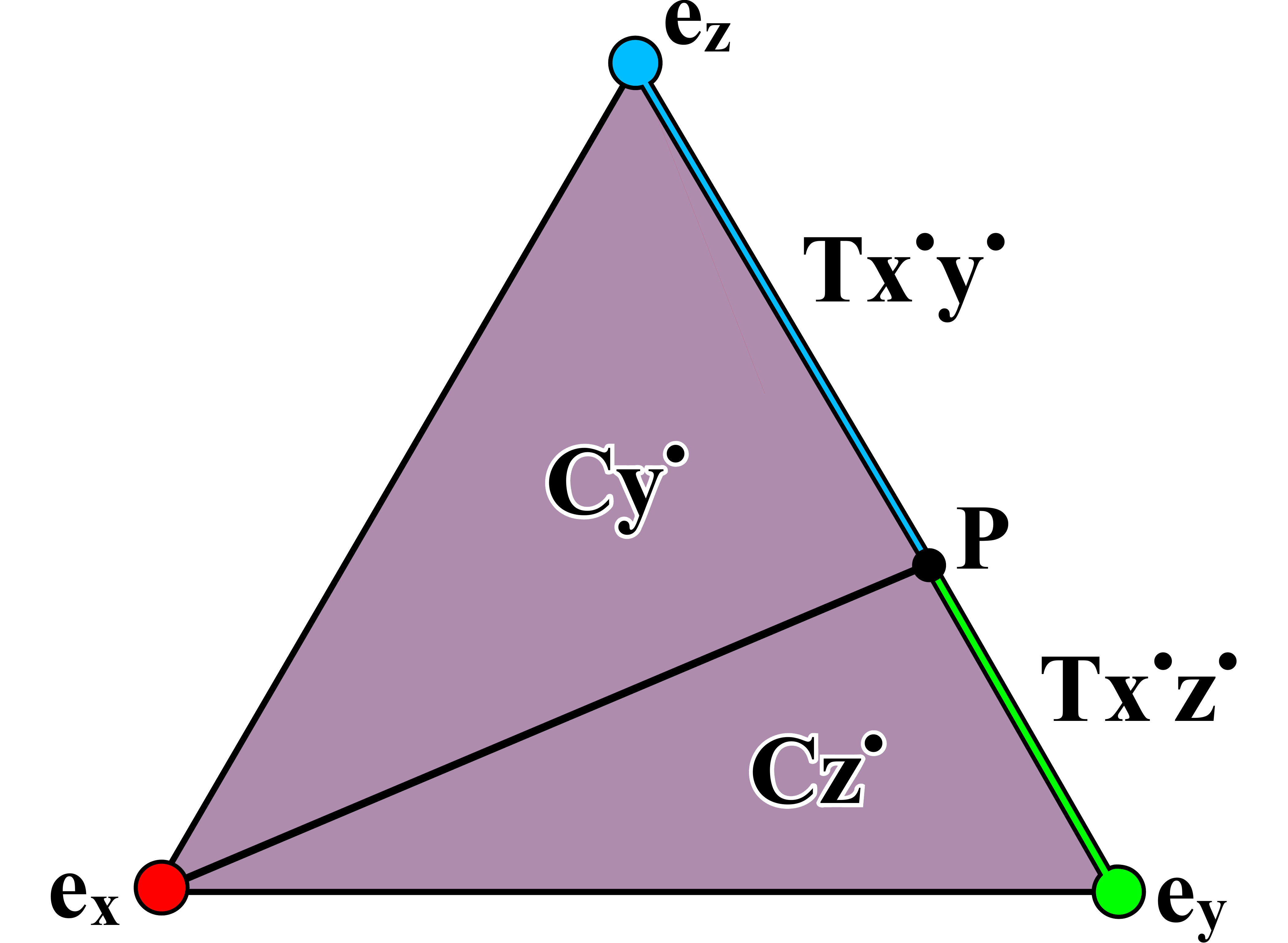}}
\caption{A ``meeting of champions'' and a ``long side''}
\label{figure-24}
\end{figure}

\item \label{item-marks-a-long-side}
\em ``Long side'': \rm $\chi$ 
marks a line out of $e_x$, $e_y$ or $e_z$ which runs 
all the way to a vertex $P$ on the opposite 
side of $\Delta$. The CT-subdivision for $\chi$ is 
as depicted on Figure \ref{figure-24b}. 

\item \label{item-marks-nothing} 
$\chi$ is the trivial character and marks nothing in Reid's recipe.
\end{enumerate}
\end{proposition}

\begin{proof}
In light of
\cite{Craw-AnexplicitconstructionoftheMcKaycorrespondenceforAHilbC3},
we need only show that if $\chi$ marks an edge in $\Sigma$ then
precisely one of 
\eqref{item-marks-a-single-chain},
\eqref{item-marks-a-meeting-of-champions}
or
\eqref{item-marks-a-long-side} holds. Permuting $x, y, z$ if
necessary, the edge is cut out by $z^f : x^dy^e$ for $d\geq 0$ and $e,
f > 0$. Now,
\cite[Lemma~5.3]{Craw-AnexplicitconstructionoftheMcKaycorrespondenceforAHilbC3}
shows that the edge forms part of the piecewise-linear boundary of the
convex region $Cz^\bullet$. If $\chi$ marks an edge that touches a
vertex of valency 3, then the local picture of $\Sigma$ shown in
Figure~\ref{figure-20b} forces the CT-subdivision to be as in 
Figure~\ref{figure-24a} and 
we are in case~\eqref{item-marks-a-meeting-of-champions}.
Otherwise, $\chi$ marks a piecewise-linear curve in the boundary of
$Cz^\bullet$ where, following
\cite[Section~3]{Craw-AnexplicitconstructionoftheMcKaycorrespondenceforAHilbC3},
a chain of edges marked $\chi$ changes direction in $\Sigma$ only if
the chain crosses a line emanating from $e_z$.  Label by $P_1$ the end
of the chain which is closer to $e_x$, then label by $P_2,\dots , P_m$
all the points in succession where the chain crosses a line out of
$e_z$, and finally label the other end by $P_{m+1}$. We now analyze
the boundary of $Cz^\bullet$. If it is a straight line from $e_x$ or
$e_y$, then every edge in the boundary is cut out by $z^f : y^e$. 
Then by 
\cite[Lemma~5.3]{Craw-AnexplicitconstructionoftheMcKaycorrespondenceforAHilbC3}
the CT-subdivision comprises two convex regions
as shown in Figure~\ref{figure-24b} and we are in case~\eqref{item-marks-a-long-side}. Otherwise, if $P_1 \neq e_x$, then the boundary
of $Cz^\bullet$ that joins $e_x$ to $P_1$ is with $Ty^\bullet
z^\bullet$. Recall, that $Cz^\bullet$ is a convex region
and $Ty^\bullet z^\bullet$ is made up of convex regions
each of which contains$e_x$. Consider the convex region of
$Ty^\bullet z^\bullet$ which contains $P_1$. Since it 
also contains $e_x$, it must contains all of $e_xP_1$. 
Thus $e_xP_1$ is a boundary between two convex regions - 
and hence a straight line.

We argue similarly for $P_{m+1}$ and conclude that the boundary of $Cz^\bullet$
consists of two straight lines $e_xP_1$ and $P_{m+1}e_y$ (either may
have length zero) joined by the $\chi$-marked chain $P_1 \dots
P_{m+1}$.  Since the chain lies on the boundary of $Cz^\bullet$, each
edge in the chain is carved out by ratios of type $z^\bullet \colon
x^\bullet$, $z^\bullet \colon x^\bullet y^\bullet$,  $z^\bullet \colon
y^\bullet$. By inspecting every possible local picture of $\Sigma$
around a vertex, we see that if the chain crosses a line $l$ out of
$e_z$, then the edge adjacent to $l$ on the same side as $e_x$ is cut
out by $z^c \colon x^a y^b$ for some $b,c > 0$ and $a \geq 0$, while
on the other side of $l$ the adjacent edge is cut out by $z^c \colon
x^{a'} y^{b'}$ for some $a' > 0$, $b' \geq 0$. Moreover, we have $a
\le a'$ and $b \ge b'$, and both inequalities are strict if and only
if the chain changes direction after crossing $l$. Thus,  each segment
$P_i P_{i+1}$ is carved out by a ratio $z^c \colon x^{a_i} y^{b_i}$
with
$c, a_i, b_i \in \mathbb{Z}$ satisfying the inequalities 
\eqref{eqn-ratio-coefficients-for-single-e_z-concave-chain}
and we are in case~\eqref{item-marks-a-single-chain}.
\end{proof}

\begin{example}
Consider the worked example of Reid’s recipe from 
Section~\ref{section-worked-examples},
especially Figure~\ref{figure-34b}. 
Proposition~\ref{prps-four-configurations-for-chi-marked-edges}
classifies each $\chi\in G^\vee$ as one of five types, where
\begin{enumerate}
\item characters $\chi_1$, $\chi_2$, $\chi_4$, $\chi_7$, $\chi_8$ 
are type \eqref{item-marks-a-divisor}.
\item characters $\chi_3$, $\chi_5$, $\chi_6$, $\chi_9$, $\chi_{10}$, 
$\chi_{11}$, $\chi_{13}$, $\chi_{14}$ are type 
\ref{item-marks-a-single-chain}, cf.
the CT-subdivisions shown in
Figures~\ref{figure-36}-\ref{figure-38} for 
the relevant C-region in each case. 
\item there are no characters of type \eqref{item-marks-a-meeting-of-champions}
because the fan has no meeting of champions. 
\item character $\chi_{12}$ is of type \eqref{item-marks-a-long-side}, 
cf. Figure~\ref{figure-36m} for the relevant C-regions. 
\item the trivial character $\chi_0$ is of type \ref{item-marks-nothing}.
\end{enumerate}
\end{example}

\subsection{The case of $\chi$ marking a single vertex}
\label{section-case-of-chi-marking-a-single-vertex}
Let $\chi \in G^\vee$ marks a single vertex $e$ of $\Sigma$ and 
let $E$ be the corresponding toric divisor on $Y$. 
\begin{proposition}
\label{prps-case-of-chi-marking-a-single-vertex}
We have $\Psi(\mathcal{O}_0 \otimes \chi) = \mathcal{L}^{-1}_\chi \otimes \mathcal{O}_E$.
\end{proposition}
\begin{proof}
This was proved in \cite[Prop.\ 9.3] {Craw-Ishii-02}, but 
a direct proof is short enough to include it here. 
Recall that a point $p \in Y$ lies in $D^1 \cap D^2 \cap D^3$ 
if and only if $\chi$ is in the socle of the $G$-cluster $\mathcal{Z}_p$. 
By Prop.\ \ref{prps-chi-in-the-socle} a $G$-cluster $\mathcal{Z}_p$
has $\chi$ in the socle only when the point $p$ lies on a divisor 
or curve marked by $\chi$. Since $\chi$ marks only $E$, we must have 
$D^1 \cap D^2 \cap D^3 \subseteq E$. On the other hand, 
by Theorem \ref{theorem-sink-source-graph-to-divisor-type-correspondence}
(or, in fact, \cite[Prop.\ 9.1]{Craw-Ishii-02}), character $\chi$ marks $e$ 
if and only if $E \subseteq D^1 \cap D^2 \cap D^3$. We conclude 
that $D^1 \cap D^2 \cap D^3 = E$, so Lemma 
\ref{lemma-cube_cohomology} gives 
$$ \mathcal{H}^0(\Psi(\mathcal{O}_0 \otimes \chi)) = 
\mathcal{L}^{-1}_\chi \otimes \mathcal{O}_{E}.$$
Finally, Theorem \ref{theorem-transform-is-a-sheaf-for-non-trivial-chi} 
implies that all the other cohomology sheaves of 
$\Psi(\mathcal{O}_0 \otimes \chi)$ vanish. 
\end{proof}

\subsection{The case of $\chi$ marking a single concave chain of edges}
\label{section-case-of-a-single-concave-chain}

Let $\chi \in G^\vee$ be such that $\chi$ marks
a single concave chain of edges of $\Sigma$ contained
within the boundary of a unique $C$-area. Permuting
$x$, $y$ or $z$ if necessary, we may assume that the 
the chain is contained within the boundary of $Cz^\bullet$, 
as described in 
Prop.
\ref{prps-four-configurations-for-chi-marked-edges}\eqref{item-marks-a-single-chain} and illustrated on Fig.\ \ref{figure-23}. 

\begin{proposition}
\label{prps-chi-marks-a-single-concave-chain}
\begin{enumerate}
\item
\label{item-single-edge-chain}
If the chain consists of a single edge $P_1 P_2$, then
$$ \Psi(\mathcal{O}_0 \otimes \chi)) = \mathcal{L}^{-1}_\chi \otimes
\mathcal{O}_{C} $$
where $C$ is the toric curve $P_1 \cap P_2$ corresponding to that
edge. 
\item
\label{item-multiple-edge-chain}
If the chain consists of more than one edge, then 
$$ \Psi(\mathcal{O}_0 \otimes \chi) =  
\mathcal{L}^{-1}_{\chi}(-P_1 - P_{m+1}) \otimes \O_Z [1] $$
where $Z$ is the union of the divisors which correspond
to the internal vertices of chain $P_1 \dots P_{m+1}$.
\end{enumerate}
\end{proposition}
\begin{proof}
Suppose first, that the chain consist of a single edge $P_1 P_2$. 
Consider the CT-subdivision for $\chi$. By 
Cor.\ \ref{cor-marked-by-chi-lies-on-wedge-or-a-tip} the edges
marked by $\chi$ are those in the tip of a ``wedge'' or 
in a squashed leg of $T$ snapping off a $T$-area. So either there
is a single wedge with tip $P_1P_2$
and no ``snaps'' (Fig.\ \ref{figure-29b}, \ref{figure-31a}, \ref{figure-31b}) 
or a single ``snap'' $P_1 P_2$ and no ``wedge'' 
(Fig.\ \ref{figure-31e}). In both cases 
$P_1 P_2$ touches all three $T$-areas and hence by 
Cor.\ \ref{cor-D1-D2-D3-via-CT-subdivision} we have 
$P_1 \cap P_2 \subseteq D^1 \cap D^2 \cap D^3$. Then we argue 
as in Prop.\ \ref{prps-case-of-chi-marking-a-single-vertex}: 
by Prop.\ \ref{prps-chi-in-the-socle} $P_1 \cap P_2$ is the whole of $D^1 \cap D^2 \cap D^3$,
by Lemma \ref{lemma-cube_cohomology}
$$ \mathcal{H}^0 \left(\Psi(\mathcal{O}_0 \otimes \chi)\right) = \mathcal{L}^{-1}_\chi \otimes
\mathcal{O}_{P_1 \cap P_2} $$
and by Theorem \ref{theorem-transform-is-a-sheaf-for-non-trivial-chi} all
the other cohomology sheaves of $\Psi(\mathcal{O}_0 \otimes \chi)$ 
vanish. 

Suppose now the chain consists of more than one edge. 
Taking $I = \{23\}$, $J = \{13\}$ and $K = \{ 12 \}$ in 
Lemma \ref{lemma-cube_cohomology}\eqref{item-degree-minus-one-cohomology}
it follows that there is a three-step filtration of 
$\mathcal{H}^{-1}\left(\Psi(\mathcal{O}_0 \otimes \chi)\right)$ 
with successive quotients
\begin{itemize} 
\item 
$\mathcal{F}''_{23} = \O_{Z_{23}} \otimes \L_{23}(\gcd(D_2^3,D_3^2))$ 
with $Z_{23}$ being the scheme theoretic intersection of 
$\gcd(D_2^3,D_3^2)$ and the effective part of 
$D^1 + \lcm(\tD_1^2, \tD_1^3) - \tD_2^3 - D^2$.
\item 
$\mathcal{F}_{13} = \O_{Z_{13}} \otimes \L_{13}(\gcd(D_1^3,D_3^1))$ 
with $Z_{13}$ being the scheme theoretic intersection 
of $\gcd(D_1^3,D_3^1)$ and the effective part of $D^2 +
\lcm(\tD_2^1,D_2^3) - \tD_3^1 - D^3$ 
\item 
$\mathcal{F}_{12} = \O_{Z_{12}} \otimes \L_{12}(\gcd(D_1^2,D_2^1))$ 
with $Z_{12}$ is the scheme theoretic intersection of 
$\gcd(D_1^2,D_2^1)$ and the effective part of 
$D^3 + \lcm(D_3^1,D_3^2) - \tD_1^2 - D^1$ 
\end{itemize} 
where $\tD^i_j =  D^i_j - \gcd(D^i_j,D^j_i)$. 

Corollaries \ref{cor-gcd-D_2^3-D_3^2-etc-via-CT-subdivisions}
and \ref {cor-Z_23-Z_13-Z_12-via-CT-subdivisions} imply that
$$ 
Z_{12} = \left(Cz^\bullet \cap 
\left(Cx^\bullet \cup Tx^\bullet y^\bullet \cup Cy^\bullet\right) \right)
\setminus 
\left(Ty^\bullet z^\bullet \cup Tx^\bullet z^\bullet\right). 
$$
The chain $P_1 \dots P_{m+1}$ is the boundary of $Cz^\bullet$ with 
$Cx^\bullet \cup Tx^\bullet y^\bullet \cup Cy^\bullet$, while
$e_x P_1$ and $P_{m+1}e_y$ are its boundaries with 
$Ty^\bullet z^\bullet$ and with $Tx^\bullet z^\bullet$. 
Therefore $Z_{12}$ is the union of the toric divisors 
corresponding to the internal vertices of the chain $P_1 \dots P_{m+1}$. 

On the other, we claim that $Z_{13} = \emptyset = Z_{23}$.
To verify the claim for $Z_{13}$, note that 
$$ 
Z_{13} = \left(Cy^\bullet \cap 
\left(Cx^\bullet \cup Tx^\bullet z^\bullet \right)\right) \setminus
\left(Ty^\bullet z^\bullet \cup Tx^\bullet y^\bullet\right). 
$$
Suppose $Z_{13}$ is non-empty. 
Since $Cy^\bullet$ and $Cx^\bullet \cup Tx^\bullet z^\bullet$
are disconnected by $Cz^\bullet \cup Tx^\bullet y^\bullet$, 
there has to be a basic triangle $\tau$ in $Cz^\bullet$ with 
one vertex on the border with $Cy^\bullet$ and 
another on the border with $Cx^\bullet \cup Tx^\bullet z^\bullet$. 
The border between $Cz^\bullet$ and $Cy^\bullet$ exists only 
when $P_1 = e_x$ and is $P_1 P_2$. Therefore $P_1 = e_x$
and one vertex of $\tau$ must be an internal vertex of $P_1 P_2$. 
Since $\chi$ doesn't mark a ``long side'', line $P_1 P_2$ 
terminates at an internal line $e_z P_2$ out of $e_z$, see Fig.
\ref{figure-23}. The Craw-Reid algorithm dictates then
that $\tau$ must lie on the same side of $e_z P_2$ as $P_1 P_2$.
On the other hand, the border of $Cz^\bullet$ with 
$Tx^\bullet z^\bullet$ is $P_{m+1} e_y$ and
the border of $Cz^\bullet$ with $Cx^\bullet$ 
exists only when $P_{m+1} = e_y$ and is $P_m P_{m+1}$.
Both of these lie on the other side of line $e_z P_2$ from $P_1 P_2$, 
and can't therefore contain a vertex of $\tau$. This is a contradiction,
and we conclude that $Z_{13} = \emptyset$. A similar argument for
$$
Z_{23} = \left(Cx^\bullet \cap  
Ty^\bullet z^\bullet \right) \setminus 
\left(Tx^\bullet z^\bullet \cup Tx^\bullet y^\bullet \right).
$$
shows that $Z_{23} = \emptyset$.

Thus far, we've shown that 
$$\mathcal{H}^{-1}\left(\Psi(\mathcal{O}_0 \otimes \chi)\right)
= \mathcal{L}_{12}(\gcd(D^2_1, D^1_2)) \otimes \mathcal{O}_{Z} $$
where $Z$ is the union of all the divisors 
corresponding to the internal vertices of $P_1 \dots P_{m+1}$.
In Fig.\ \ref{figure-13b}
we have $\mathcal{L} = \mathcal{L}_{123} = \mathcal{L}_{\chi}^{-1}$. 
Since $D^3_{12}$ was defined to be the vanishing divisor of the regular
map $\mathcal{L}_{123} \xrightarrow{\alpha^3_{12}} \mathcal{L}_{12}$
we have $\mathcal{L}_{12} = \mathcal{L}_{123}(D^3_{12})$. 
Therefore 
$$\mathcal{L}_{12}\left(\gcd(D^2_1, D^1_2)\right) = 
\mathcal{L}^{-1}_{\chi}\left(D^3_{12} + \gcd(D^2_1, D^1_2)\right).$$

The sum of vanishing divisors along any path from $\mathcal{L}_{123}$
to $\mathcal{L}$ is the principal divisor $(xyz)$, so 
$D^3_{12} + \gcd(D^2_1, D^1_2) = (xyz) - \lcm(D^1,D^2)$. 
By Cor.\ \ref{cor-D1-D2-D3-via-CT-subdivision}
$\lcm(D^1, D^2) = Ty^\bullet z^\bullet \cup Tx^\bullet z^\bullet$, 
so
$$\mathcal{H}^{-1}\left(\Psi(\mathcal{O}_0 \otimes \chi)\right)
= \mathcal{L}^{-1}_\chi(- Ty^\bullet z^\bullet \cup Tx^\bullet z^\bullet)
\otimes \mathcal{O}_{Z}. $$
Two toric divisors intersect if and only if the corresponding
vertices of $\Sigma$ are adjacent. So when we restrict
$\mathcal{O}_Y(- Ty^\bullet z^\bullet \cup Tx^\bullet z^\bullet)$
to $Z$ only those vertices in $Ty^\bullet z^\bullet$ and 
$Tx^\bullet z^\bullet$ which are adjacent 
to the internal vertices of $P_1 \dots P_{m+1}$ 
make a non-zero contribution. But $P_1 \dots P_{m+1}$ is the border of 
$Cz^\bullet$ and $Cx^\bullet \cup Tx^\bullet y^\bullet \cup Cy^\bullet$
and it is clear from Figure \ref{figure-23} that the only vertices
of $Ty^\bullet z^\bullet \cup Tx^\bullet z^\bullet$
which are adjacent to the internal vertices of $P_1 \dots P_{m+1}$
are $P_1$ and $P_{m+1}$. Hence 
$\mathcal{O}_Y(- Ty^\bullet z^\bullet \cup Tx^\bullet z^\bullet)$
and $\mathcal{O}_Y(-P_1 - P_{m+1})$ restrict to the same line-bundle on $Z$. 
We conclude that 
$$ \mathcal{H}^{-1}\left(\Psi(\mathcal{O}_0 \otimes \chi)\right)
= \mathcal{L}^{-1}_\chi(-P_1 - P_{m+1}) \otimes \mathcal{O}_Z. $$
Since $Z$ is non-empty,
$\mathcal{L}^{-1}_\chi(-P_1 - P_{m+1}) \otimes \mathcal{O}_Z$ 
is non-zero and by Theorem 
\ref{theorem-transform-is-a-sheaf-for-non-trivial-chi}
all the other cohomology sheaves of 
$\Psi(\mathcal{O}_0 \otimes \chi)$ vanish.
\end{proof}

\begin{example}
In the worked example in Section~\ref{section-worked-examples}, 
cf. Figure~\ref{figure-34b}, 
Prop.\ \ref{prps-chi-marks-a-single-concave-chain}\eqref{item-single-edge-chain}
applies to characters $\chi_3$, $\chi_{11}$ and $\chi_{14}$,
while Prop.\ \ref{prps-chi-marks-a-single-concave-chain}\eqref{item-multiple-edge-chain} 
applies to characters $\chi_5$, $\chi_6$, $\chi_9$, $\chi_{10}$ and
$\chi_{13}$. E.g. examining Figure~\ref{figure-36e} we see 
that 
$\Psi(\mathcal{O}_0\otimes \chi_5) = \mathcal{L}_{5}^{-1}(-D_4 -D_x)
\otimes \mathcal{O}_Z[1]$ for $Z=D_7\cup D_{10}$.
\end{example}

\subsection{The case of a ``meeting of champions''}
\label{section-case-of-a-meeting-of-champions}

Throughout this section let $\chi \in G^\vee$ be such that $\chi$ marks
a ''meeting of champions'' as depicted on Figure \ref{figure-24a}.
The configuration of edges of $\Sigma$ can 
can be viewed as three concave chains 
$e_{y}Pe_{z}$, $e_{x}Pe_{z}$ and $e_{x}Pe_{y}$. 
We show below that
$\mathcal{H}^{-1}\left(\Psi(\mathcal{O}_{0} \otimes \chi)\right)$
consists of the three answers Section 
\ref{section-case-of-a-single-concave-chain} would give
for each of these chains individually, glued together 
in a certain natural way. It is important to understand 
the geometry of the support of these three chains. 
In the next three paragraphs we establish some needed notation
and translate into concrete geometrical terms
the toric data on Figure \ref{figure-24a}. 

Denote by $Z_{x}$, $Z_{y}$ and $Z_{z}$ the chains of 
reduced divisors corresponding to the interior vertices of
$e_{y}Pe_{z}$, $e_{x}Pe_{z}$ and $e_{x}Pe_{y}$. 
The total support of $Z_{x}$, $Z_{y}$ and $Z_{z}$ consists 
of $P \simeq \mathbb{P}^2$ and the three chains of 
(possibly blownup) rational scrolls which link coordinate 
hyperplanes $E_x$, $E_y$ and $E_z$
to $P$. Each of $Z_{x}$, $Z_{y}$ and $Z_{z}$ consists of two 
of these chains of rational scrolls joined together by $P$  
(see Fig.\ \ref{figure-27}). Note that some of these chains
of rational scrolls may be empty, i.e.~$E_x$, $E_y$ or $E_z$ 
may intersect $P$ directly with no rational scrolls in between. 

Let $Q_x$, $Q_y$ and $Q_z$ denote the divisors on 
$e_xP$, $e_yP$ and $e_zP$ which are adjacent to $P$.
In other words, $Q_x$ is either the last scroll in  
the chain of rational scrolls joining $E_x$ to $P$ or $E_x$ itself
if the chain is empty. Then each of $Q_x$, $Q_y$ and $Q_z$ intersect
$P \simeq \mathbb{P}^2$ in a $\mathbb{P}^1$ passing through 
two out of the three toric fixed points of $P$. They also intersect 
each other in three $\mathbb{P}^1$s which are attached to $P$ at 
its toric fixed points. We've depicted this configuration on 
Fig.\ $\ref{figure-27a}$ in general case and on Fig.\ $\ref{figure-27b}$ 
when one of the rational scroll chains is empty. 

\begin{figure}[!htb] \centering 
\subfigure[General case] { \label{figure-27a}
\includegraphics[scale=0.11]{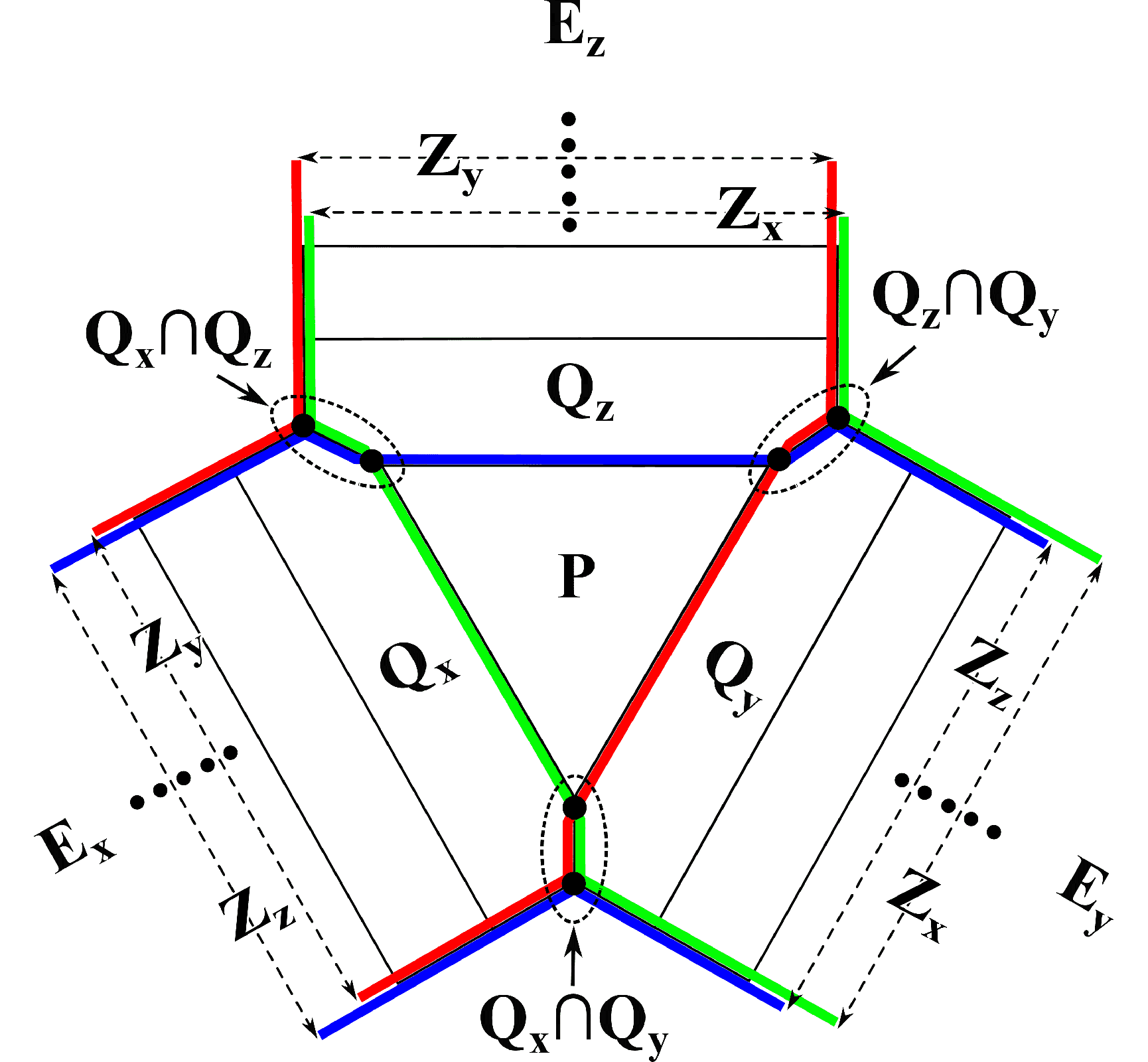}} 
\subfigure[Case $Q_z = E_z$] { \label{figure-27b}
\includegraphics[scale=0.27]{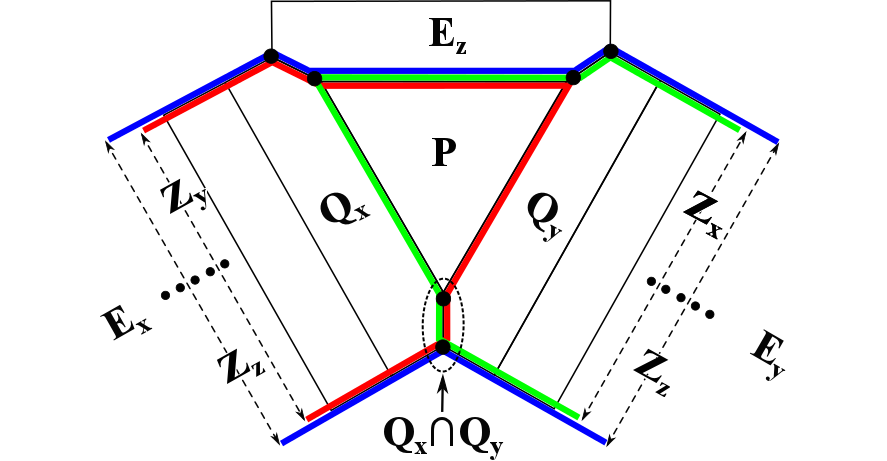}}
\caption{Divisor chains $Z_x$, $Z_y$ and $Z_z$
around divisor $P$}
\label{figure-27}
\end{figure}

Thus $S = Z_x \cup Z_y \cup Z_z$ has the following stratification:
\begin{itemize}
\item $S_2$: Points which lie on precisely two out of $Z_x$, $Z_y$ and 
$Z_z$. This is a disjoint union of the three rational scroll chains joining
$E_x$, $E_y$ and $E_z$ to $P$ minus their pairwise intersections. 
\item $S_3$: Points which lie on all three of $Z_x$, $Z_y$ and $Z_z$. 
This is a union of $P \simeq \mathbb{P}^2$ and $\mathbb{P}^1$s formed 
by pairwise intersections of those of $Q_x$, $Q_y$ and $Q_z$ which 
are rational scrolls. 
\end{itemize}

\begin{proposition}
\label{prps-chi-marks-a-meeting-of-champions}
\begin{enumerate}
\item 
\label{item-H-minus-one-for-the-meeting-of-champions}
We have
$$ \mathcal{H}^{-1}\left(\Psi(\mathcal{O}_{0} \otimes \chi)\right)
 = \mathcal{L}^{-1}_{\chi}(-E_x - E_y - E_z) \otimes \mathcal{F} $$
where $\mathcal{F}$ is the cokernel of 
\begin{align*}
\mathcal{O}_{Y}
\xrightarrow{E_{x} \oplus E_{y} \oplus E_{z}}
\mathcal{O}_{Z_{x}}(E_x) 
\oplus
\mathcal{O}_{Z_{y}}(E_y) 
\oplus
\mathcal{O}_{Z_{z}}(E_z).
\end{align*}
\item
\label{item-stratification-of-F}
The support of $\mathcal{F}$ is $S = Z_{x} \cup Z_{y} \cup Z_{z}$. Moreover:
\begin{itemize}
\item $\mathcal{F}|_{S_2}$ is a free sheaf of rank $1$. 
\item If $P$ doesn't intersect any of $E_x$, $E_y$ or $E_z$, then 
$\mathcal{F}|_{S_3}$ is a free sheaf of rank $2$. 
\item Otherwise, up to permuting $x$, $y$ and $z$:
\begin{itemize}
\item If $P$ intersects $E_x$, but not $E_y$ or $E_z$, 
then $\mathcal{F}|_{S_3} = \mathcal{O}(E_x) \oplus \mathcal{O}$. 
\item If $P$ intersects $E_x$ and $E_y$, but not $E_z$, 
then $\mathcal{F}|_{S_3} = \mathcal{O}(E_x) \oplus \mathcal{O}(E_y)$.
\item If $P$ intersects $E_x$, $E_y$ and $E_z$,
then $S_3 = P$ and $\mathcal{F}|_{P} = T_{\mathbb{P}^2}$. 
\end{itemize}
\end{itemize}
\end{enumerate}
\end{proposition}
\begin{proof}
We compute the sheaf $\mathcal{L}_\chi \otimes
\mathcal{H}^{-1}\left(\Psi(\mathcal{O}_{0} \otimes \chi)\right)$. 
Prop.\ \ref{prps-vertex-type-to-CT-subdivision-role-correspondence} 
and Figure \ref{figure-24a} together tell us 
the vertex type of $\chi^{-1}$ for every toric divisor of $Y$. 
We can therefore write down explicitly the
skew-commutative cube of line bundles corresponding to
$\mathcal{L}_\chi \otimes \hex(\chi^{-1})_{\widetilde{\mathcal{M}}}$:
\begin{align}
\label{eqn-cube-of-line-bundles-for-meeting-of-champions}
\vcenter{
\xymatrix{ 
& \mathcal{O}_Y(-Cx^\bullet) \ar"2,5"^<<{\alpha^3_{2}}
\ar"3,5"_<<{\alpha^2_{3}} & & & \mathcal{O}_Y(-E_x) \ar"2,6"^{\alpha^1} & 
\\
\mathcal{O}_Y(- \Delta) \ar"1,2"^{\alpha^1_{23}} \ar"2,2"^{\alpha^2_{13}}
\ar"3,2"_{\alpha^3_{12}} & \mathcal{O}_Y(-Cy^\bullet) 
\ar"1,5"^>>>>{\alpha^3_{1}} \ar"3,5"_>>>>>{\alpha^1_{3}} & & &
\mathcal{O}_Y(-E_y) \ar"2,6"^{\alpha^2} & \mathcal{O}_Y 
\\ 
& \mathcal{O}_Y(-Cz^{\bullet}) 
\ar"1,5"^<<{\alpha^2_{1}} \ar"2,5"_<<{\alpha^1_{2}} 
& & &
\mathcal{O}_Y(-E_z) \ar"2,6"_{\alpha^3} & } 
}
\end{align}
where maps $\alpha^\bullet_\bullet$ are given by natural inclusions. 
Denote by $T^\bullet$ the total complex of this cube. 

The maps $\alpha^1$, $\alpha^2$ and $\alpha^3$ are defined by
pairwise co-prime divisors $E_x$, $E_y$ and $E_z$. Arguing 
as in \cite[Lemma $3.1 (2)$]{CautisLogvinenko} we see that 
$\krn(T^{-1} \rightarrow T^0)$ is isomorphic to the
cokernel of 
\begin{align}
\label{eqn-kernel-of-T-1-to-T0}
\mathcal{O}_{Y}(-E_{x} - E_{y} - E_{z}) 
\xrightarrow{E_{x} \oplus E_{y} \oplus E_{z}}
\mathcal{O}_{Y}(-E_{y} - E_{z}) 
\oplus
\mathcal{O}_{Y}(-E_{x} - E_{z}) 
\oplus
\mathcal{O}_{Y}(-E_{x} - E_{y}) 
\end{align}
and that $\img(T^{-2} \rightarrow T^{-1})$ is then
the image in $\krn(T^{-1} \rightarrow T^0)$ of the subsheaf 
\begin{align}
\label{eqn-image-of-T-2-to-T-1}
\mathcal{O}_{Y}(-Cx^\bullet) 
\oplus
\mathcal{O}_{Y}(-Cy^\bullet) 
\oplus
\mathcal{O}_{Y}(-Cz^\bullet)
\end{align}
of 
\begin{align}
\label{eqn-kernel-of-T-1-to-T0-without-the-relations}
\mathcal{O}_{Y}(-E_{y} - E_{z}) 
\oplus
\mathcal{O}_{Y}(-E_{x} - E_{z}) 
\oplus
\mathcal{O}_{Y}(-E_{x} - E_{y}).
\end{align}
So $\mathcal{H}^{-1}\left(T^\bullet\right)$ is the quotient
of the cokernel of \eqref{eqn-kernel-of-T-1-to-T0} by the
image of \eqref{eqn-image-of-T-2-to-T-1}. On the other hand, 
claim \eqref{item-H-minus-one-for-the-meeting-of-champions}
is equivalent to $\mathcal{H}^{-1}\left(T^\bullet\right)$ being
the quotient of the cokernel of \eqref{eqn-kernel-of-T-1-to-T0}
by the image of 
\begin{align}
\label{eqn-image-of-T-2-to-T-1-reduced}
\mathcal{O}_{Y}(-e_{y}Pe_{z}) 
\oplus
\mathcal{O}_{Y}(-e_{x}Pe_{z}) 
\oplus
\mathcal{O}_{Y}(-e_{x}Pe_{y}).
\end{align}
This is because $e_y P e_z = E_y + E_z + Z_x$ as Weil divisors, 
and similarly for $-e_{x}Pe_{z}$ and $-e_{x}Pe_{y}$. 

To establish claim \eqref{item-H-minus-one-for-the-meeting-of-champions}
it now suffices to show that the natural inclusion 
of \eqref{eqn-image-of-T-2-to-T-1} into 
\eqref{eqn-image-of-T-2-to-T-1-reduced} becomes 
an isomorphism in the cokernel of
\eqref{eqn-kernel-of-T-1-to-T0}. This is a local problem, so 
take any basic triangle $\sigma$ in $\Sigma$ and let $A_\sigma$
be the corresponding affine chart. Since 
$Cx^\bullet$, $Cy^\bullet$ and $Cz^\bullet$ subdivide $\Delta$
triangle $\sigma$ must lie in precisely one of the three. Suppose, 
without loss of generality, it lies in $Cx^\bullet$. But then $E_x$
and $A_\sigma$ do not intersect, so \eqref{eqn-kernel-of-T-1-to-T0} 
restricts to $A_\sigma$ as
$$ \mathcal{O}(- E_{y} - E_{z}) 
\xrightarrow{\id \oplus E_{y} \oplus E_{z}}
\mathcal{O}(-E_{y} - E_{z}) 
\oplus
\mathcal{O}(- E_{z}) 
\oplus
\mathcal{O}(- E_{y}) 
$$ 
and its cokernel can therefore be identified with 
$\mathcal{O}(- E_{z}) \oplus \mathcal{O}(- E_{y})$. 
Under this identification the images of \eqref{eqn-image-of-T-2-to-T-1}
and \eqref{eqn-image-of-T-2-to-T-1-reduced} become
subsheaves of $\mathcal{O}(- E_{z}) \oplus \mathcal{O}(- E_{y})$
generated by
\begin{align}
\label{eqn-image-of-T-2-to-T-1-under-identification}
\img\left(\mathcal{O}(-Cx^\bullet) \hookrightarrow 
\mathcal{O}(- E_{z}) \oplus \mathcal{O}(- E_{y})\right), &
\quad
\mathcal{O}(-Cy^\bullet) 
\oplus
\mathcal{O}(-Cz^\bullet)\\
\label{eqn-image-of-T-2-to-T-1-reduced-under-identification}
\img(\mathcal{O}(-e_{y}Pe_{z})
\hookrightarrow 
\mathcal{O}(- E_{z}) \oplus \mathcal{O}(- E_{y})), &
\quad
\mathcal{O}(-e_{x}Pe_{z}) 
\oplus
\mathcal{O}(-e_{x}Pe_{y})
\end{align}
respectively.
On the other hand, since $\sigma \in Cx^\bullet$ we have on $A_\sigma$
\begin{align*}
\mathcal{O}(- Cy^\bullet) = 
\mathcal{O}(- e_xPe_z) = 
\mathcal{O}(- e_zP)  
\\
\mathcal{O}(- Cz^\bullet) = 
\mathcal{O}(- e_xPe_y) = 
\mathcal{O}(- e_yP)  
\end{align*}
and therefore 
$$ \mathcal{O}(-Cy^\bullet) 
\oplus
\mathcal{O}(-Cz^\bullet) =
\mathcal{O}(-e_{x}Pe_{z}) 
\oplus
\mathcal{O}(-e_{x}Pe_{y}) =
\mathcal{O}(- E_{z}) \oplus \mathcal{O}(- E_{y}).
$$
Lines $e_yP$ and $e_zP$ are both contained in 
$e_y P e_z$ and therefore in $Cx^\bullet$. Hence 
images of $\mathcal{O}(-Cx^\bullet)$ and $\mathcal{O}(-e_{y}Pe_{z})$
in $\mathcal{O}(- E_{z}) \oplus \mathcal{O}(- E_{y})$ are both 
contained in its subsheaf $\mathcal{O}(- e_zP) \oplus
\mathcal{O}(-e_yP)$. They are therefore redundant in  
\eqref{eqn-image-of-T-2-to-T-1-under-identification} and
\eqref{eqn-image-of-T-2-to-T-1-reduced-under-identification}.
Thus after identifying 
the cokernel of \eqref{eqn-kernel-of-T-1-to-T0} with 
$\mathcal{O}(- E_{z}) \oplus \mathcal{O}(- E_{y})$ 
the image of the natural inclusion of \eqref{eqn-image-of-T-2-to-T-1} into 
\eqref{eqn-image-of-T-2-to-T-1-reduced} becomes
$$ \mathcal{O}(- e_zP) \oplus
\mathcal{O}(- e_yP) \xrightarrow{\id} \mathcal{O}(- e_zP) \oplus
\mathcal{O}(- e_yP). $$ 
This establishes claim
\eqref{item-H-minus-one-for-the-meeting-of-champions}. 

For claim \eqref{item-stratification-of-F} first note that 
$S_2$ is a union of three disjoint pieces $S \setminus Z_x$, 
$S \setminus Z_y$ and $S \setminus Z_z$. 
Piece $S \setminus Z_z$ is the rational scroll chain joining
$E_z$ to $P$, so it is contained within $Z_x$ and $Z_y$
and it is disjoint from $E_x$ and $E_y$.  
Therefore $\mathcal{F}_{S \setminus Z_z}$ is the cokernel of 
$$ \mathcal{O}
\xrightarrow{\id \oplus \id \oplus 0}
\mathcal{O}
\oplus
\mathcal{O} 
\oplus
0, $$
which is isomorphic to $\mathcal{O}$, as required.
$F|_{S \setminus Z_x}$ and $\mathcal{F}|_{S \setminus
Z_y}$ are computed similarly. 

Next we note that $S_3$ intersects $E_x$, $E_y$ and $E_z$ if 
and only if $P$ does. To compute 
the restriction of $\mathcal{F}$ to $S_3$, suppose
first that $P$ doesn't intersect one of $E_x$, $E_y$ or $E_z$. 
Let it be $E_z$, then $\mathcal{F}|_{S_3}$ is the cokernel of 
$$ \mathcal{O} \xrightarrow{E_x \oplus E_y \oplus \id}
\mathcal{O}(E_x) \oplus \mathcal{O}(E_y) \oplus \mathcal{O}, $$
which is isomorphic to $\mathcal{O}(E_x) \oplus \mathcal{O}(E_y)$,
as required. 

Finally, suppose $S_3$ intersects $E_x$, $E_y$ and $E_z$.  
This means that all the rational scroll chains are empty and 
$S_3$ is just $P \simeq \mathbb{P}^2$. 
We then have a short exact sequence
$$ 
0 \longrightarrow \mathcal{O}
\longrightarrow \mathcal{O}(1)^{\oplus 3}
\longrightarrow \mathcal{F} \longrightarrow 0. $$
Dualizing Euler exact sequence 
$$ 
0 \longrightarrow \Omega_{\mathbb{P}^2}
\longrightarrow \mathcal{O}(-1)^{\oplus 3}
\longrightarrow \mathcal{O} 
\longrightarrow 0
$$
we see that $\mathcal{F}$ is isomorphic to
$\Omega^*_{\mathbb{P}^2}$, i.e. to $T_{\mathbb{P}^2}$. 
\end{proof}

\subsection{The case of a ``long side''} 
\label{section-case-of-a-long-side}

Suppose that in Reid's recipe $\chi$ marks a single straight
chain of edges running from one of the vertices of $\Delta$
to the opposite side. We may assume, without loss of generality,
that it is as depicted on Figure \ref{figure-24b}. 

The expression for $\Psi(\mathcal{O}_0 \otimes \chi)$ we obtain 
below is the same as in Section 
\ref{section-case-of-a-single-concave-chain}, but the proof is 
quite different. In the setup of Section 
\ref{section-case-of-a-single-concave-chain} it was always to choose
the filtration in Lemma \ref{lemma-cube_cohomology} so that all 
the quotients vanish but one. Lemma \ref{lemma-cube_cohomology}
then implied that the remaining quotient is isomorphic to  
$\Psi(\mathcal{O}_0 \otimes \chi)$. Here, for any choice of filtration 
Lemma \ref{lemma-cube_cohomology} only gives us two non-zero quotients 
and says that $\Psi(\mathcal{O}_0 \otimes \chi)$ is some extension of them.
We therefore have to analyze the skew-commutative cube of line bundles 
$\mathcal{L}_\chi \otimes \hex(\chi^{-1})_{\widetilde{\mathcal{M}}}$
directly, as in Section \ref{section-case-of-a-meeting-of-champions}. 

\begin{proposition}
\label{prps-chi-marks-a-long-side}
\begin{enumerate}
\item
\label{item-longside-single-edge}
If the chain is a single edge $e_x P$, i.e.
$G = \frac{1}{r}(0,1,r-1)$, then
$$ \Psi(\mathcal{O}_0 \otimes \chi)) = \mathcal{L}^{-1}_\chi \otimes
\mathcal{O}_{C} $$
where $C$ is the toric curve $E_x \cap P$ corresponding to that edge. 
\item
\label{item-longside-multiple-edge}
If the chain consists of more than one edge, then 
$$ \Psi(\mathcal{O}_0 \otimes \chi) =  
\mathcal{L}^{-1}_{\chi}(-E_x - P) \otimes \O_Z [1] $$
where $Z$ is the union of the divisors which correspond
to the internal vertices of $e_x P$.
\end{enumerate}
\end{proposition}
\begin{proof}
Claim \eqref{item-longside-single-edge} is proved
as in Prop.\ \ref{prps-chi-marks-a-single-concave-chain}. 
For claim \eqref{item-longside-multiple-edge}
we proceed as in
Section \ref{section-case-of-a-meeting-of-champions}
and compute $\mathcal{L}_\chi \otimes 
\mathcal{H}^{-1}\left(\Psi(\mathcal{O}_{0} \otimes \chi)\right)$.
Using Prop.\ \ref{prps-vertex-type-to-CT-subdivision-role-correspondence} 
and Figure \ref{figure-24b} we write down explicitly the
skew-commutative cube of line bundles corresponding to
$\mathcal{L}_\chi \otimes \hex(\chi^{-1})_{\widetilde{\mathcal{M}}}$:
\begin{align}
\label{eqn-cube-of-line-bundles-for-a-long-side}
\xymatrix{ 
& \mathcal{O}_Y(-e_z e_y) 
\ar"2,5"^<<{\alpha^3_{2}}
\ar"3,5"_<<{\alpha^2_{3}} & & & \mathcal{O}_Y(-E_x) \ar"2,6"^{\alpha^1} & 
\\
\mathcal{O}_Y(- \Delta) \ar"1,2"^{\alpha^1_{23}} \ar"2,2"^{\alpha^2_{13}}
\ar"3,2"_{\alpha^3_{12}} & \mathcal{O}_Y(-Cy^\bullet) 
\ar"1,5"^>>>>{\alpha^3_{1}} \ar"3,5"_>>>>>{\alpha^1_{3}} & & &
\mathcal{O}_Y(-Tx^\bullet z^\bullet) \ar"2,6"^>>>>>>{\alpha^2} & \mathcal{O}_Y 
\\ 
& \mathcal{O}_Y(-Cz^{\bullet}) 
\ar"1,5"^<<{\alpha^2_{1}} \ar"2,5"_<<{\alpha^1_{2}} 
& & &
\mathcal{O}_Y(-Tx^\bullet y^\bullet) \ar"2,6"_{\alpha^3} & } 
\end{align}
where maps $\alpha^\bullet_\bullet$ are given by natural inclusions. 

Denote by $T^\bullet$ the total complex of this cube. 
Arguing as in the proof of \cite[Lemma $3.1$]{CautisLogvinenko} we see 
that $\krn(T^{-1} \rightarrow T^0)$ is isomorphic to the
cokernel of 
\begin{align*}
%\label{eqn-longside-kernel-of-T-1-to-T0}
\mathcal{O}_{Y}(-E_{x} - e_y e_z) 
\xrightarrow{E_{x} \oplus (Tx^\bullet z^\bullet \setminus P) \oplus (Tx^\bullet y^\bullet \setminus P)}
\mathcal{O}_{Y}(-e_y e_z) \oplus \mathcal{O}_{Y}(-E_{x} - Tx^\bullet y^\bullet) \oplus \mathcal{O}_{Y}(-E_{x} - Tx^\bullet z^\bullet) 
\end{align*}
and that $\img(T^{-2} \rightarrow T^{-1})$ is then
the image in $\krn(T^{-1} \rightarrow T^0)$ of the subsheaf 
\begin{align*}
%\label{eqn-longside-image-of-T-2-to-T-1}
\mathcal{O}_{Y}(-e_y e_z) 
\oplus
\mathcal{O}_{Y}(-Cy^\bullet) 
\oplus
\mathcal{O}_{Y}(-Cz^\bullet).
\end{align*}
Let $Z_y$ and $Z_z$ be the reduced divisors 
of $Cy^\bullet \setminus (e_x \cup Tx^\bullet y^\bullet)$
and $Cz^\bullet \setminus (e_x \cup Tx^\bullet z^\bullet)$, 
then $\mathcal{H}^{-1}(T^\bullet)$ is the cokernel of 
\begin{align}
\label{longside-degree-minus-cohomology-cokernel-exression}
\mathcal{O}_{Y}(-E_{x} - e_y e_z) 
\xrightarrow{(Tx^\bullet z^\bullet \setminus P) \oplus (Tx^\bullet y^\bullet \setminus P)} \mathcal{O}_{Z_y}(-E_{x} - Tx^\bullet y^\bullet) 
\oplus
\mathcal{O}_{Z_z}(-E_{x} - Tx^\bullet z^\bullet). 
\end{align}
Observe that $Z = Z_x \cap Z_y$, so we have the natural map
\begin{align}
\label{longside-degree-minus-cohomology-natural-map}
\mathcal{O}_{Z_y}(-E_{x} - Tx^\bullet y^\bullet) 
\oplus
\mathcal{O}_{Z_z}(-E_{x} - Tx^\bullet z^\bullet)
\xrightarrow{
(Tx^\bullet y^\bullet \setminus P) \oplus (Tx^\bullet z^\bullet \setminus P)}
\mathcal{O}_Z(-E_x - P). 
\end{align}
It remains to show that the composition of 
\eqref{longside-degree-minus-cohomology-cokernel-exression}
and 
\eqref{longside-degree-minus-cohomology-natural-map}
is exact in last two terms. We check it on an open cover: 
on the open piece of $Y$ which corresponds to triangle $e_x e_y P$ 
it is
$$
\mathcal{O}_{Y}(-E_{x} - Tx^\bullet y^\bullet) 
\xrightarrow{1 \oplus (Tx^\bullet y^\bullet \setminus P)}
\mathcal{O}_{Z_y}(-E_{x} - Tx^\bullet y^\bullet) 
\oplus
\mathcal{O}_{Z_y \cap Z_z}(-E_{x} - P)
\xrightarrow{(Tx^\bullet y^\bullet \setminus P) \oplus 1}
\mathcal{O}_{Z_y \cap Z_z}(-E_x - P)
$$
and on the open piece of $Y$ which corresponds to triangle $e_x e_z P$ 
it is
$$
\mathcal{O}_{Y}(-E_{x} - Tx^\bullet z^\bullet) 
\xrightarrow{(Tx^\bullet z^\bullet \setminus P) \oplus 1}
\mathcal{O}_{Z_y \cap Z_z}(-E_{x} - P) 
\oplus
\mathcal{O}_{Z_z}(-E_{x} - Tx^\bullet y^\bullet)
\xrightarrow{1 \oplus (Tx^\bullet z^\bullet \setminus P)}
\mathcal{O}_{Z_y \cap Z_z}(-E_x - P). 
$$
Both sequences are manifestly exact in their last two terms. 
\end{proof}

\begin{example}
In the worked example in Section~\ref{section-worked-examples}, 
cf. Figure~\ref{figure-34b}, 
Prop.\ \ref{prps-chi-marks-a-single-concave-chain}\eqref{item-single-edge-chain}
applies to the character $\chi_{12}$. E.g. examining 
Figure~\ref{figure-36m} we see 
that $\Psi(\mathcal{O}_0\otimes \chi_{12}) = \mathcal{L}_{12}^{-1}(-D_y -D_9)
\otimes \mathcal{O}_Z[1]$ for $Z=D_5\cup D_7$.
\end{example}

\subsection{The case of $\chi = \chi_0$} 
\label{section-case-of-chi=chi_0}

Suppose that $\chi$ marks nothing in Reid's recipe, i.e. $\chi = \chi_0$.
Denote by $F_{23}$, $F_{13}$ and $F_{12}$ 
the chains of divisors corresponding to internal vertices of 
sides $e_y e_z$, $e_x e_z$ and $e_x e_y$, respectively.  
Recall that $\zerofibre$ denotes the fibre of $Y$ over 
$0 \in \mathbb{C}^3/G$, cf. \S\ref{section-ghilb-and-toric}
and $\zerofibre_1$ and $\zerofibre_2$ denote the unions of
$1$- and $2$-dimensional irreducible components of $\zerofibre$.  
For any separated scheme $S$ of finite type over $\mathbb{C}$
we denote by $\boldomega_S$ the dualizing complex of $S$ obtained
as the twisted inverse image $f^!(\mathbb{C})$ over
the structure morphism $S \xrightarrow{f} \spec \mathbb{C}$.
If $S$ is Cohen-Macaulay, then $\boldomega_S = \omega_S[\dim S]$
where $\omega_S$ is a sheaf on $S$. If $S$ is proper, then 
$\omega_S$ is the dualizing sheaf in the sense of 
\cite[\S III.7]{Harts77}. If $S$ is regular, 
then $\omega_S$ is the canonical bundle
\cite[Theorem 3]{Verdier-BaseChangeForTwistedInverseImageOfCoherentSheaves}. 

\begin{proposition}
Let $\chi$ be the trivial character $\chi_0$. Then
$\Psi(\mathcal{O}_0 \otimes \chi) = \boldomega_{\zerofibre}$ and
moreover:
\begin{enumerate}
\item 
\label{item-chitriv-degree-minus-2-cohomology}
$\mathcal{H}^{-2}\left(\Psi(\mathcal{O}_0 \otimes \chi)\right)
= \omega_{\zerofibre_2} = 
\mathcal{O}_Y(\zerofibre_2) \otimes \mathcal{O}_{\zerofibre_2}$. 
\item \label{item-chitriv-degree-minus-1-cohomology}
$\mathcal{H}^{-1}\left(\Psi(\mathcal{O}_0 \otimes \chi)\right)
= \omega_{\zerofibre_1}(\zerofibre_2)$.
\item \label{item-chitriv-other-degrees-cohomology}
$\mathcal{H}^{i}\left(\Psi(\mathcal{O}_0 \otimes \chi)\right)
= 0$ for $i \neq -1,-2$. 
\end{enumerate}
\end{proposition}
\begin{proof}
We use the following trick: instead of writing down 
the skew-commutative cube of line bundles corresponding to
$\hex(\chi_0)_{\widetilde{\mathcal{M}}}$, we write down its 
derived dual. The total complex of this dual cube
is the derived dual of the total complex of
the original cube, i.e. of 
$\Psi(\mathcal{O}_0 \otimes \chi)$. 

Theorem \ref{theorem-sink-source-graph-to-divisor-type-correspondence}
tells us the vertex type of $\chi_0$ for every toric divisor on $Y$,
using this we write down the cube corresponding
to $\hex(\chi_0)_{\widetilde{\mathcal{M}}}$ as:
\begin{align}
\label{eqn-cube-of-line-bundles-for-trivial-character}
\vcenter{
\xymatrix{ 
& \mathcal{O}_Y(-E_y - F_{23} - E_z) \ar"2,5"
\ar"3,5" & & & \mathcal{O}_Y(-E_x) \ar"2,6" & 
\\
\mathcal{O}_Y(- \Delta) \ar"1,2"
\ar"2,2" \ar"3,2" & \mathcal{O}_Y(-E_x - F_{13} - E_z) 
\ar"1,5" \ar"3,5" & & &
\mathcal{O}_Y(-E_y) \ar"2,6" & \underline{\mathcal{O}_Y}
\\ 
& \mathcal{O}_Y(-E_x - F_{12} - E_y) 
\ar"1,5" \ar"2,5" 
& & &
\mathcal{O}_Y(-E_z) \ar"2,6" & 
} },
\end{align}
where we underline the degree $0$ term. Its derived dual is:
\begin{align}
\label{eqn-dual-cube-of-line-bundles-for-trivial-character}
\vcenter{
\xymatrix{ 
& 
\mathcal{O}_Y(E_x) \ar"2,5"
\ar"3,5" 
& & & 
\mathcal{O}_Y(E_y + F_{23} + E_z) \ar"2,6" & 
\\
\underline{\mathcal{O}_Y} \ar"1,2"
\ar"2,2" \ar"3,2" &
\mathcal{O}_Y(E_y) 
\ar"1,5" \ar"3,5" 
& & &
\mathcal{O}_Y(E_x + F_{13} + E_z) \ar"2,6" & \mathcal{O}_Y(\Delta)
\\ 
& \mathcal{O}_Y(E_z) 
\ar"1,5" \ar"2,5" 
& & &
\mathcal{O}_Y(E_x + F_{12} + E_y) \ar"2,6" & 
} }. 
\end{align}
Denote by $T^\bullet$ the total complex of the dual cube
and denote by $\iota_{\zerofibre}$ the inclusion map 
$\zerofibre \hookrightarrow Y$.  
Lemma \ref{lemma-cube_cohomology} implies immediately that
$\mathcal{H}^0(T^\bullet) = \iota_{\zerofibre *} 
\mathcal{O}_{\zerofibre}$ and all the
other cohomologies of $T^\bullet$ vanish. We therefore 
have 
$$ \Psi(\mathcal{O}_0 \times \chi_0) \simeq (\iota_{\zerofibre *}
\mathcal{O}_{\zerofibre})^\vee[3].$$

The map $\iota_{\zerofibre}$ is proper, so $\iota^!_{\zerofibre}$ is 
right adjoint to $\iota_{\zerofibre *}$, and by
sheafified Grothendieck duality:
$$ 
(\iota_{\zerofibre *}
\mathcal{O}_{\zerofibre})^\vee = 
\rder \shhomm_Y \left(\iota_{\zerofibre *} \mathcal{O}_{\zerofibre}, 
\mathcal{O}_Y\right) 
\simeq 
\iota_{\zerofibre *} 
\rder \shhomm_{\zerofibre} \left( \mathcal{O}_{\zerofibre}, 
\iota_{\zerofibre}^! \mathcal{O}_Y \right) 
\simeq 
\iota_{\zerofibre *} \iota_{\zerofibre}^! \mathcal{O}_Y. $$
Let $f_{Y}$ and $f_{\zerofibre}$ denote the structural morphisms
to $\spec \mathbb{C}$. As $f_{\zerofibre} = f_Y \circ
\iota_{\zerofibre}$, we have
$$ 
\iota^!_{\zerofibre} \mathcal{O}_Y[3]
\simeq 
\iota^!_{\zerofibre} \boldomega_Y 
=
\iota^!_{\zerofibre} f^!_Y(\mathbb{C})
\simeq 
f^!_{\zerofibre}(\mathbb{C})
= \boldomega_{\zerofibre}.
$$ 
Thus we have $\Psi(\mathcal{O}_0 \times \chi_0) =
\iota_{\zerofibre *} \boldomega_{\zerofibre}$ or, 
returning to our convention of omitting the pushforward functor
where the source scheme is obvious, just 
$\Psi(\mathcal{O}_0 \times \chi_0) = \boldomega_{\zerofibre}$. 
This settles the first claim.

There is a short exact sequence of sheaves on $Y$
\begin{align}
\label{eqn-zf1-zf2-zf-ses}
\mathcal{O}_{\zerofibre_1}(- \zerofibre_2) \rightarrow 
\mathcal{O}_{\zerofibre} \rightarrow \mathcal{O}_{\zerofibre_2} 
\end{align}
which is an instance of the standard short exact sequence 
$$ \mathcal{O}_{Z_1} \otimes \mathcal{I}_{Z_2}
\rightarrow \mathcal{O}_{Z} \rightarrow \mathcal{O}_{Z_2} $$
which exists for any reduced scheme $Z$ which is 
a union of schemes $Z_1$ and $Z_2$. 

Hence \eqref{eqn-zf1-zf2-zf-ses} is an exact triangle in $D(Y)$. 
Taking the derived dual we obtain the triangle
$$ \boldomega_{\zerofibre_2} \rightarrow \boldomega_{\zerofibre}
\rightarrow \boldomega_{\zerofibre_1}(\zerofibre_2). $$
Whence the claims \eqref{item-chitriv-degree-minus-2-cohomology}-\eqref{item-chitriv-other-degrees-cohomology}
follow, since $\zerofibre_1$ and $\zerofibre_2$ are Cohen-Macaulay and 
their dualizing complexes are just the shifts of their dualizing sheaves. 
\end{proof}
 
\section{Worked example}
\label{section-worked-examples}

In this section we illustrate our results by explicit computations 
for $G = \frac{1}{15}(1,5,9)$. 

\subsection{The group $G = \frac{1}{15}(1,5,9)$, the toric variety 
$G$-Hilb$(\mathbb{C}^3)$ and classical Reid's recipe} 
\label{section-example-the-group-and-the-resolution}
 
Let $G$ be the group $\frac{1}{15}(1,5,9)$. It is the image 
in $\gsl_3(\mathbb{C})$ of group $\mu_{15}$ of $15$th roots of unity 
under the embedding
$\xi \mapsto 
\left( \begin{smallmatrix}
\xi^1 & & \\
& \xi^5 & \\
& & \xi^9
\end{smallmatrix} \right).$ 
We denote by $\chi_{i}$ the character of $G$ induced by $\xi \mapsto
\xi^i$, then, as explained in Section
\ref{section-prelims-action-of-G-on-C3},
 $\kappa(x_1) = \chi_{14}$, $\kappa(x_2) = \chi_{10}$ and
$\kappa(x_3) = \chi_6$.

Let $Y = G$-$\hilb(\mathbb{C}^3)$. We next describe $Y$ as 
a toric variety following Section \ref{section-ghilb-and-toric}.  
First we compute the toric fan of $Y$
as described in 
\cite[Section 2]{Craw-AnexplicitconstructionoftheMcKaycorrespondenceforAHilbC3}.
On Fig.\ \ref{figure-34a} we depict the triangulation 
$\Sigma$ of the junior simplex $\Delta$ defined by this fan, together
with the monomial ratios which carve out each edge of $\Sigma$. 

\begin{figure}[!htb] \centering 
\subfigure[Monomial ratios for the edges of $\Sigma$]{\label{figure-34a}
\includegraphics[scale=0.12]{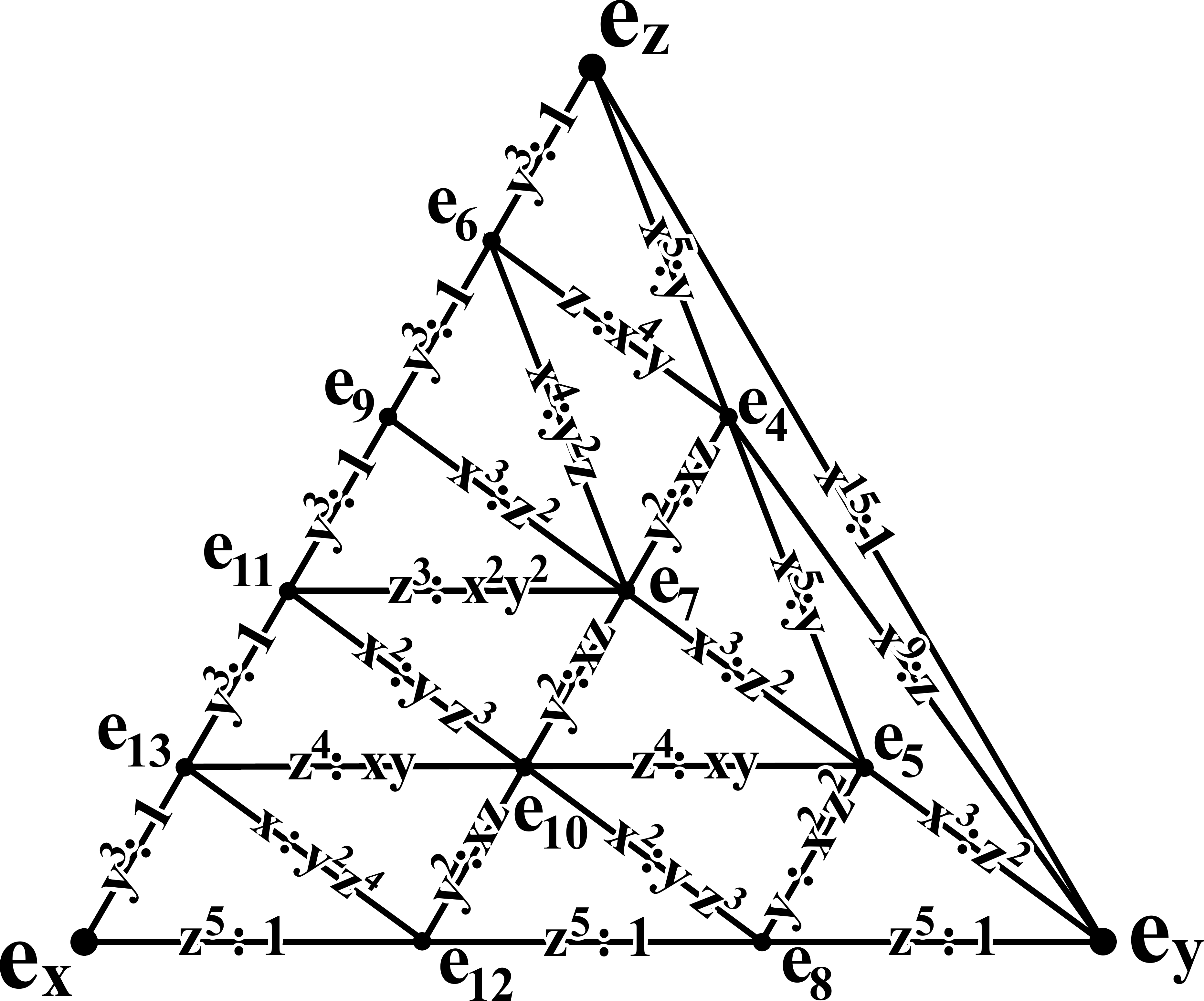}} 
\subfigure[Classical Reid's recipe]{\label{figure-34b}
\includegraphics[scale=0.12]{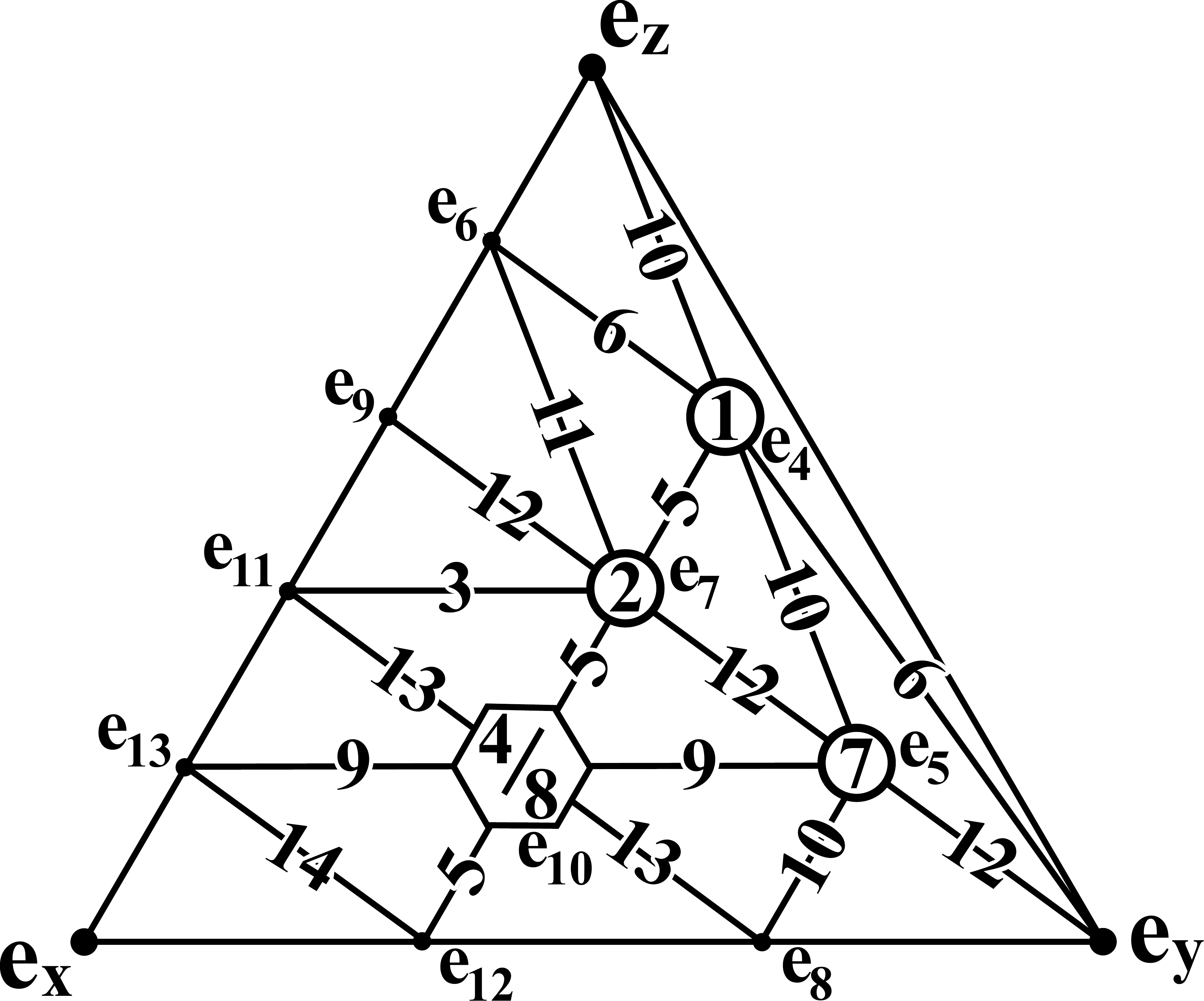}}
\caption{ $G$-$\hilb(\mathbb{C}^3)$ for $\frac{1}{15}(1,5,9)$ }
\label{figure-34}
\end{figure}

The generators $e_i$ of the one-dimensional cones of the fan have 
the following coordinates:
\begin{align} 
\label{eqn-e_i-15-1-5-9}
\begin{tabular}{l l l}
$e_1 = (1,0,0)$  &  $e_2 = (0,1,0)$  &  $e_3 = (0,0,1)$ \\
$e_4 = \frac{1}{15}(1,5,9)$  &  $e_5 = \frac{1}{15}(2,10,3)$ 
& $e_6 = \frac{1}{15}(3,0,12)$ \\
$e_7 = \frac{1}{15}(4,5,6)$  &  $e_8 = \frac{1}{15}(5,10,0)$ 
& $e_9 = \frac{1}{15}(6,0,9)$ \\
$e_{10} = \frac{1}{15}(7,5,3)$ & $ e_{11} = \frac{1}{15}(9,0,6)$
& $e_{12} = \frac{1}{15}(10,5,0)$ \\ 
$e_{13} = \frac{1}{15}(12,0,3)$ & & \\ 
\end{tabular} 
\end{align} 

The corresponding toric divisors are:
\begin{enumerate}
\item \em Strict transforms of coordinate hyperplanes of
$\mathbb{C}^3/G$: \rm 
\begin{itemize}
\item  $E_x$, $E_y$, $E_z$. 
\end{itemize}
\item \em Non-compact exceptional divisors: \rm 
\begin{itemize}
\item $E_9$, isomorphic to $\mathbb{P}^1 \times \mathbb{A}^1$. 
\item $E_6$, $E_8$, $E_{11}$, $E_{12}$ and $E_{13}$, each isomorphic 
to $\mathbb{P}^1 \times \mathbb{A}^1$ blown up in a point. 
\end{itemize}
\item \em Compact exceptional divisors: \rm 
\begin{itemize}
\item $E_4$ and $E_5$, each isomorphic to a rational scroll 
blown up in a point. 
\item $E_7$, isomorphic to a rational scroll blown up in two points. 
\item $E_8$, isomorphic to the Del Pezzo surface $dP_6$.  
\end{itemize}
\end{enumerate}

The fiber $\zerofibre$ of $Y$ over $0 \in \mathbb{C}^3/G$ is
a reducible variety which breaks up into
\begin{itemize}
\item \em Two-dimensional stratum $\zerofibre_2$: \rm
the compact exceptional divisors $E_4,E_5,E_7$ and $E_{10}$.
\item \em One-dimensional stratum $\zerofibre_1$: \rm
a single curve $E_{12} \cap E_{13} \simeq \mathbb{P}^1$. 
\end{itemize}

Finally, we compute classical Reid's recipe for $G = \frac{1}{15}(1,5,9)$,
as described in \S \ref{subsection-reids-recipe}, and list the result
on Fig.\ \ref{figure-34b}. 

\subsection{CT-subdivisions}

We now begin to compute derived Reid's recipe for
$G = \frac{1}{15}(1,5,9)$. 
The standard way to do this is via 
explicit computations with $G$-Weil divisors, cf. 
\cite[\S6]{CautisLogvinenko}. But Prop.\ 
\ref{prps-vertex-type-to-CT-subdivision-role-correspondence} 
allows for a new way to do this, which we illustrate below. 

The first step is to compute the CT-subdivisions of $\Delta$ for 
all the non-trivial characters of $G$. These are
defined in \S\ref{section-CT-subdivisions} in terms of the monomials
which represent $\chi$ in the $G$-graphs of the basic triangles of
$\Sigma$. These can be computed as in 
\cite[\S5]{Craw-AnexplicitconstructionoftheMcKaycorrespondenceforAHilbC3}. 

On Fig.\ \ref{figure-36}-\ref{figure-38} 
we display the resulting CT-subdivisions together with  
the above-mentioned monomials. By doing so we are explicitly writing down 
tautological sheaves $\mathcal{L}_\chi$: each $\mathcal{L}_\chi$
is the subsheaf of constant sheaf $K(\mathbb{C}^3)$ 
which is generated on each toric affine piece of $Y$ by the
monomial which represents $\chi$ in the graph of 
the corresponding basic triangle of $\Sigma$.  

\begin{figure}[!htb] \centering 
\subfigure[$\chi_1$] { \label{figure-36a}
\includegraphics[scale=0.11]{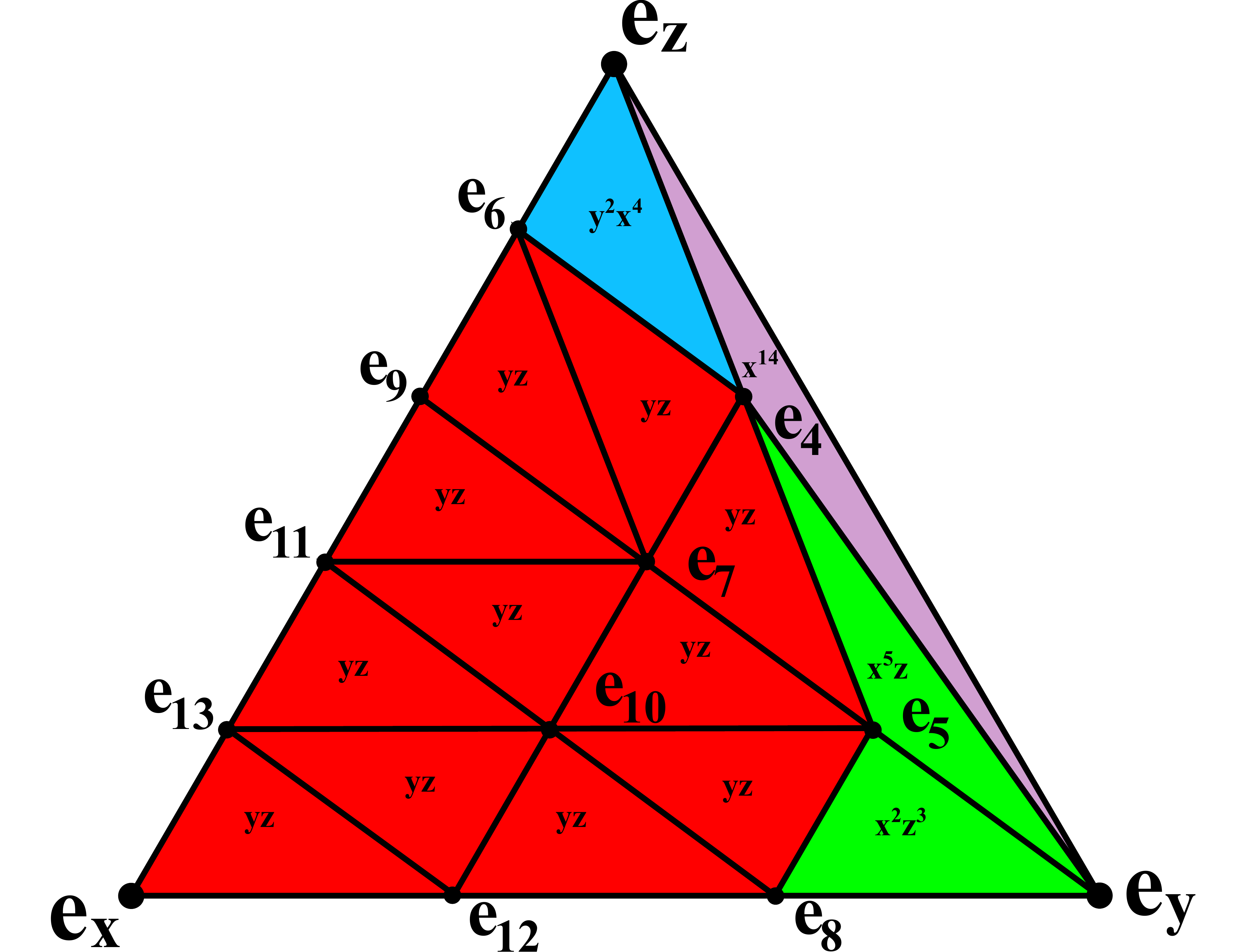}} 
\subfigure[$\chi_2$] { \label{figure-36b}
\includegraphics[scale=0.11]{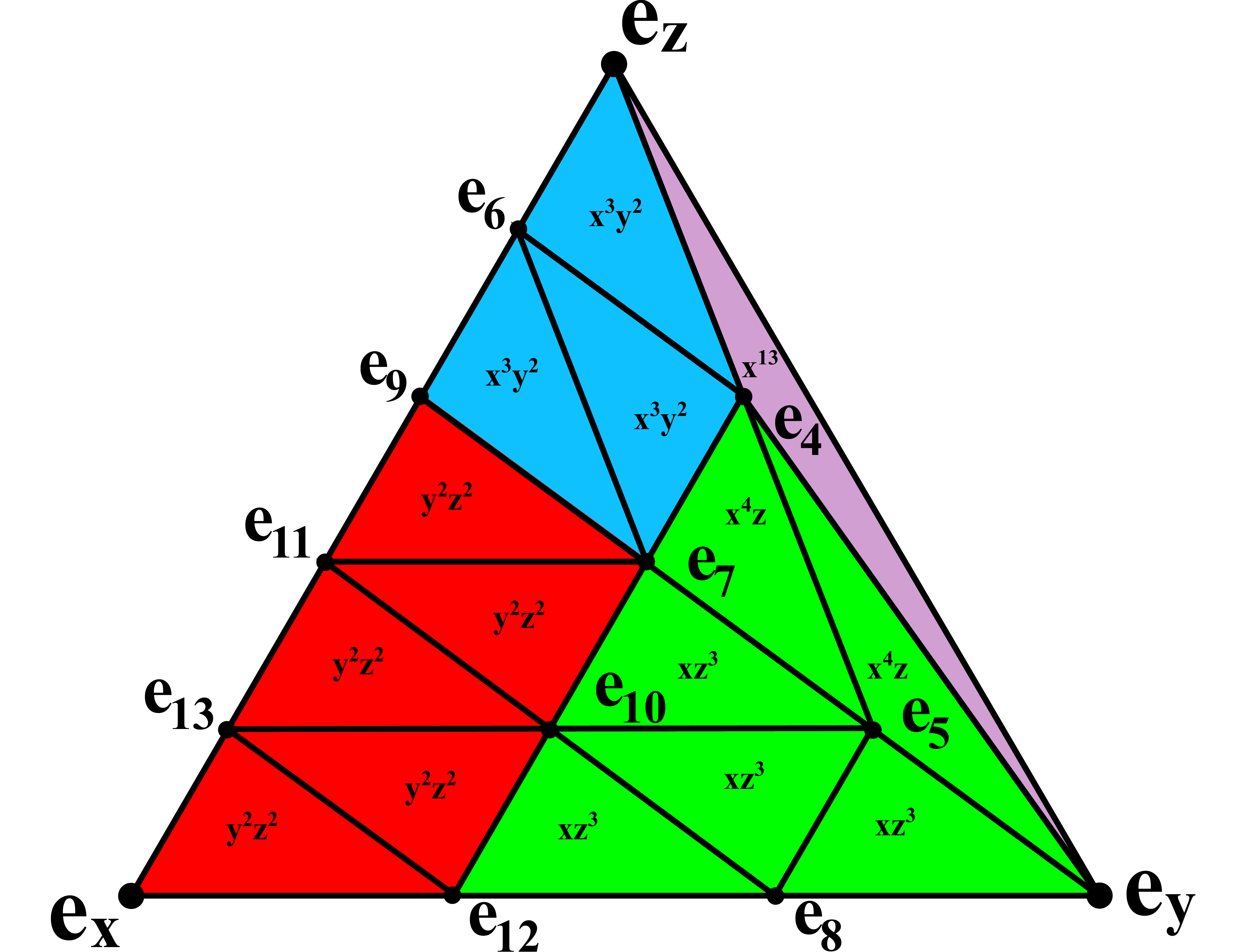}} 
\subfigure[$\chi_3$] { \label{figure-36c}
\includegraphics[scale=0.11]{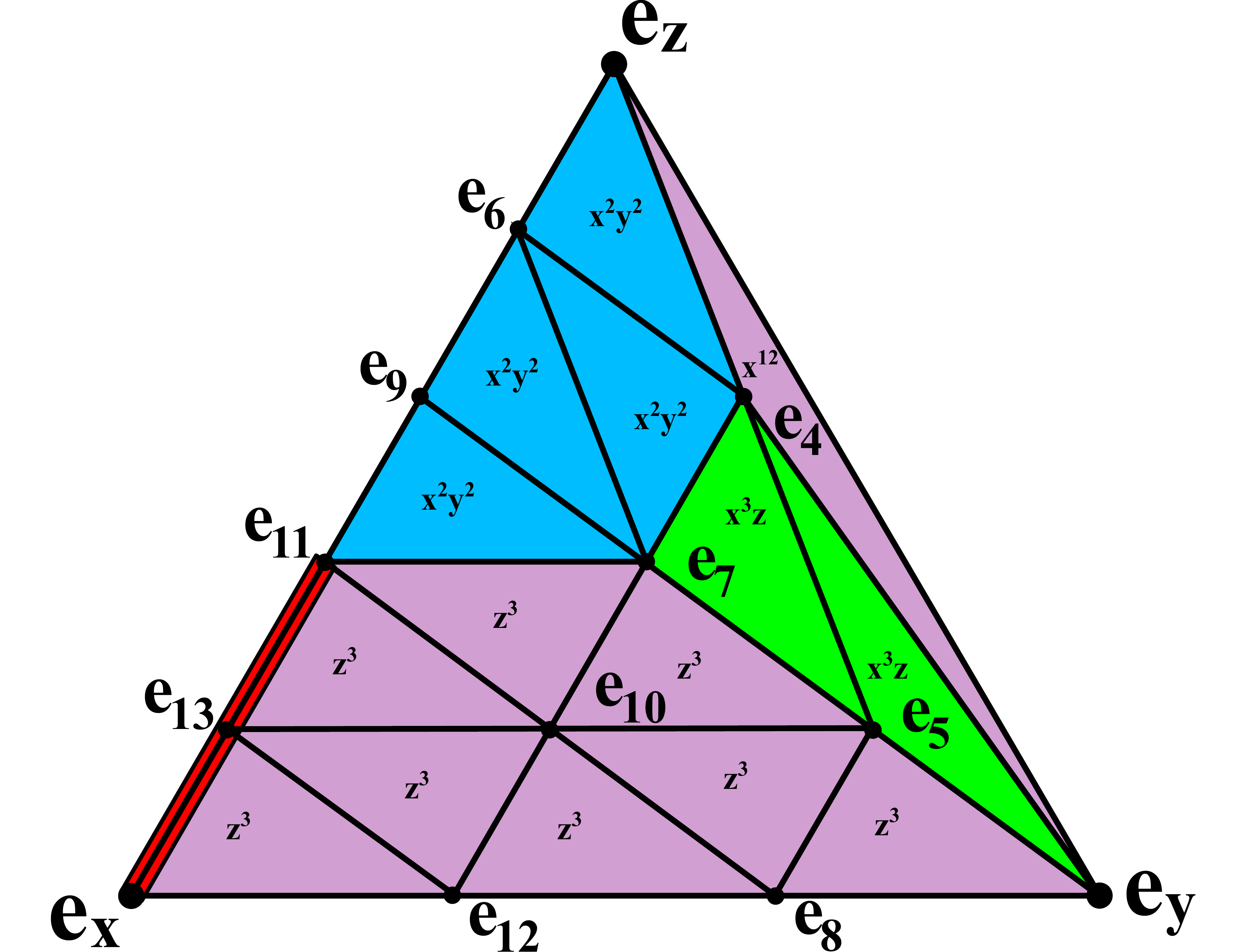}}
\subfigure[$\chi_4$] { \label{figure-36d}
\includegraphics[scale=0.11]{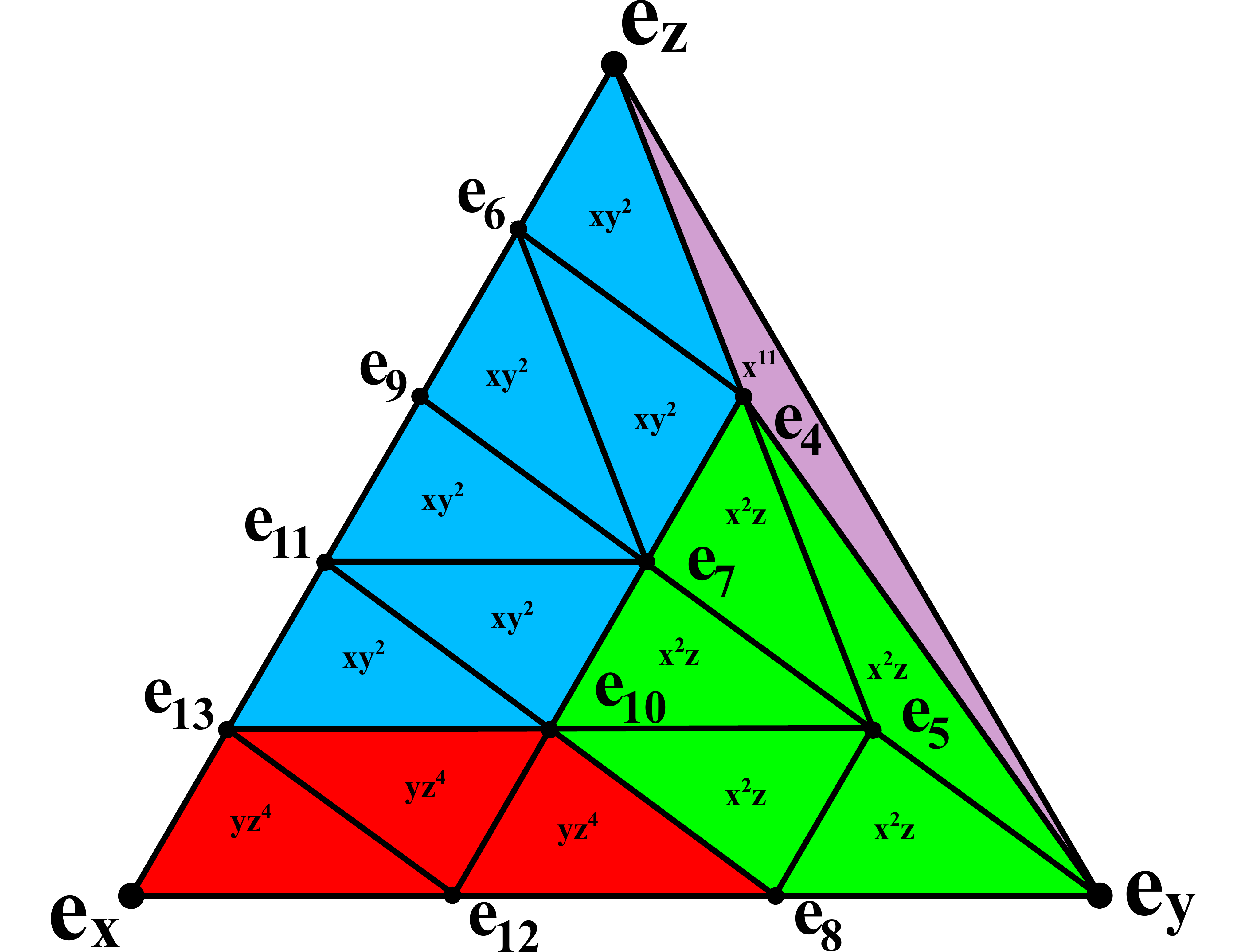}}
\subfigure[$\chi_5$] { \label{figure-36e}
\includegraphics[scale=0.11]{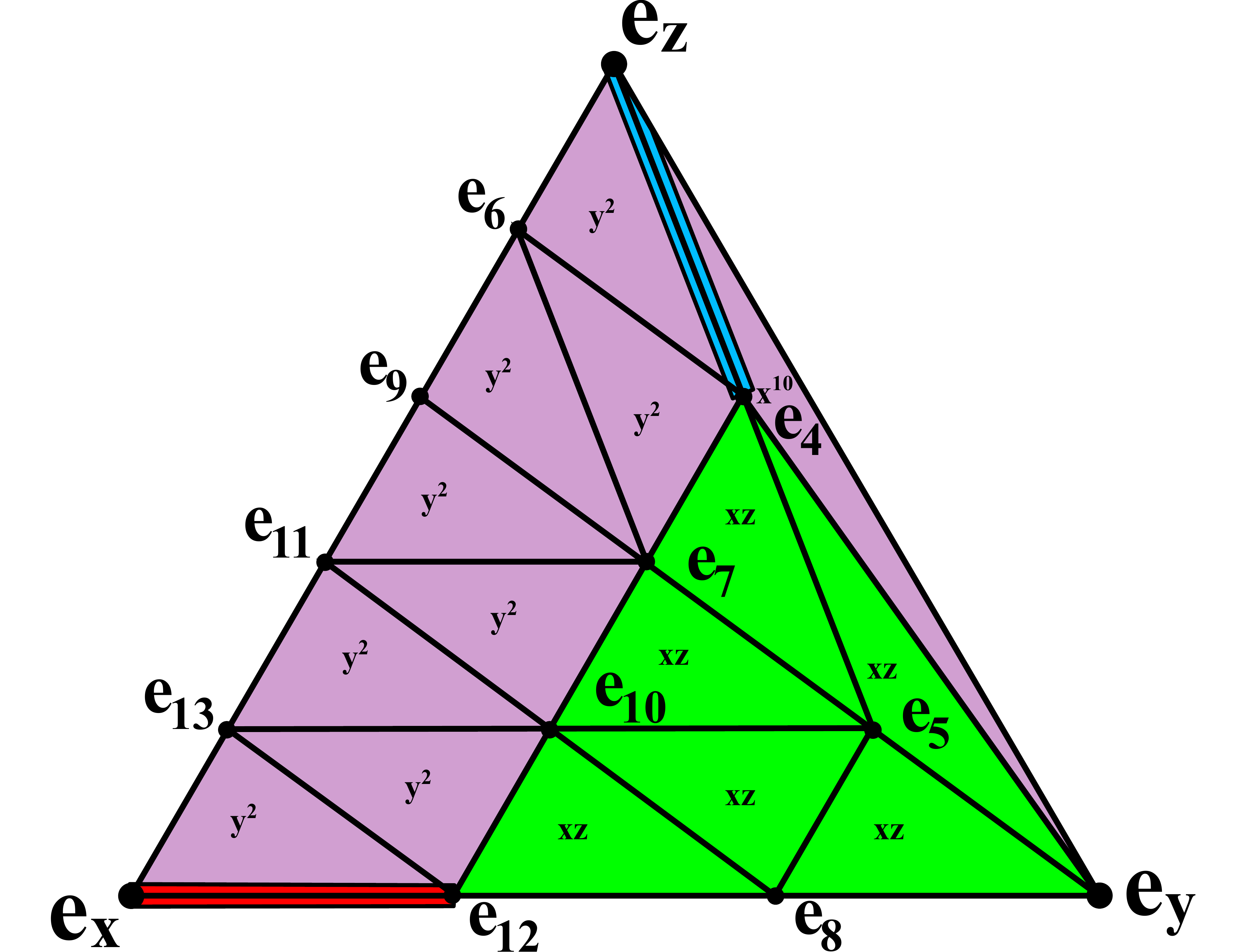}}
\subfigure[$\chi_6$] { \label{figure-36f}
\includegraphics[scale=0.11]{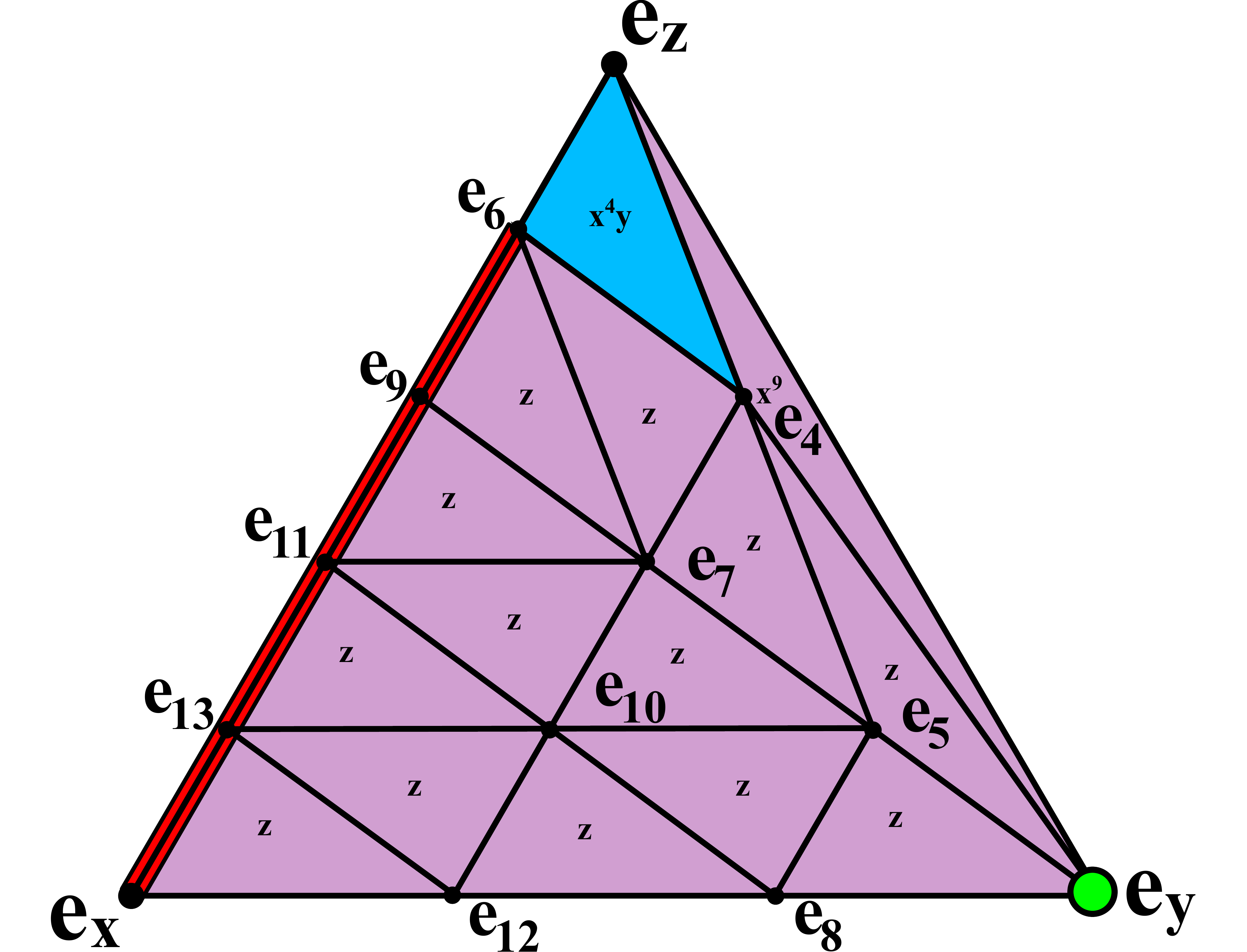}}
\caption{\label{figure-36} CT-subdivisions and monomial generators of $\mathcal{L}_\chi$
for $\chi_1$ - $\chi_6$. }
\end{figure}
\newpage~\newpage
\begin{figure}[!htb] \centering 
\subfigure[$\chi_7$] { \label{figure-36g}
\includegraphics[scale=0.11]{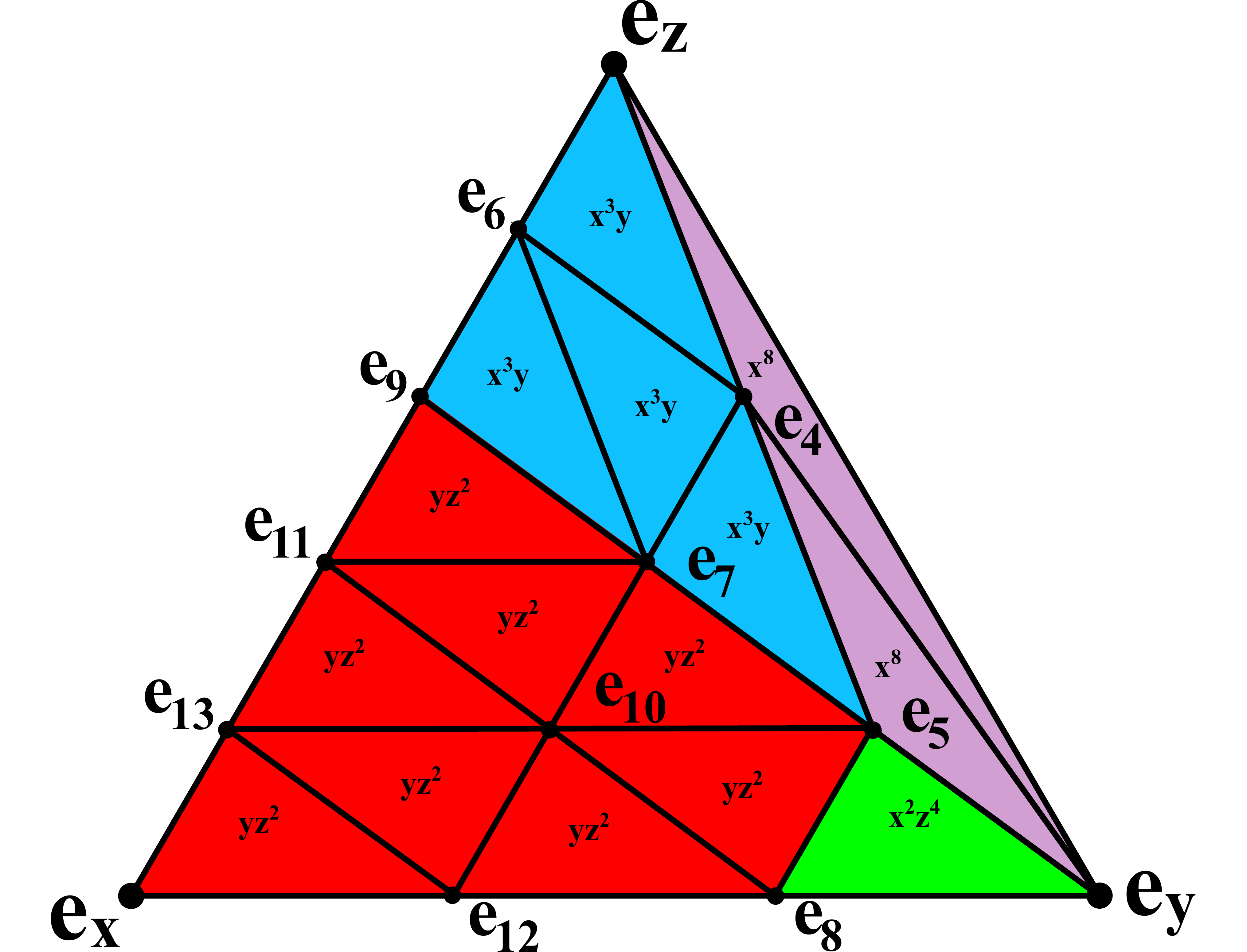}}
\subfigure[$\chi_8$] { \label{figure-36h}
\includegraphics[scale=0.11]{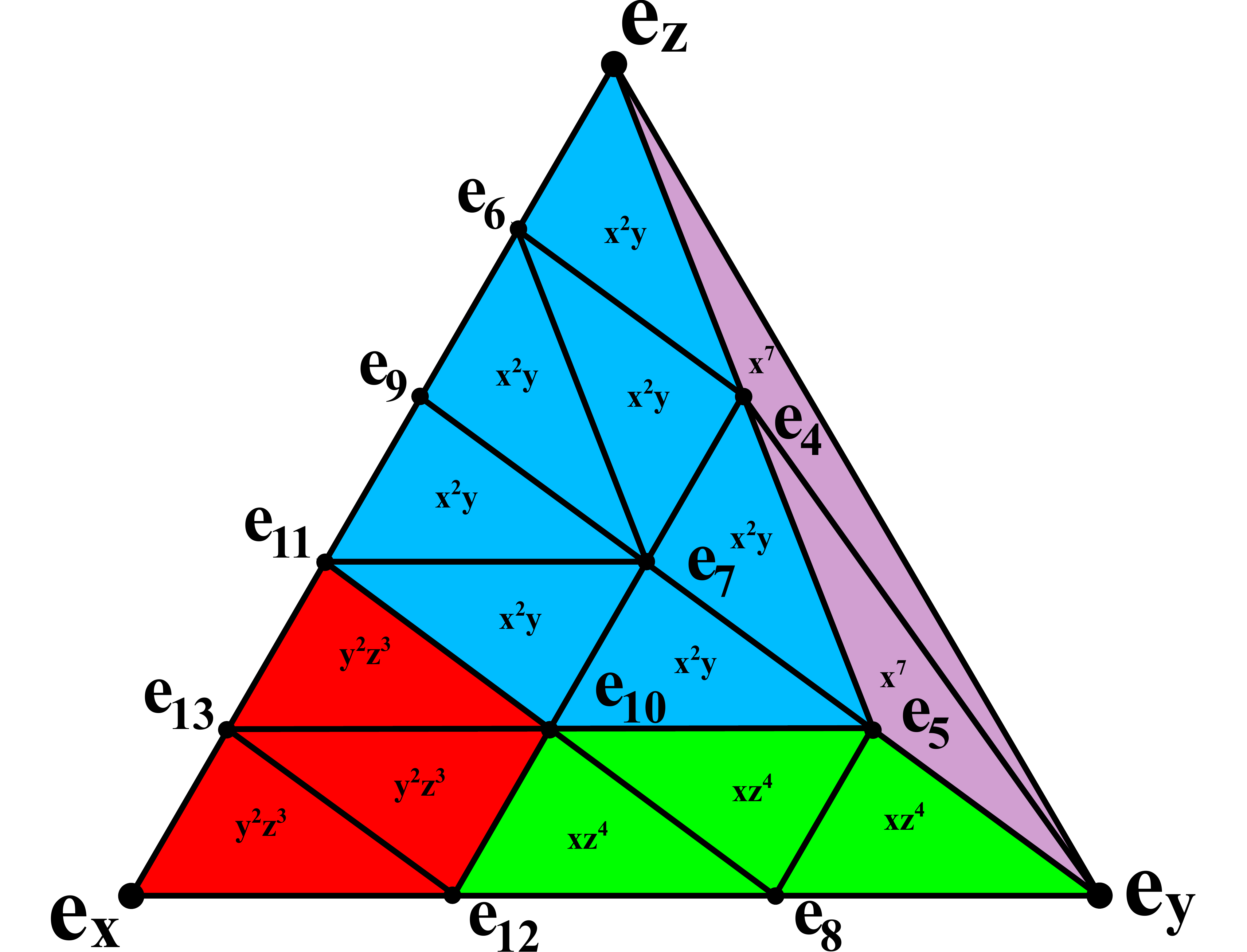}}
\subfigure[$\chi_9$] { \label{figure-36i}
\includegraphics[scale=0.11]{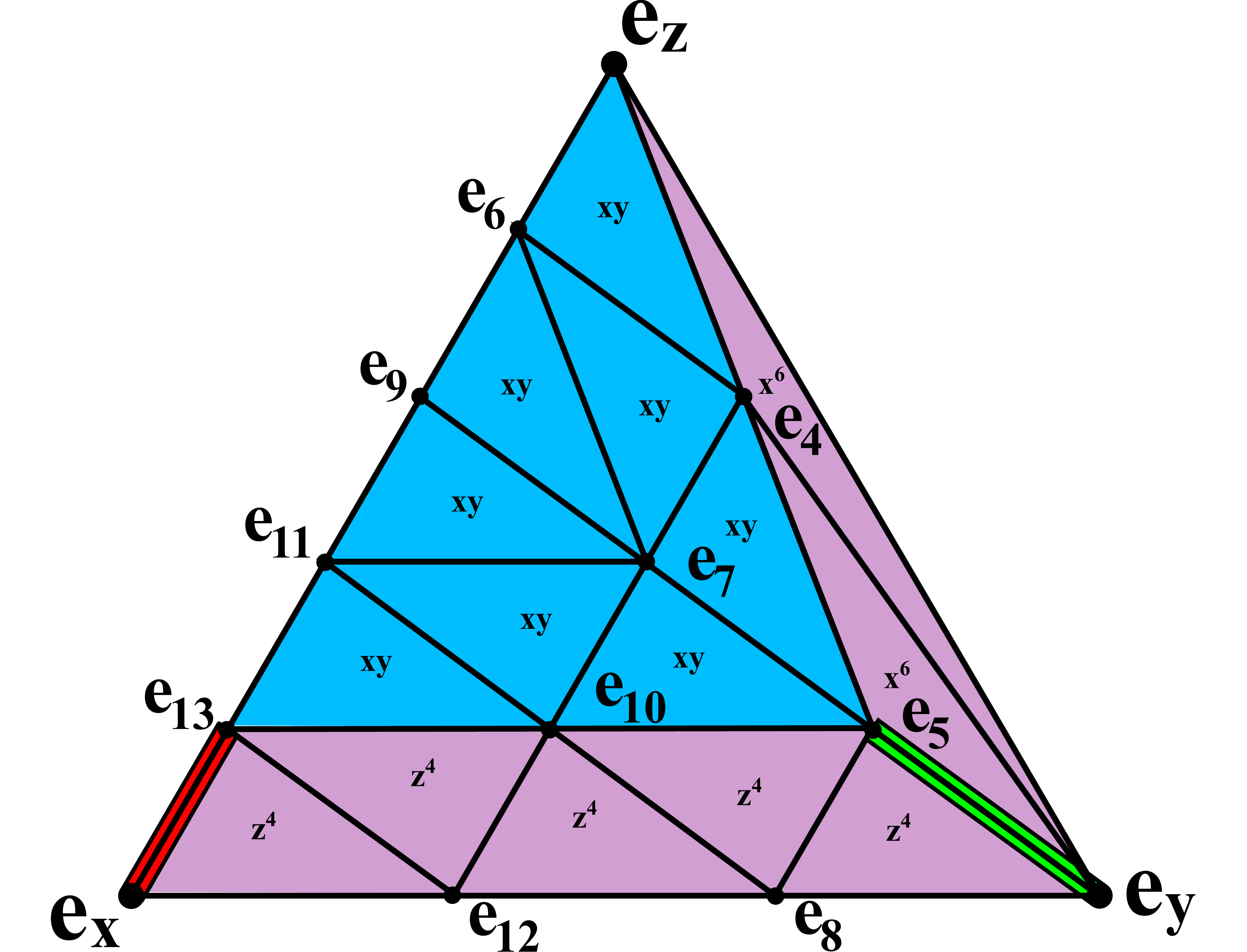}}
\subfigure[$\chi_{10}$] { \label{figure-36k}
\includegraphics[scale=0.11]{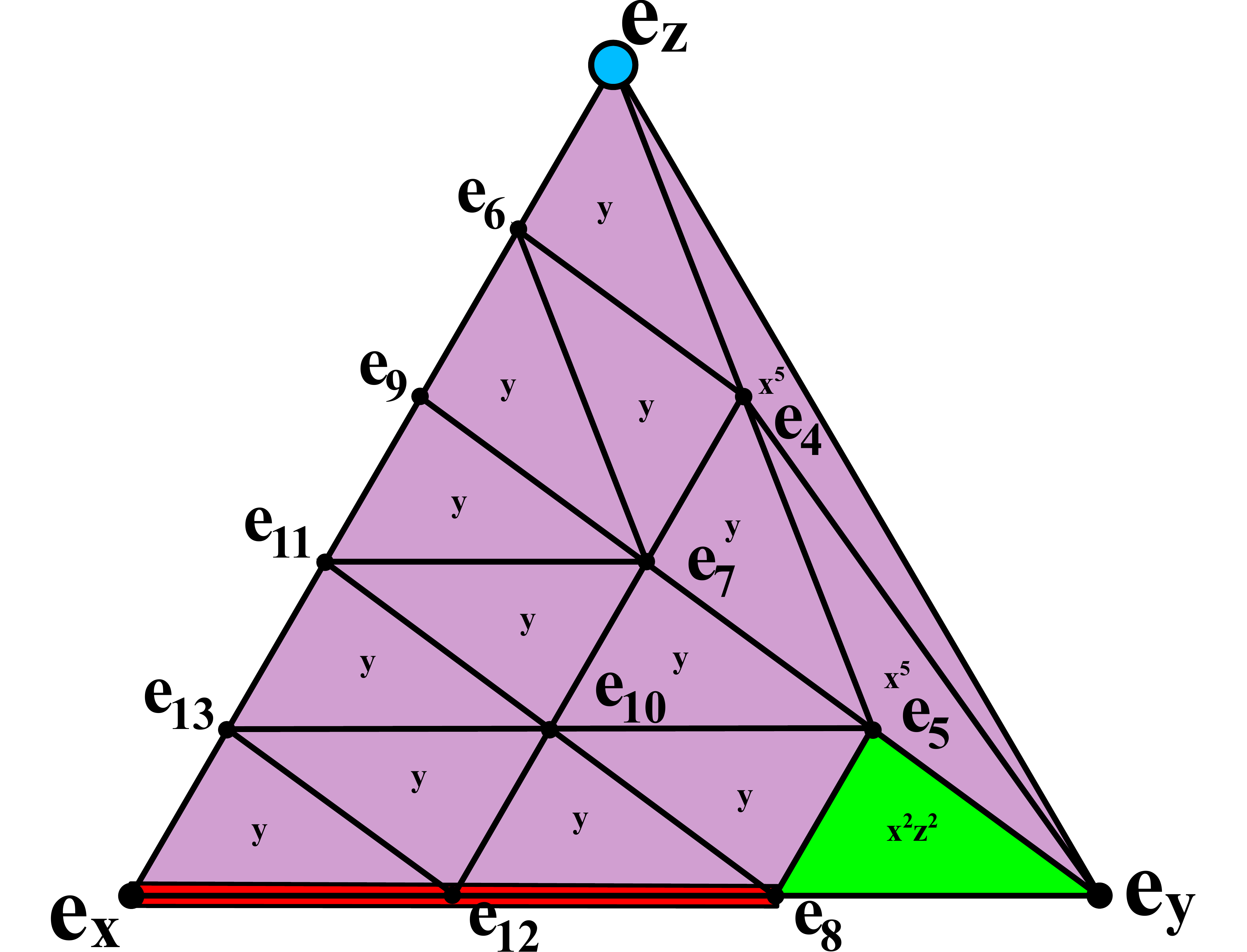}}
\subfigure[$\chi_{11}$] { \label{figure-36l}
\includegraphics[scale=0.11]{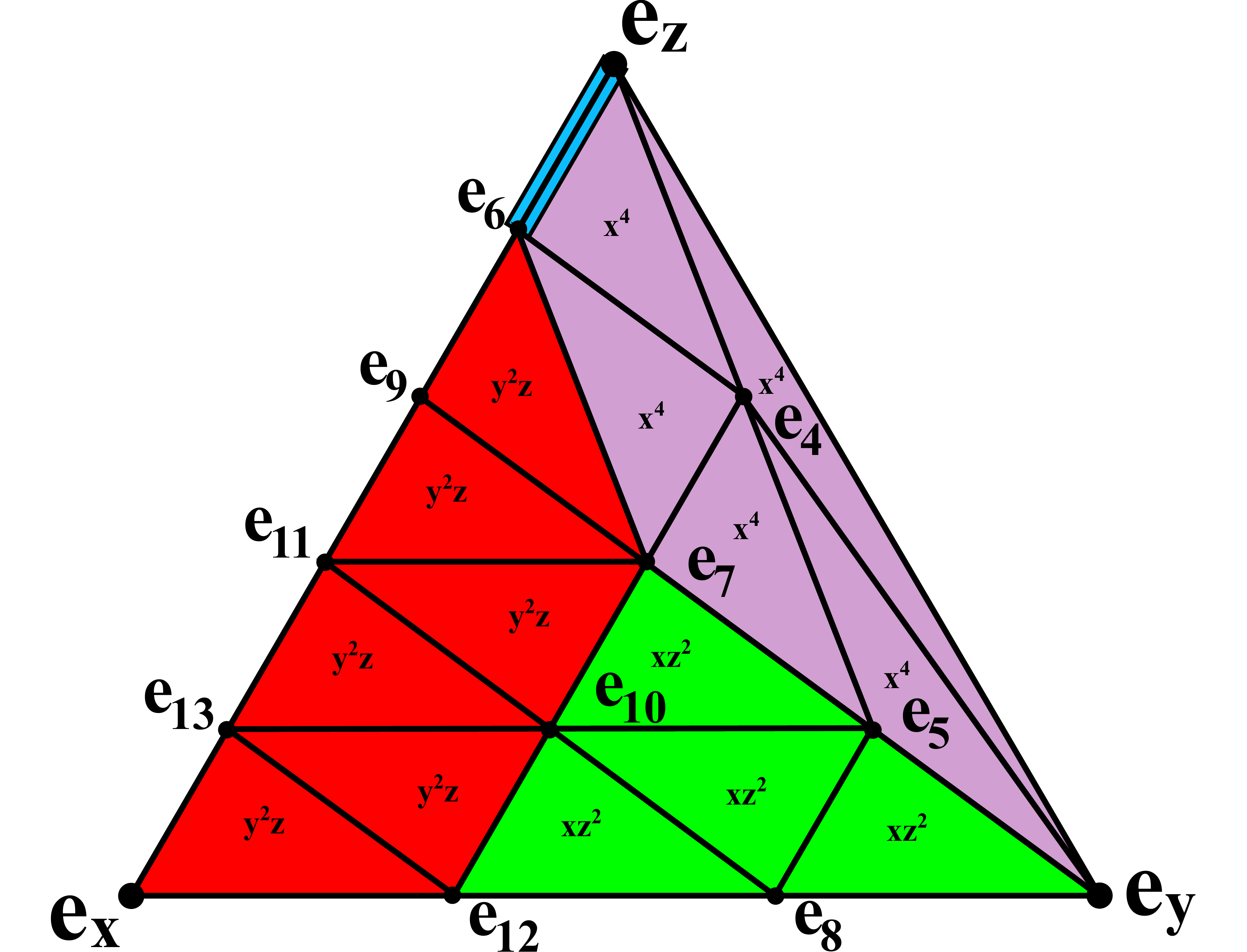}}
\subfigure[$\chi_{12}$] { \label{figure-36m}
\includegraphics[scale=0.11]{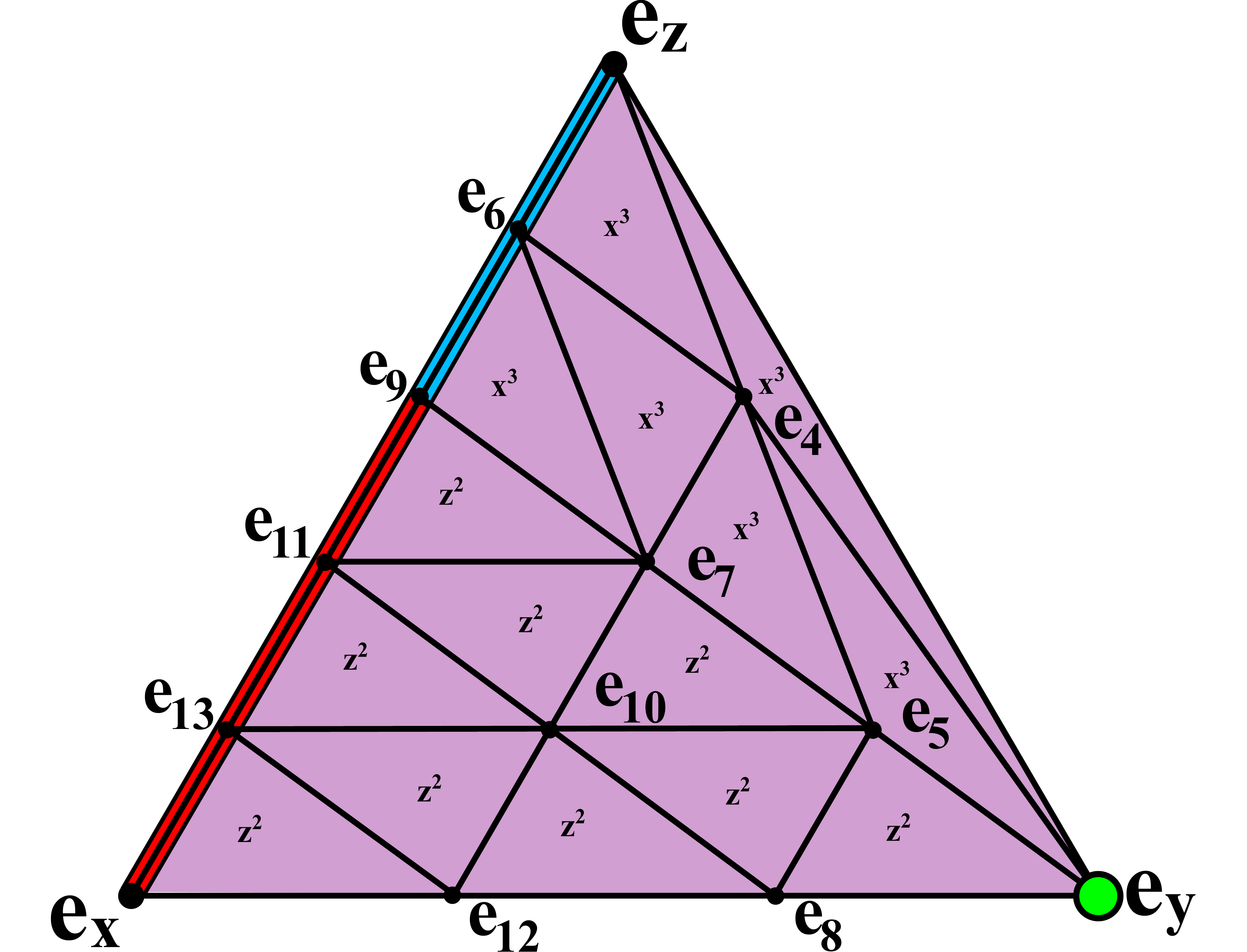}}
\caption{\label{figure-37} CT-subdivisions and monomial generators of $\mathcal{L}_\chi$
for $\chi_7$ - $\chi_{12}$.}
\end{figure}
\newpage 
\begin{figure}[!htb] \centering 
\subfigure[$\chi_{13}$] { \label{figure-36n}
\includegraphics[scale=0.11]{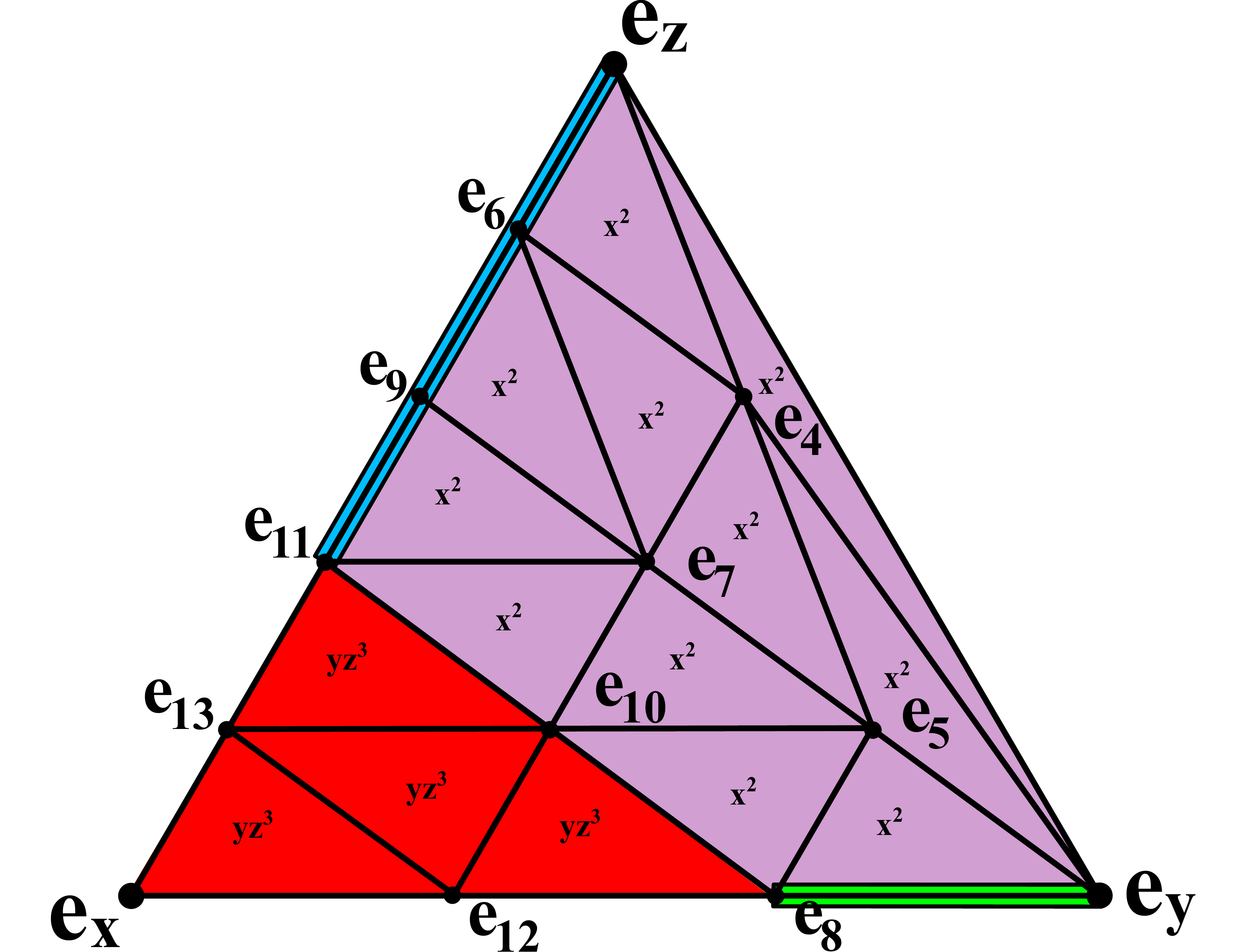}}
\subfigure[$\chi_{14}$] { \label{figure-36o}
\includegraphics[scale=0.11]{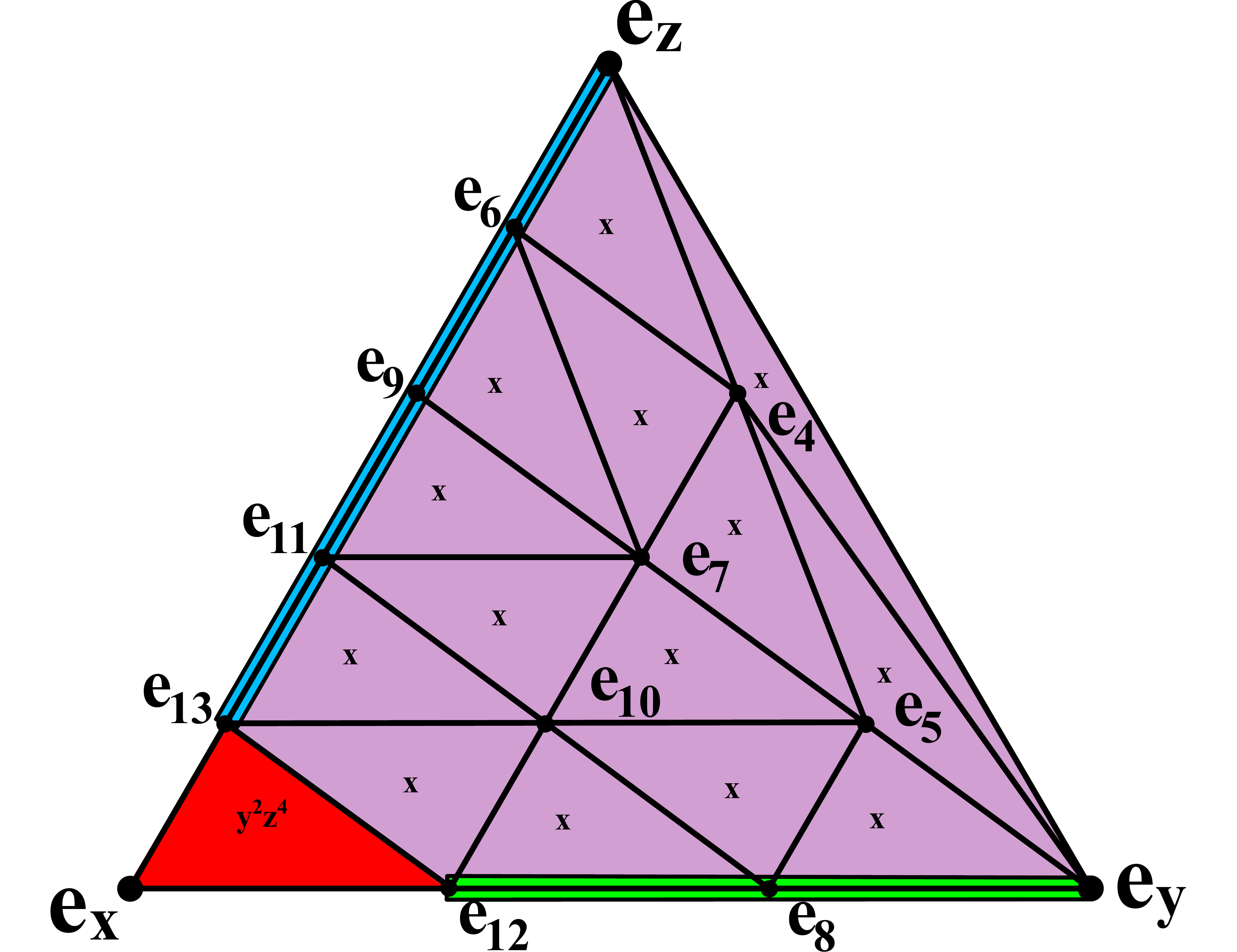}}
\caption{\label{figure-38} CT-subdivisions and monomial generators of $\mathcal{L}_\chi$
for $\chi_{13}$ - $\chi_{14}$.}
\end{figure}

\subsection{Skew-commutative cubes corresponding to $\Psi(\mathcal{O}_0 \otimes \chi)$}

Now for each $\chi \in G^\vee$ we draw the skew-commutative cube of 
line bundles corresponding to $\hex(\chi^{-1})_{\widetilde{\mathcal{M}}}$ 
as per Fig.\ \ref{figure-13}. Next, for each toric divisor $E$ on $Y$ we 
use Prop.\ \ref{prps-vertex-type-to-CT-subdivision-role-correspondence}
to determine the arrows of the cube whose corresponding
maps in $\hex(\chi^{-1})_{\widetilde{\mathcal{M}}}$ vanish along $E$. 
We then mark each arrow of the cube by its vanishing divisor in 
a following shorthand: 
$E_{456101213} = E_4 + E_5 + E_6 + E_{10} + E_{12} + E_{13}$ etc. 
 
We thus obtain:

\begin{large}
\begin{align*} 
\Psi(\mathcal{O}_0 \otimes \chi_{0}):\;
\xymatrix{
& &
\mathcal{L}^{-1}_{\chi_{1}} 
\ar"2,6"^<<{E_z} 
\ar"3,6"_<<<<<{E_y}
& & & 
\mathcal{L}^{-1}_{\chi_{14}} 
\ar"2,8"^{E_x} 
& & 
\\
\mathcal{L}^{-1}_{\chi_{0}} 
\ar"1,3"^<<<<<<<<<<<{E_{x45678910111213}\quad\;} 
\ar"2,3"^{E_{y45781012}}
\ar"3,3"_{E_{z45679101113}\;} 
& & 
\mathcal{L}^{-1}_{\chi_{5}} 
\ar"1,6"^>>>>>>>{E_{z691113}} 
\ar"3,6"_>>>>>>>{E_{x691113}}
& & & 
\mathcal{L}^{-1}_{\chi_{10}} 
\ar"2,8"^{E_y} 
& & 
\mathcal{L}^{-1}_{\chi_{0}} 
\\
& & 
\mathcal{L}^{-1}_{\chi_{9}} 
\ar"1,6"^<<<{E_{y812}} 
\ar"2,6"_<<{E_{x812}} 
& & & 
\mathcal{L}^{-1}_{\chi_{6}} 
\ar"2,8"_{E_z} 
& &
}
\end{align*}
\end{large}
\begin{large}
\begin{align*} 
\Psi(\mathcal{O}_0 \otimes \chi_{1}):\;
\xymatrix{
& &
\mathcal{L}^{-1}_{\chi_{2}} 
\ar"2,6"^<<{E_{z6}} 
\ar"3,6"_<<<<<{E_{y58}}
& & & 
\mathcal{L}^{-1}_{\chi_{0}} 
\ar"2,8"^{\quad E_{x45678910111213}} 
& & 
\\
\mathcal{L}^{-1}_{\chi_{1}} 
\ar"1,3"^{E_{x7910111213}\quad} 
\ar"2,3"^{E_{y}}
\ar"3,3"_{E_{z}} 
& & 
\mathcal{L}^{-1}_{\chi_{6}} 
\ar"1,6"^>>>>>>>{E_{z}} 
\ar"3,6"_>>>>>>>{E_{x578910111213}\;}
& & & 
\mathcal{L}^{-1}_{\chi_{11}} 
\ar"2,8"^{E_{y458}} 
& & 
\mathcal{L}^{-1}_{\chi_{1}} 
\\
& & 
\mathcal{L}^{-1}_{\chi_{10}} 
\ar"1,6"^<<<{E_{y}} 
\ar"2,6"_<<{\quad\quad E_{x67910111213}} 
& & & 
\mathcal{L}^{-1}_{\chi_{7}} 
\ar"2,8"_{E_{z46}} 
& &
}
\end{align*}
\end{large}
\begin{large}
\begin{align*} 
\Psi(\mathcal{O}_0 \otimes \chi_{2}):\;
\xymatrix{
& &
\mathcal{L}^{-1}_{\chi_{3}} 
\ar"2,6"^<<{E_{z69}} 
\ar"3,6"_<<<<<{E_{y581012}}
& & & 
\mathcal{L}^{-1}_{\chi_{1}} 
\ar"2,8"^{\quad E_{x7910111213}} 
& & 
\\
\mathcal{L}^{-1}_{\chi_{2}} 
\ar"1,3"^{E_{x1113}\quad} 
\ar"2,3"^{E_{y58}}
\ar"3,3"_{E_{z6}} 
& & 
\mathcal{L}^{-1}_{\chi_{7}} 
\ar"1,6"^>>>>>>>{E_{z46}} 
\ar"3,6"_>>>>>>>{E_{x10111213}}
& & & 
\mathcal{L}^{-1}_{\chi_{12}} 
\ar"2,8"^{E_{y45781012}} 
& & 
\mathcal{L}^{-1}_{\chi_{2}} 
\\
& & 
\mathcal{L}^{-1}_{\chi_{11}} 
\ar"1,6"^<<<{E_{y458}} 
\ar"2,6"_<<{\quad\quad E_{x91113}} 
& & & 
\mathcal{L}^{-1}_{\chi_{8}} 
\ar"2,8"_{E_{z4679}} 
& &
}
\end{align*}
\end{large}
\begin{large}
\begin{align*} 
\Psi(\mathcal{O}_0 \otimes \chi_{3}):\;
\xymatrix{
& &
\mathcal{L}^{-1}_{\chi_{4}} 
\ar"2,6"^<<{E_{z6911}} 
\ar"3,6"_<<<<<{E_{y5}}
& & & 
\mathcal{L}^{-1}_{\chi_{2}} 
\ar"2,8"^{\quad E_{x1113}} 
& & 
\\
\mathcal{L}^{-1}_{\chi_{3}} 
\ar"1,3"^{E_{x8101213}\;} 
\ar"2,3"^{E_{y581012}}
\ar"3,3"_{E_{z69}} 
& & 
\mathcal{L}^{-1}_{\chi_{8}} 
\ar"1,6"^>>>>>>>{E_{z4679}} 
\ar"3,6"_>>>>>>>{E_{x13}}
& & & 
\mathcal{L}^{-1}_{\chi_{13}} 
\ar"2,8"^{E_{y457}} 
& & 
\mathcal{L}^{-1}_{\chi_{3}} 
\\
& & 
\mathcal{L}^{-1}_{\chi_{12}} 
\ar"1,6"^<<<{E_{y45781012}} 
\ar"2,6"_<<{\quad\;E_{x810111213}} 
& & & 
\mathcal{L}^{-1}_{\chi_{9}} 
\ar"2,8"_{E_{z467911}} 
& &
}
\end{align*}
\end{large}
\begin{large}
\begin{align*} 
\Psi(\mathcal{O}_0 \otimes \chi_{4}):\;
\xymatrix{
& &
\mathcal{L}^{-1}_{\chi_{5}} 
\ar"2,6"^<<{\quad E_{z691113}} 
\ar"3,6"_<<<<<{E_{y58}}
& & & 
\mathcal{L}^{-1}_{\chi_{3}} 
\ar"2,8"^{\quad E_{x8101213}} 
& & 
\\
\mathcal{L}^{-1}_{\chi_{4}} 
\ar"1,3"^{E_{x12}\;} 
\ar"2,3"^{E_{y5}}
\ar"3,3"_{E_{z6911}} 
& & 
\mathcal{L}^{-1}_{\chi_{9}} 
\ar"1,6"^>>>>>>>{E_{z467911}} 
\ar"3,6"_>>>>>>>{E_{x1213}}
& & & 
\mathcal{L}^{-1}_{\chi_{14}} 
\ar"2,8"^{E_{y457810}} 
& & 
\mathcal{L}^{-1}_{\chi_{4}} 
\\
& & 
\mathcal{L}^{-1}_{\chi_{13}} 
\ar"1,6"^<<<{E_{y457}} 
\ar"2,6"_<<{\quad\;E_{x1213}} 
& & & 
\mathcal{L}^{-1}_{\chi_{10}} 
\ar"2,8"_{\quad E_{z4679101113}} 
& &
}
\end{align*}
\end{large}
\begin{large}
\begin{align*} 
\Psi(\mathcal{O}_0 \otimes \chi_{5}):\;
\xymatrix{
& &
\mathcal{L}^{-1}_{\chi_{6}} 
\ar"2,6"^<<{E_{z}} 
\ar"3,6"_<<<<<{E_{y5781012}}
& & & 
\mathcal{L}^{-1}_{\chi_{4}} 
\ar"2,8"^{\quad E_{x12}} 
& & 
\\
\mathcal{L}^{-1}_{\chi_{5}} 
\ar"1,3"^{E_{x691113}\;} 
\ar"2,3"^{E_{y58}}
\ar"3,3"_{E_{z691113}} 
& & 
\mathcal{L}^{-1}_{\chi_{10}} 
\ar"1,6"^>>>>>>>{E_{z4679101113}\quad} 
\ar"3,6"_>>>>>>>{E_{x67910111213}\;}
& & & 
\mathcal{L}^{-1}_{\chi_{0}} 
\ar"2,8"^{E_{y45781012}} 
& & 
\mathcal{L}^{-1}_{\chi_{5}} 
\\
& & 
\mathcal{L}^{-1}_{\chi_{14}} 
\ar"1,6"^<<<{E_{y457810}} 
\ar"2,6"_<<{\quad E_{x}} 
& & & 
\mathcal{L}^{-1}_{\chi_{11}} 
\ar"2,8"_{E_{z4}} 
& &
}
\end{align*}
\end{large}
\begin{large}
\begin{align*} 
\Psi(\mathcal{O}_0 \otimes \chi_{6}):\;
\xymatrix{
& &
\mathcal{L}^{-1}_{\chi_{7}} 
\ar"2,6"^<<{E_{z46}} 
\ar"3,6"_<<<<<{E_{y}}
& & & 
\mathcal{L}^{-1}_{\chi_{5}} 
\ar"2,8"^{\quad E_{x691113}} 
& & 
\\
\mathcal{L}^{-1}_{\chi_{6}} 
\ar"1,3"^{E_{x578910111213}\quad} 
\ar"2,3"^{E_{y5781012}}
\ar"3,3"_{E_{z}} 
& & 
\mathcal{L}^{-1}_{\chi_{11}} 
\ar"1,6"^>>>>>>>{E_{z4}} 
\ar"3,6"_>>>>>>>{E_{x91113}}
& & & 
\mathcal{L}^{-1}_{\chi_{1}} 
\ar"2,8"^{E_{y}} 
& & 
\mathcal{L}^{-1}_{\chi_{6}} 
\\
& & 
\mathcal{L}^{-1}_{\chi_{0}} 
\ar"1,6"^<<<{E_{y45781012}} 
\ar"2,6"_<<{\quad\quad E_{x45678910111213}} 
& & & 
\mathcal{L}^{-1}_{\chi_{12}} 
\ar"2,8"_{E_{z46}} 
& &
}
\end{align*}
\end{large}
\begin{large}
\begin{align*} 
\Psi(\mathcal{O}_0 \otimes \chi_{7}):\;
\xymatrix{
& &
\mathcal{L}^{-1}_{\chi_{8}} 
\ar"2,6"^<<{\;E_{z4679}} 
\ar"3,6"_<<<<<{E_{y8}}
& & & 
\mathcal{L}^{-1}_{\chi_{6}} 
\ar"2,8"^{\quad E_{x578910111213}} 
& & 
\\
\mathcal{L}^{-1}_{\chi_{7}} 
\ar"1,3"^{E_{x10111213}\quad} 
\ar"2,3"^{E_{y}}
\ar"3,3"_{E_{z46}} 
& & 
\mathcal{L}^{-1}_{\chi_{12}} 
\ar"1,6"^>>>>>>>{E_{z46}} 
\ar"3,6"_>>>>>>>{E_{x810111213}}
& & & 
\mathcal{L}^{-1}_{\chi_{2}} 
\ar"2,8"^{E_{y58}} 
& & 
\mathcal{L}^{-1}_{\chi_{7}} 
\\
& & 
\mathcal{L}^{-1}_{\chi_{1}} 
\ar"1,6"^<<<{E_{y}} 
\ar"2,6"_<<{\quad\quad E_{x7910111213}} 
& & & 
\mathcal{L}^{-1}_{\chi_{13}} 
\ar"2,8"_{E_{z45679}} 
& &
}
\end{align*}
\end{large}
\begin{large}
\begin{align*} 
\Psi(\mathcal{O}_0 \otimes \chi_{8}):\;
\xymatrix{
& &
\mathcal{L}^{-1}_{\chi_{9}} 
\ar"2,6"^<<{\quad E_{z467911}} 
\ar"3,6"_<<<<<{E_{y812}}
& & & 
\mathcal{L}^{-1}_{\chi_{7}} 
\ar"2,8"^{\quad E_{x10111213}} 
& & 
\\
\mathcal{L}^{-1}_{\chi_{8}} 
\ar"1,3"^{E_{x13}\quad} 
\ar"2,3"^{E_{y8}}
\ar"3,3"_{E_{z4679}} 
& & 
\mathcal{L}^{-1}_{\chi_{13}} 
\ar"1,6"^>>>>>>>{E_{z45679}} 
\ar"3,6"_>>>>>>>{E_{x1213}}
& & & 
\mathcal{L}^{-1}_{\chi_{3}} 
\ar"2,8"^{E_{y581012}} 
& & 
\mathcal{L}^{-1}_{\chi_{8}} 
\\
& & 
\mathcal{L}^{-1}_{\chi_{2}} 
\ar"1,6"^<<<{E_{y58}} 
\ar"2,6"_<<{\quad\quad E_{x1113}} 
& & & 
\mathcal{L}^{-1}_{\chi_{14}} 
\ar"2,8"_{\quad E_{z456791011}} 
& &
}
\end{align*}
\end{large}
\begin{large}
\begin{align*} 
\Psi(\mathcal{O}_0 \otimes \chi_{9}):\;
\xymatrix{
& &
\mathcal{L}^{-1}_{\chi_{10}} 
\ar"2,6"^<<{\quad E_{z4679101113}} 
\ar"3,6"_<<<<<{E_{y}}
& & & 
\mathcal{L}^{-1}_{\chi_{8}} 
\ar"2,8"^{\quad E_{x13}} 
& & 
\\
\mathcal{L}^{-1}_{\chi_{9}} 
\ar"1,3"^{E_{x812}\quad} 
\ar"2,3"^{E_{y812}}
\ar"3,3"_{E_{z467911}} 
& & 
\mathcal{L}^{-1}_{\chi_{14}} 
\ar"1,6"^>>>>>>>>{E_{z456791011}\quad} 
\ar"3,6"_>>>>>>>{E_{x}}
& & & 
\mathcal{L}^{-1}_{\chi_{4}} 
\ar"2,8"^{E_{y5}} 
& & 
\mathcal{L}^{-1}_{\chi_{9}} 
\\
& & 
\mathcal{L}^{-1}_{\chi_{3}} 
\ar"1,6"^<<<{E_{y581012}} 
\ar"2,6"_<<{\quad\quad E_{x8101213}} 
& & & 
\mathcal{L}^{-1}_{\chi_{0}} 
\ar"2,8"_{\quad E_{z45679101113}} 
& &
}
\end{align*}
\end{large}
\begin{large}
\begin{align*} 
\Psi(\mathcal{O}_0 \otimes \chi_{10}):\;
\xymatrix{
& &
\mathcal{L}^{-1}_{\chi_{11}} 
\ar"2,6"^<<{\quad E_{z4}} 
\ar"3,6"_<<<<<{E_{y458}}
& & & 
\mathcal{L}^{-1}_{\chi_{9}} 
\ar"2,8"^{\quad E_{x812}} 
& & 
\\
\mathcal{L}^{-1}_{\chi_{10}} 
\ar"1,3"^{E_{x67910111213}\quad} 
\ar"2,3"^{E_{y}}
\ar"3,3"_{E_{z4679101113}} 
& & 
\mathcal{L}^{-1}_{\chi_{0}} 
\ar"1,6"^>>>>>>>{E_{z45679101113}\quad} 
\ar"3,6"_>>>>>>>>>{E_{x45678910111213}\;\;}
& & & 
\mathcal{L}^{-1}_{\chi_{5}} 
\ar"2,8"^{E_{y58}} 
& & 
\mathcal{L}^{-1}_{\chi_{10}} 
\\
& & 
\mathcal{L}^{-1}_{\chi_{4}} 
\ar"1,6"^<<<{E_{y5}} 
\ar"2,6"_<<{\; E_{x12}} 
& & & 
\mathcal{L}^{-1}_{\chi_{1}} 
\ar"2,8"_{\quad E_{z}} 
& &
}
\end{align*}
\end{large}
\begin{large}
\begin{align*} 
\Psi(\mathcal{O}_0 \otimes \chi_{11}):\;
\xymatrix{
& &
\mathcal{L}^{-1}_{\chi_{12}} 
\ar"2,6"^<<{\quad E_{z46}} 
\ar"3,6"_<<<<<{E_{y45781012}}
& & & 
\mathcal{L}^{-1}_{\chi_{10}} 
\ar"2,8"^{\quad E_{x67910111213}} 
& & 
\\
\mathcal{L}^{-1}_{\chi_{11}} 
\ar"1,3"^{E_{x91113}\quad} 
\ar"2,3"^{E_{y458}}
\ar"3,3"_{E_{z4}} 
& & 
\mathcal{L}^{-1}_{\chi_{1}} 
\ar"1,6"^>>>>>{E_{z}} 
\ar"3,6"_>>>>>>>{E_{x7910111213}}
& & & 
\mathcal{L}^{-1}_{\chi_{6}} 
\ar"2,8"^{E_{y5781012}} 
& & 
\mathcal{L}^{-1}_{\chi_{11}} 
\\
& & 
\mathcal{L}^{-1}_{\chi_{5}} 
\ar"1,6"^<<<{E_{y58}} 
\ar"2,6"_<<{\quad E_{x691113}} 
& & & 
\mathcal{L}^{-1}_{\chi_{2}} 
\ar"2,8"_{\quad E_{z6}} 
& &
}
\end{align*}
\end{large}
\begin{large}
\begin{align*} 
\Psi(\mathcal{O}_0 \otimes \chi_{12}):\;
\xymatrix{
& &
\mathcal{L}^{-1}_{\chi_{13}} 
\ar"2,6"^<<{\quad E_{z45679}} 
\ar"3,6"_<<<<<{E_{y457}}
& & & 
\mathcal{L}^{-1}_{\chi_{11}} 
\ar"2,8"^{\quad E_{x91113}} 
& & 
\\
\mathcal{L}^{-1}_{\chi_{12}} 
\ar"1,3"^{E_{x810111213}\quad} 
\ar"2,3"^{E_{y45781012}}
\ar"3,3"_{E_{z46}} 
& & 
\mathcal{L}^{-1}_{\chi_{2}} 
\ar"1,6"^>>>>>{E_{z6}} 
\ar"3,6"_>>>>>>>{E_{x1113}}
& & & 
\mathcal{L}^{-1}_{\chi_{7}} 
\ar"2,8"^{E_{y}} 
& & 
\mathcal{L}^{-1}_{\chi_{12}} 
\\
& & 
\mathcal{L}^{-1}_{\chi_{6}} 
\ar"1,6"^<<<{E_{y5781012}} 
\ar"2,6"_<<{\quad \quad E_{x578910111213}} 
& & & 
\mathcal{L}^{-1}_{\chi_{3}} 
\ar"2,8"_{\quad E_{z69}} 
& &
}
\end{align*}
\end{large}
\begin{large}
\begin{align*} 
\Psi(\mathcal{O}_0 \otimes \chi_{13}):\;
\xymatrix{
& &
\mathcal{L}^{-1}_{\chi_{14}} 
\ar"2,6"^<<<<{\quad E_{z456791011}} 
\ar"3,6"_<<<<<{E_{y457810}}
& & & 
\mathcal{L}^{-1}_{\chi_{12}} 
\ar"2,8"^{\quad E_{x810111213}} 
& & 
\\
\mathcal{L}^{-1}_{\chi_{13}} 
\ar"1,3"^{E_{x1213}\;} 
\ar"2,3"^{E_{y457}}
\ar"3,3"_{E_{z45679}} 
& & 
\mathcal{L}^{-1}_{\chi_{3}} 
\ar"1,6"^>>>>>{E_{z69}} 
\ar"3,6"_>>>>>>>{E_{x8101213}}
& & & 
\mathcal{L}^{-1}_{\chi_{8}} 
\ar"2,8"^{E_{y8}} 
& & 
\mathcal{L}^{-1}_{\chi_{13}} 
\\
& & 
\mathcal{L}^{-1}_{\chi_{7}} 
\ar"1,6"^<<<{E_{y}} 
\ar"2,6"_<<<<{\quad E_{x10111213}} 
& & & 
\mathcal{L}^{-1}_{\chi_{4}} 
\ar"2,8"_{\quad E_{z6911}} 
& &
}
\end{align*}
\end{large}
\begin{large}
\begin{align*} 
\Psi(\mathcal{O}_0 \otimes \chi_{14}):\;
\xymatrix{
& &
\mathcal{L}^{-1}_{\chi_{0}} 
\ar"2,6"^<<<<{\quad E_{z45679101113}} 
\ar"3,6"_<<<<<{E_{y45781012}}
& & & 
\mathcal{L}^{-1}_{\chi_{13}} 
\ar"2,8"^{\quad E_{x1213}} 
& & 
\\
\mathcal{L}^{-1}_{\chi_{14}} 
\ar"1,3"^{E_{x}} 
\ar"2,3"^{E_{y457810}}
\ar"3,3"_{E_{z456791011}} 
& & 
\mathcal{L}^{-1}_{\chi_{4}} 
\ar"1,6"^>>>>>{E_{z6911}} 
\ar"3,6"_>>>>>>>{E_{x12}}
& & & 
\mathcal{L}^{-1}_{\chi_{9}} 
\ar"2,8"^{E_{y812}} 
& & 
\mathcal{L}^{-1}_{\chi_{14}} 
\\
& & 
\mathcal{L}^{-1}_{\chi_{8}} 
\ar"1,6"^<<<{E_{y8}} 
\ar"2,6"_<<{\; E_{x13}} 
& & & 
\mathcal{L}^{-1}_{\chi_{5}} 
\ar"2,8"_{\quad E_{z691113}} 
& &
}
\end{align*}
\end{large}

\subsection{Derived Reid's recipe}

 We now analyze the skew-commutative cubes computed in the previous 
section and, using Lemma \ref{lemma-cube_cohomology} or a more
involved analysis where necessary, compute the cohomology sheaves 
of their total complexes. This completes the derived Reid's recipe 
computation and we list the results in  Fig.\ \ref{figure-33}. 

\begin{figure}[h] 
\begin{footnotesize}
\begin{align*}
\setlength{\extrarowheight}{0.2cm} 
\begin{array}{|c|c|c|c|}
\hline
\chi \setminus \mathcal{H}^\bullet\left(\Psi(\mathcal{O}_0 \otimes \chi)\right) & 
\mathcal{H}^{-2} & \mathcal{H}^{-1} & \mathcal{H}^{0}
\\
\hline
\chi_0  & 
\mathcal{O}_{E_{45710}}(E_{45710})  &  
\mathcal{O}_{C_{12\cap13}}(-1)
&  0  
\\
\chi_1 & 
0 & 0 &
\mathcal{L}^{-1}_{\chi_{1}} \otimes \mathcal{O}_{E_4}
\\
\chi_2 &
0 & 0 &
\mathcal{L}^{-1}_{\chi_{2}} \otimes \mathcal{O}_{E_7}
\\
\chi_3 &
0 &  0 &
\mathcal{L}^{-1}_{\chi_{3}} \otimes \mathcal{O}_{C_{7 \cap 11}}
\\
\chi_4 &
0 &  0 &
\mathcal{L}^{-1}_{\chi_{4}} \otimes \mathcal{O}_{E_{10}}
\\
\chi_5 &
0 &
\mathcal{L}^{-1}_{\chi_{5}}( - E_{412}) \otimes
\mathcal{O}_{E_{710}} &
0
\\
\chi_6 &
0 &
\mathcal{L}^{-1}_{\chi_{6}}( - E_{y6}) \otimes
\mathcal{O}_{E_{4}} 
& 0
\\
\chi_7 &
0 & 0 &
\mathcal{L}^{-1}_{\chi_{7}} \otimes \mathcal{O}_{E_{5}}
\\
\chi_8 &
0 & 0 &
\mathcal{L}^{-1}_{\chi_{8}} \otimes \mathcal{O}_{E_{10}}
\\
\chi_9 &
0 & 
\mathcal{L}^{-1}_{\chi_{9}}( - E_{513}) \otimes \mathcal{O}_{E_{10}}
& 0
\\
\chi_{10} &
0 & 
\mathcal{L}^{-1}_{\chi_{10}}( - E_{z8}) \otimes \mathcal{O}_{E_{45}}
&
0 
\\
\chi_{11} &
0 & 
0 &
\mathcal{L}^{-1}_{\chi_{11}} \otimes \mathcal{O}_{C_{6 \cap 7}}
\\
\chi_{12} &
0 & 
\mathcal{L}^{-1}_{\chi_{12}}(- E_{y9}) \otimes \mathcal{O}_{E_{57}}
&
0
\\
\chi_{13} &
0 & 
\mathcal{L}^{-1}_{\chi_{13}} (- E_{811}) \otimes \mathcal{O}_{E_{10}}
&
0
\\
\chi_{14} &
0 & 
0 &
\mathcal{L}^{-1}_{\chi_{14}} \otimes \mathcal{O}_{C_{12 \cap 13}}
\\
\hline 
\end{array}
\end{align*}
\end{footnotesize}
\caption{\label{figure-33} 
Derived Reid's recipe for $G = \frac{1}{15}(1,5,9)$}
\end{figure}

\subsection{Sink-source graphs}

Prop.\ \ref{prps-vertex-type-to-CT-subdivision-role-correspondence}
determines the vertex type of each $\chi \in G^\vee$ in the
sink-source graph of each toric divisor of $Y$. We therefore 
use the $CT$-subdivision data displayed 
in Fig.\ \ref{figure-36} - \ref{figure-38}  
to compute the sink-source graphs of the exceptional divisors of $Y$. 
We list the results on Fig.\ \ref{figure-35}. Though not necessary 
to compute derived Reid's recipe, it provides both 
an illustraton of the correspondence of Theorem 
\ref{theorem-sink-source-graph-to-divisor-type-correspondence}
and a visual aid for the proof of Prop.
\ref{prps-vertex-type-to-CT-subdivision-role-correspondence}. 

\begin{figure}[!htb] \centering 
\subfigure[Divisor $E_4$] { \label{figure-35a}
\includegraphics[scale=0.60]{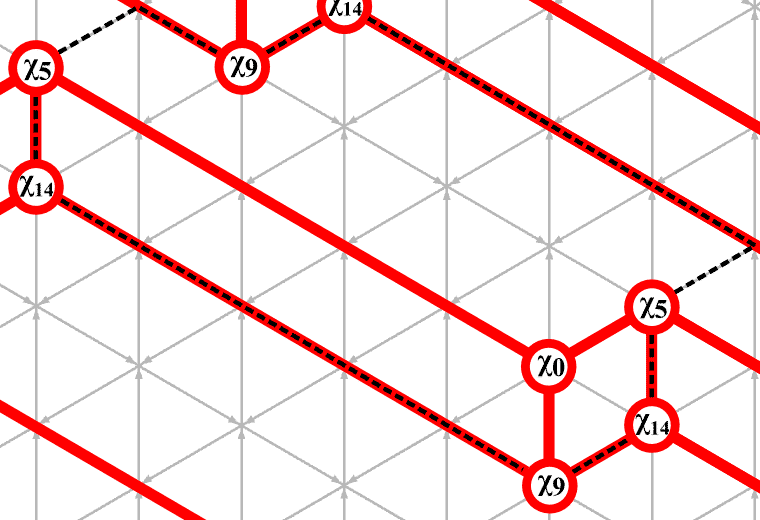}} 
\subfigure[Divisor $E_5$] { \label{figure-35b}
\includegraphics[scale=0.60]{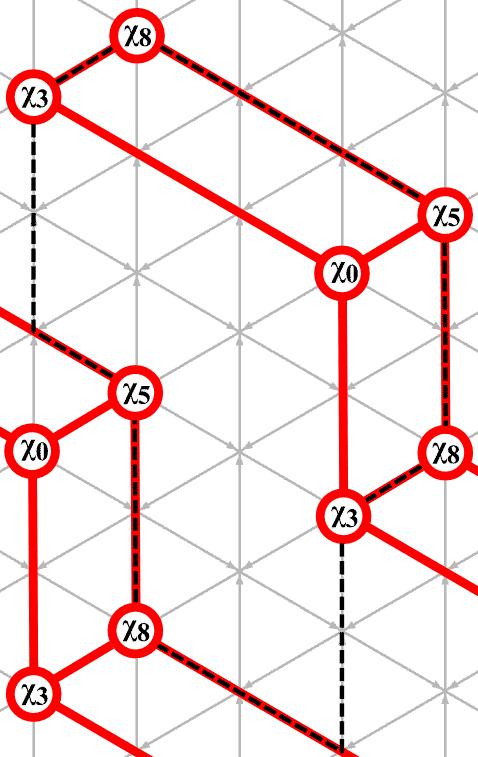}}
\subfigure[Divisor $E_6$] { \label{figure-35c}
\includegraphics[scale=0.55]{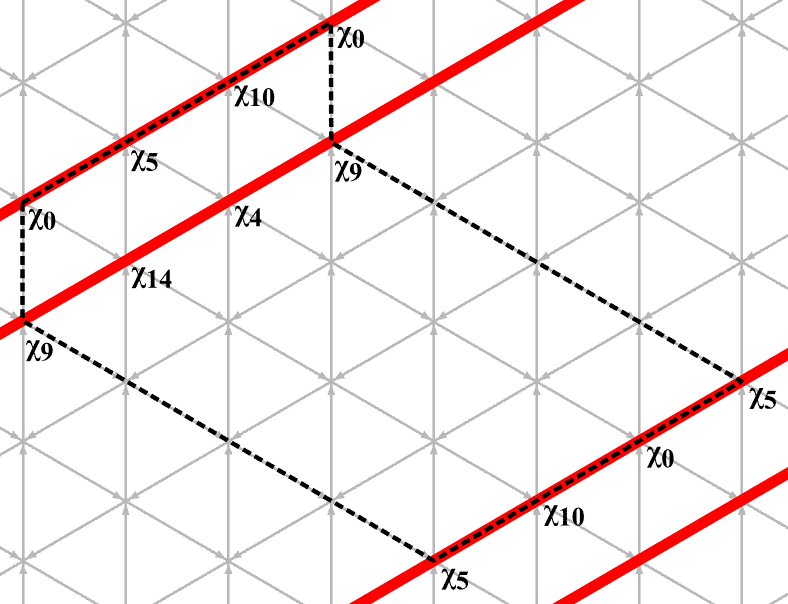}}
\subfigure[Divisor $E_7$] { \label{figure-35d}
\includegraphics[scale=0.60]{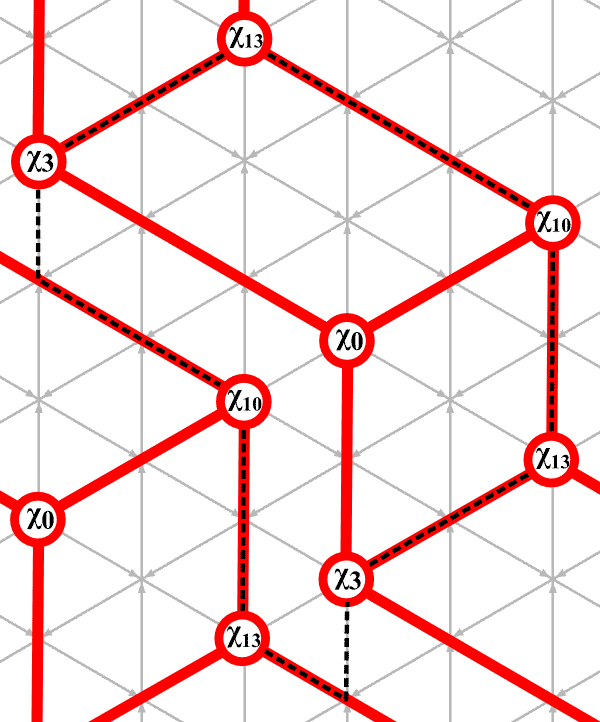}}
\subfigure[Divisor $E_8$] { \label{figure-35e}
\includegraphics[scale=0.60]{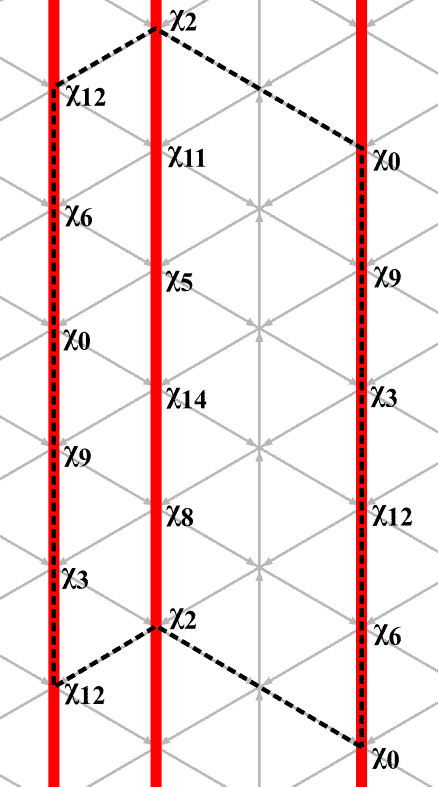}}
\subfigure[Divisor $E_9$] { \label{figure-35f}
\includegraphics[scale=0.60]{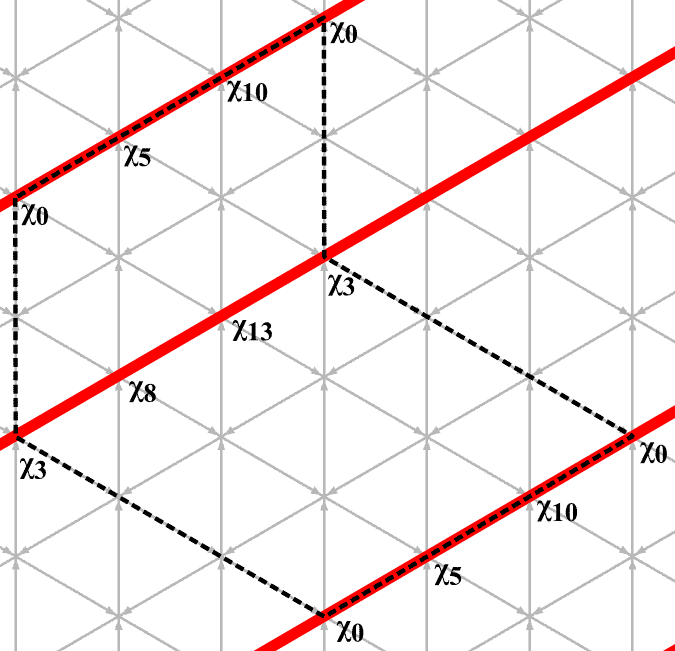}}
\subfigure[Divisor $E_{10}$] { \label{figure-35g}
\includegraphics[scale=0.60]{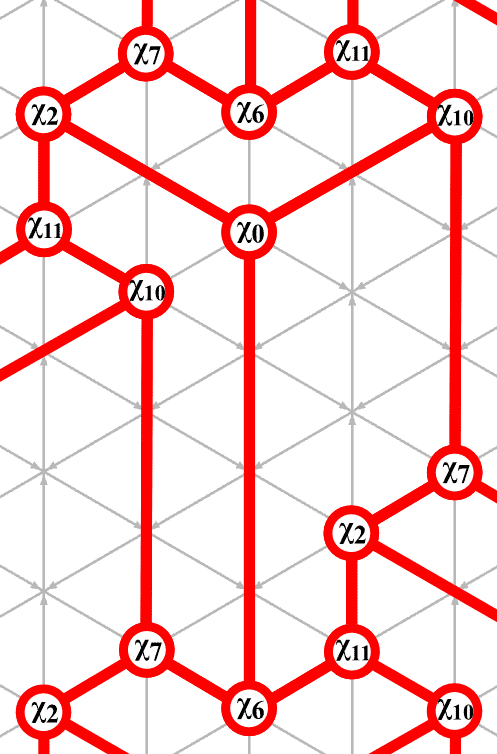}}
\subfigure[Divisor $E_{11}$] { \label{figure-35h}
\includegraphics[scale=0.60]{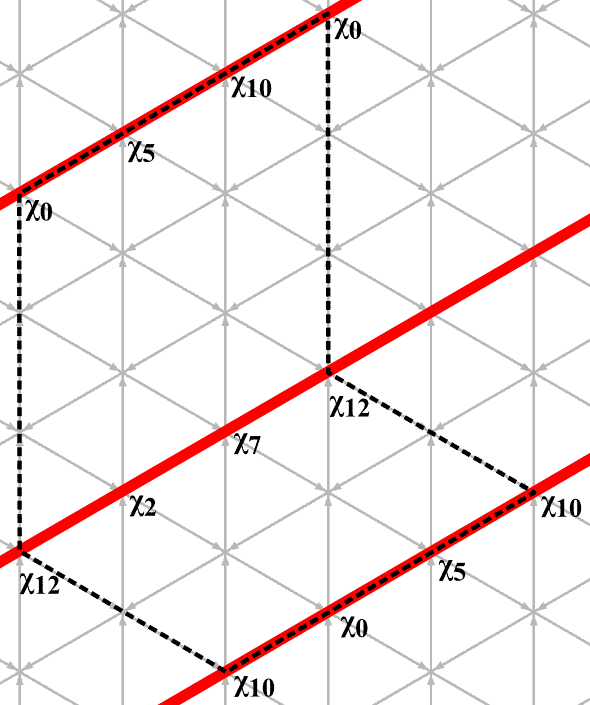}}
\subfigure[Divisor $E_{12}$] { \label{figure-35i}
\includegraphics[scale=0.60]{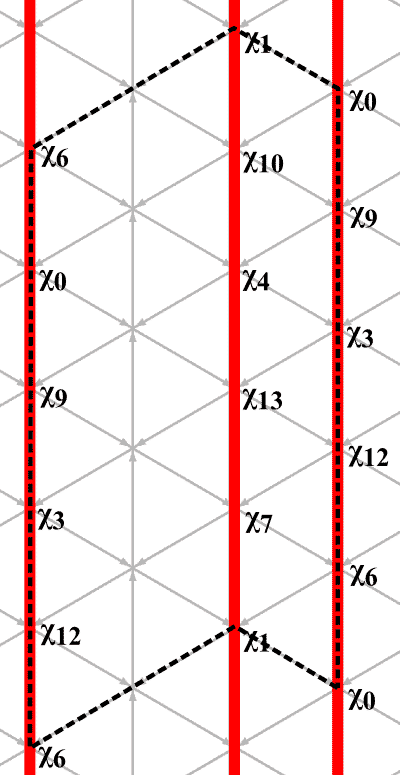}}
\subfigure[Divisor $E_{12}$] { \label{figure-35k} 
\includegraphics[scale=0.60]{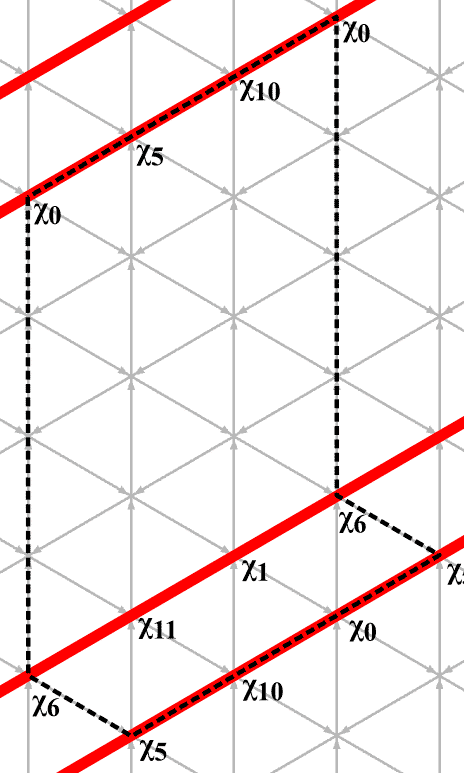}}
\caption{Sink-source graphs for $G = \frac{1}{15}(1,5,9)$}
\label{figure-35}
\end{figure}

\newpage
\bibliographystyle{amsalpha}
\bibliography{CautisCrawLogvinenko}

\end{document}